\documentclass[a4paper,10pt]{article}
\author{Stefano Scrobogna}
\title{Highly rotating fluids with vertical stratification for periodic data and vanishing vertical viscosity.}

\usepackage[a4paper]{geometry}
\usepackage{fourier}
\usepackage[T1]{fontenc}
\usepackage{amssymb}
\usepackage[english]{babel}
\usepackage{amssymb,amsthm,amsmath}
\usepackage{mathtools}
\DeclareMathAlphabet{\mathcal}{OMS}{cmsy}{m}{n}
\usepackage{hyperref}
\usepackage{cite}
\usepackage{dsfont}

\renewcommand{\d}{\textnormal{d}}
\newcommand{\cPLtwo}{{L^2\left(\mathbb{T}^3\right)}}

\newcommand{\dx}{\textnormal{d}{x}}
\newcommand{\dive}{\textnormal{div\hspace{0.7mm}}}
\newcommand{\diveh}{\textnormal{div}_h\hspace{0.7mm}}
\newcommand{\fine}{\hfill$\blacklozenge$}
\newcommand{\osc}{\textnormal{osc}}
\newcommand{\QG}{\textnormal{QG}}
\newcommand{\loc}{\textnormal{loc}}
\newcommand{\nhp}{\nabla_h^\perp}

\newcommand{\intT}{\int_{\mathbb{T}^3}}
\newcommand{\nh}{\nabla_h}
\newcommand{\Lplus}{\mathcal{L}\left(\frac{t}{\varepsilon}\right)}
\newcommand{\Lminus}{\mathcal{L}\left(-\frac{t}{\varepsilon}\right)}
\newcommand{\Lh}{\Lambda_h}
\newcommand{\Lv}{\Lambda_v}
\newcommand{\definizioneVQG}{V_\QG=\left(
\begin{array}{c}
\nhp\\
0\\
-F\partial_3 
\end{array}\right)\Delta^{-1}_F\Omega}
\newcommand{\PA}{\mathbb{P}\mathcal{A}}

\newcommand{\Hs}{{H^{0,s}}}
\newcommand{\anuno}{{\mathcal{C}\left(\mathbb{R}_+;\Hs\right)}}
\newcommand{\andue}{{L^2\left(\mathbb{R}_+;\Hs\right)}}
\newcommand{\hra}{\hookrightarrow}

\newcommand{\Hud}{{H^{1/2,0}}}
\newcommand{\Hz}{{H^{0,s_0}}}
\newcommand{\Hmud}{{H^{0,-1/2}}}
\newcommand{\qg}{{quasi-geostrophic}}

\newcommand{\NS}{Navier-Stokes }
\newcommand{\LF}{\textnormal{LF}}
\newcommand{\HF}{\textnormal{HF}}
\newcommand{\PLF}{\Psi^{\varepsilon,N}_\LF}
\newcommand{\RLF}{\widetilde{R}^{\varepsilon, N}_{\osc,\LF}}
\newcommand{\SLF}{\widetilde{S}^{\varepsilon, N}_{\osc,\LF}}
\newcommand{\R}{\mathbb{R}}
\newcommand{\T}{\mathbb{T}}
\newcommand{\C}{\mathbb{C}}
\newcommand{\sumf}{\sum_{\left| q-q' \right|\leqslant 4}}
\newcommand{\sumi}{\sum_{ q'> q-4}}

\newcommand{\set}[1]{\left\lbrace #1 \right\rbrace}

\newcommand{\cPLp}{L^p\left(\mathbb{T}^3\right)}
\newcommand{\cPtv}{\triangle^\textnormal{v}_q}
\newcommand{\cPTv}{\triangle^\textnormal{v}_{q'}}

\newcommand{\cPSvq}{S^\textnormal{v}_{q'-1}}
\newcommand{\cPHs}{{H^{0,s}}}

\theoremstyle{theorem}
\newtheorem{theorem}{Theorem}[section]
\newtheorem{prop}[theorem]{Proposition}
\newtheorem{lemma}[theorem]{Lemma}
\newtheorem{cor}[theorem]{Corollary}
\theoremstyle{definition}
\newtheorem{definition}[theorem]{Definition}
\newtheorem{rem}[theorem]{Remark}
\makeatletter
\newcommand{\settheoremtag}[1]{
  \let\oldthetheorem\thetheorem
  \renewcommand{\thetheorem}{#1}
  \g@addto@macro\endtheorem{
    \addtocounter{theorem}{-1}
    \global\let\thetheorem\oldthetheorem}
  }
\makeatother

\makeatletter
\newcommand{\setword}[2]{%
  \phantomsection
  #1\def\@currentlabel{\unexpanded{#1}}\label{#2}%
}
\makeatother

\usepackage{xcolor}

\numberwithin{equation}{section}

\AtEndDocument{\bigskip{\footnotesize
  \textsc{IMB, Universit\'e de Bordeaux} \par
  351, cours de la Lib\'eration, Bureau 255, 33405 Talence , France \par
  \textit{E-mail address:}  \texttt{\href{mailto:stefano.scrobogna@math.u-bordeaux.fr}{stefano.scrobogna@math.u-bordeaux.fr}}}}

\begin{document}
\maketitle

\begin{abstract}
We prove that the primitive equations without vertical diffusivity are globally well-posed (if the Rossby and Froude number are sufficiently small) in suitable Sobolev anisotropic spaces. Moreover if the Rossby and Froude number tend to zero at a comparable rate the global solutions of the primitive equations converge globally to the global solutions of a suitable limit system. The space domain considered has to belong to a class of tori which is general enough to include all non-resonant tori and many resonant tori as well.
\end{abstract}

\section{Introduction}
The primitive equations describe the hydro-dynamical flow in a  large scale (of order of hundreds or thousands of kilometers) on the Earth, typically the ocean or the atmosphere, under the assumption that the vertical motion is much smaller than the horizontal one and that the fluid layer depth is small compared to the radius of the earth. Concerning the difference between horizontal and vertical scale, it is also observed that for geophysical fluids the vertical component of the diffusion term (viscosity or thermal diffusivity in the case of primitive equations) is much smaller than the horizontal components. In the case of rotating fluids between two planes (see \cite{ekmanwell} for the first work in which the initial data is well prepared, in the sense that it is a two-dimensional vector field, \cite{ekmanperiodic} and \cite{ekmanill} for the generic case) the viscosity assumes the form $\left( -\nu_h \Delta_h - \varepsilon \beta \partial^2_3\right)$, with $ \Delta_h= \partial_1^2 + \partial_2^2 $, whence it makes sense to consider primitive equations with vanishing vertical diffusivity. 
\\ 
The primitive system consists in the following equations 
\begin{equation}\tag{\ref{primitive equations}}
\left\lbrace
\begin{array}{lll}
\partial_t v^{1,\varepsilon} + v^\varepsilon\cdot \nabla v^{1,\varepsilon} - \nu_h \Delta_h v^{1,\varepsilon}-\nu_v \partial_3^2 v^{1,\varepsilon}-\displaystyle\frac{1}{\varepsilon} v^{2,\varepsilon} & = & -\displaystyle\frac{1}{\varepsilon} \partial_1 \Phi_\varepsilon +f_1\\[3mm]
\partial_t v^{2,\varepsilon} + v^\varepsilon\cdot \nabla v^{2,\varepsilon} - \nu_h \Delta_h v^{2,\varepsilon}-\nu_v \partial_3^2 v^{2,\varepsilon}+\displaystyle\frac{1}{\varepsilon} v^{1,\varepsilon} & = & -\displaystyle \frac{1}{\varepsilon} \partial_2 \Phi_\varepsilon +f_2\\[3mm]
\partial_t v^{3,\varepsilon} + v^\varepsilon\cdot \nabla v^{3,\varepsilon} - \nu_h \Delta_h v^{3,\varepsilon}-\nu_v \partial_3^2 v^{2,\varepsilon}+\displaystyle\frac{1}{F  \varepsilon} T^\varepsilon & = & -\displaystyle \frac{1}{\varepsilon} \partial_3 \Phi_\varepsilon +f_3\\[3mm]
\partial_t T^\varepsilon +v^\varepsilon\cdot \nabla T^\varepsilon -\nu_h'\Delta_h T^\varepsilon -\nu'_v \partial_3^2 T^\varepsilon - \displaystyle\frac{1}{F  \varepsilon} v^{3,\varepsilon} & = & f_4\\[3mm]
\dive v^\varepsilon= 0\hfill\\
\left( v^\varepsilon,T^\varepsilon\right)\Bigr|_{t=0}=\left(v_0,T_0\right)= V_0,\hfill
\end{array}
\right.
\end{equation}
in the unknown $v^\varepsilon=\left(v^{1,\varepsilon},v^{2\varepsilon},v^{3,\varepsilon}\right), T^\varepsilon, \Phi_\varepsilon$. We will  in the following denote by $V^\varepsilon= \left( v^\varepsilon,T^\varepsilon\right) = \linebreak \left( V^{1,\varepsilon},  V^{2,\varepsilon},  V^{3,\varepsilon},  V^{4,\varepsilon}\right)$ where $v^\varepsilon$ is a vector field in $\mathbb{R}^3$(three dimensional velocity field) and $T^\varepsilon$ is a scalar function (the density fluctuation, in the case of the air it depends on the scalar (potential) temperature, in the case of the  oceans it depends on both temperature and salinity), and $\Phi_\varepsilon$ represents the hydrostatic pressure. All the functions described depend on a couple $\left( x,t\right) \in \mathbb{T}^3\times \mathbb{R}_+$ where $\mathbb{T}^3$
represents the torus
$$
\mathbb{T}^3=
\R^3 \left/ \prod_{i=1}^3 a_i \ \mathbb{Z} \right. =
 \prod_{i=1}^3 \bigl[ 0,2\pi a_i \bigr).
$$

The only assumption which is made on the vertical viscosity is $\nu_v,\nu'_v\geqslant 0$. On the other hand the horizontal viscosities  $\nu_h,\nu'_h$ are strictly positive constants. In fact the results obtained will be uniform with respect to the vertical viscosities $\left(\nu_v, \nu'_v\right)$ and hence from now on we can suppose them zero without loss of generality. We refer to \cite{paicu_critical_spaces} for the result of global well-posedness of the Navier-Stokes equation in critical spaces in the whole space with anisotropic viscosity and to \cite{paicu_NS_periodic} for the periodic case.\\

Under the  assumption $\nu_v=\nu'_v=0$  we can rewrite the System \eqref{primitive equations} in the more compact form
\begin{equation}
\tag{PE$_\varepsilon$}\label{primitive equations}
\left\lbrace
\begin{array}{l}
\partial_t V^\varepsilon +v^\varepsilon\cdot\nabla V^\varepsilon-\mathbf{D}V^\varepsilon +\frac{1}{\varepsilon} \mathcal{A}V^\varepsilon=\frac{1}{\varepsilon}\left(-\nabla \Phi_\varepsilon,0\right)+f\\
\dive v^\varepsilon=0\\
\bigl. V^\varepsilon \bigr|_{t=0}=V_0
\end{array}
\right.
\end{equation}

where
\begin{align}\label{matrici}
\mathbf{D}=&\left(
\begin{array}{cccc}
\nu_h\Delta_h&0&0&0\\
0& \nu_h\Delta_h& 0&0\\
0&0& \nu_h\Delta_h&0\\
0&0&0&\nu'_h\Delta_h
\end{array}
\right) &
\mathcal{A}=& \left(
\begin{array}{cccc}
0&-1&0&0\\
1 &0 &0 &0 \\
0&0&0&F^{-1}\\
0&0&-F^{-1}&0
\end{array}
\right)
\end{align}
with $\nu_h , \nu'_h >0 $ and $V^\varepsilon=\left( v^\varepsilon,T^\varepsilon\right)$.\\
This system is obtained combining the effects of the Coriolis force and the vertical stratification induced by the Boussinesq approximation (see for instance \cite{CDGGanisotranddispersion} for the rotating fluids in the whole space, \cite{gallagher_schochet} for the periodic case), we refer to \cite{monographrotating}, \cite{Pedlosky87} for a discussion on the model and its derivations.\\

In the study of hydrodynamical flows on this  scale two important phenomena have to be taken in consideration: the Earth rotation and the vertical stratification due to the gravity.  The Coriolis force induces a vertical rigidity on the fluid. Namely, in the asymptotic regime, 
the high rotation tends to stabilize the motion, which becomes constant in the direction parallel
to the rotation axis: the fluid moves along vertical columns (the so called "Taylor-Proudman
columns"), and the flow is purely horizontal. 
\\
Gravity forces the fluid masses to have a vertical structure: heavier layers lay under lighter ones. Internal movements of the fluid tend to destroy this structure and gravity tries to restore it, which gives a horizontal rigidity (to be opposed to the vertical rigidity induced by the rotation). In order to formally estimate the  importance of this rigidity we also compare the typical time scale of the system with the Brunt-V\"ais\"al\"a frequency and define the Froude number $Fr$. We shall not give more details in here, we refer to \cite{Pedlosky87}, \cite{monographrotating}, \cite{cushman2011introduction}.\\
The primitive equations are obtained with moment, energy and mass conservation (see \cite{embid_majda}).
The coefficient $\varepsilon >0$ denotes the  Rossby number Ro, which is defined as
$$
\varepsilon=\frac{\text{displacement due to inertial forces}}{\text{displacement due to Coriolis force}}.
$$
 As the characteristic displacement of  a particle in the ocean within a day is very small compared to the displacement caused by the rotation of the earth (generally $\varepsilon$ is of order $ 10^{-3} $ outside persistent currents such as the gulf stream), the Rossby number is supposed to be very small hence it makes sense to study the behavior of the solutions to \eqref{primitive equations} in the limit regime as $\varepsilon \to 0$.\\
With a slight abuse of notation we denote as well the coefficient $F$ Froude number, and $Fr=\varepsilon F$.
Assuming that the Brunt-V\"ais\"al\"a frequency is constant, in the whole space $\mathbb{R}^3$, when $\varepsilon\to 0$, the formal limit of the system \eqref{primitive equations} is the following quasi-geostrophic system 
\begin{equation}
 \left\lbrace
 \begin{array}{l}
 \tag{QG}
 \partial_t V_{\QG} +\Gamma \left( D\right)V_{\QG}=- \left(
 \begin{array}{c}
 \nh^\perp\\
 0\\
 -F \partial_3 
 \end{array}\right)
 \Delta^{-1}_F\left( v_{\QG}^h \cdot \nh \Omega\right),\\
  \dive v_{\QG}=0\hfill\\
 V_{\QG}\Bigr|_{t=0}=V_{{\QG},0},
  \end{array}
 \right.
 \end{equation}
 and $\Gamma\left( D \right)$ is the  pseudo-differential operator given by the formula
 $$
 \Gamma \left( D \right) u= \mathcal{F}^{-1} \left( \frac{\left| \xi \right|^2 \left( \nu \left| \xi_h \right|^2 + \nu' F^2 \xi_3^2 \right)}{\left| \xi_h \right|^2 +  F^2 \xi_3^2}  \hat{u}\left( \xi \right) \right)
 $$
We remark that in this case, since we are referring to the works \cite{charve2} and \cite{charve1}, the viscosities $ \nu, \nu' $ are to be understand as isotropic, i.e. spherically symmetric.   In particular the operator $ \Delta_F^{-1} $ is the Fourier multiplier
 $$
 -\Delta_F^{-1} f = \mathcal{F}^{-1}\left( \frac{1}{\xi_1^2+ \xi_2^2 + F^2\xi_3^2}\hat{f}\right),
 $$
and $ \Delta_F $ is the differential operator $ \Delta_F = \partial_1^2 + \partial_2^2 + F^2 \partial_3^2 $. The quantities
 $ V_{\QG} $ and $ \Omega $ are respectively called the \textit{\qg\; flow} and the \textit{potential vorticity}. We focus on the latter first, the potential vorticity is defined as
 $$
 \Omega = -\partial_2 V^1_{\QG}+ \partial_1 V^2_{\QG}- F \partial_3 V^4_{\QG},
 $$ 
 and it is related to the \qg\; flow via the 2D-like Biot -Savart law
$$
V_{\QG}=\left(
 \begin{array}{c}
 -\partial_2\\
 \partial_1 \\
 0\\
 -F \partial_3 
 \end{array}\right)
 \Delta^{-1}_F
 \Omega,
$$ 
while $ v^h_{\QG} $ and $ v_{\QG} $ are respectively the first two and three components of the vector field $ V_{\QG} $. We refer to \cite{charve_ngo_primitive} for further detail on the problem in the whole space. \\
 In our setting, i.e. with periodic data the limit system is more involved than the one mentioned above. This can easily be explained. In this case as well as in many problem with singular perturbation the idea is to decompose the unknown (in the case of the system \eqref{primitive equations} is $V^\varepsilon$) into two parts $V^\varepsilon = V_{\ker}^\varepsilon +V^\varepsilon_{\osc}$, where $V_{\ker}^\varepsilon$ belongs to the kernel of the perturbation $ \mathbb{P}\mathcal{A} $, where $ \mathbb{P}$ is the Leray projector in the first three components which leaves untouched the fourth one, and $V^\varepsilon_{\osc}$ to its orthogonal complement. In the whole space it can be proved that the oscillating part $V^\varepsilon_{\osc}$  go to zero \textit{strongly} since the waves go to infinity instantaneously. In the case of periodic data instead these  perturbations interact constructively, as in \cite{BMN1},\cite{BMNresonantdomains}, \cite{gallagher_schochet} and \cite{paicu_rotating_fluids}, only to mention some work, whence the limit system is different from the {\qg} system mentioned above (see \eqref{lim syst}). In fact \eqref{primitive equations} converge to \eqref{quasi geostrophic equation}  only weakly ( see \cite{embid_majda} and \cite{embid_majda2}).  In the present work we aim to study the behavior of strong solutions of \eqref{primitive equations} in the regime $ \varepsilon\to 0 $ in the periodic setting for a large class of tori (see Definition \ref{def:condition_P}) which may as well present resonant effects. In particular we prove that the equation \eqref{nonosc limit} is globally well posed in some suitable space of low-regularity, hence we prove the (global) convergence of solutions of \eqref{primitive equations} to solutions of \eqref{Slim}, which is globally well posed. 
 \\

We recall some result on primitive equations. We refer to J.-L. Lions, R. Temam and S. Wang (\cite{LionsTemam1} and \cite{LionsTemam2}) for the asymptotic expansion of the primitive equations with respect the Rossby number $\varepsilon$ in spherical and Cartesian geometry.\\
J.T. Beale and A. J. Bourgeois in \cite{BealeBourgeois} study the primitive equations (without viscosity, and with a simplified equation for the density) in a domain which is periodic in the horizontal direction and bounded in the vertical one. By the use of a change of variables they recover a purely periodic setting, on which they prove their result. They study as well the {\qg} system (fist on short times, then globally) as well as the convergence of primitive equations for very regular (i.e. $ H^3 $) and well prepared initial data.\\
 In \cite{embid_majda} P. Embid and A. Majda present a general formulation for the movement of geophysical fluids in the periodic setting and derive the limit equation for the kernel part of the solution. \\

Several others are the results for the primitive equation in the whole space, a non exhaustive list is \cite{charve1}, \cite{charve2}. 
In \cite{chemin_prob_antisym} J.-Y. Chemin  proved the convergence of solutions of the primitive equations toward those of the quasi geostrophic system in the case $F=1$ (hence there is no dispersive effect, see \eqref{eigenvalues}), for regular, well prepared data and under the assumption that $\left| \nu-\nu'\right|$ (the difference between the diffusion and the thermal diffusivity) is small. In this case the data is considered to be "well prepared" as long as there is a smallness hypotheses on a part of it, which is denoted  by the author as $ U_{\osc} $. 
In the case $F\neq 1$, the dissipation allows not only to prove the convergence to the primitive system to the quasi-geostrophic system but also to prove global existence for solutions when the Rossby number $\varepsilon$ is small. Thus the {\qg}  system can be interpreted as the asymptotic of the primitive system when the rotation and the stratification have important influence.\\
 For the inviscid case in the whole space when $F=1$ we mention the work of D. Iftimie \cite{iftimie_primitive} which proves that the potential vorticity $ \Omega $ propagates $ H^s \left( \R^3 \right), s>3/2 $ data under the hypothesis $ U^\varepsilon_{{\osc},0} = o_\varepsilon \left( 1 \right) $ in $ L^2 \left( \R^3 \right) $. If $ F\neq 1 $  A. Dutrifoy proved in  \cite{Dutrifoy04} the same result under much weaker assumptions, i.e. $ \Omega_0 $ is a vortex patch and $ \left\| U^\varepsilon_0 \right\|_{H^s \left( \R^3 \right)}= \mathcal{O}\left( \varepsilon^{-\gamma} \right), \gamma >0 $ and small. For the viscid case in the periodic setting I. Gallagher in \cite{gallagher_schochet} achieved the result thanks to some normal form techniques introduced by S. Schochet in \cite{schochet} which we will exploit as well in the following. We mention at last the work of F. Charve and V.-S. Ngo in \cite{charve_ngo_primitive} for the primitive equation in the whole space for $F\neq 1$ and anisotropic vanishing (horizontal) viscosity.
We recall that the primitive equations and the rotating fluid system 
\begin{equation}
\tag{RF$_\varepsilon$}\label{RF}
\partial_t v + v\cdot \nabla v -\nu\Delta v + \frac{e^3\wedge v}{\varepsilon} =-\nabla p,
\end{equation}
are intimately connected. The first result of global well posedness in the whole space for \eqref{RF} has been achieved in \cite{chemin_et_al} thanks to some Strichartz estimates due to the dispersion of the system. E. Grenier in \cite{grenierrotatingeriodic} and A. Babin et al. in \cite{BMN1} proved independently the global well posedness in the periodic case and M. Paicu in \cite{paicu_rotating_fluids} in the periodic case for anisotropic viscosity. We recall as well the results in \cite{GallagherSaint-Raymondinhomogeneousrotating} in which I. Gallagher and L. Saint-Raymond proved a weak convergence result for weak solutions for fast rotating fluids in which the rotation is inhomogeneous and given by $\frac{1}{\varepsilon} v\wedge b(x_h)e_3$.

\subsection{A survey on the notation adopted.}

\begin{rem}

In this article we are interested to obtain a global-in-time result of existence and uniqueness  for solutions of the system \eqref{primitive equations} and to some results of convergence when the Rossby and Froude number tend to zero at a comparable rate.  \\
Before stating the results that we prove let us give a brief introduction about the spaces that we are going to use.\\
All the vector fields that we consider are real i.e we consider applications of the following form $V:\mathbb{T}^3 \to \mathbb{R}^4$. We will often associate to a vector field $V$ the vector field $v:\mathbb{T}^3 \to \mathbb{R}^3$, which is simply the projection on the first three components of $V$. Moreover all the vector fields considered are periodic in all their components $x_i,i=1,2,3$ and have zero global average, i.e. $\int_{\mathbb{T}^3}v \dx =0$, which is equivalent to assume that the first Fourier coefficient $\hat{V}_{{0}} =0$. We remark that this property is preserved for the Navier-Stokes equations as well as for the primitive equations \eqref{primitive equations}. We will always work with divergence-free vector fields.\fine
\end{rem}

The anisotropy of the problem forces to introduce  anisotropic spaces, i.e. spaces which behave differently in the horizontal and vertical directions. Let us recall that, in the periodic case, the non-homogeneous Sobolev anisotropic spaces are defined by the norm
\begin{equation}
\label{eq:anisot_Sob_norm}
\left\| u\right\|_{H^{s,s'}\left(\mathbb{T}^3\right)}^2
= \left\| u\right\|_{H^{s,s'}}^2
= \sum_{n=\left(n_h,n_3\right)\in\mathbb{Z}^3} \left(1 + \left|\check{n_h}\right|^2\right)^s \left(1 + \left|\check{n_3}\right|^2\right)^{s'} \left| \hat{u}_n \right|^2 ,
\end{equation}
where we denoted $\check{n}_i=n_i/a_i, \check{n}_h=\left(\check{n}_1,\check{n}_2\right)$ and the Fourier coefficients $\hat{u}_n$ are given by $u=\sum_n \hat{u}_n e^{2\pi i\check{n}\cdot x}$. In the whole text $\mathcal{F}$ denotes the Fourier transform and $\mathcal{F}^v$ the Fourier transform in the vertical variable. \\
 We are interested to study the regularity of the product of two distributions (which is a priori not well defined), in the framework of Soboled spaces it can be proved (see \cite{gallagher_schochet}) the following product rule
\begin{lemma}
Let $ u, v $ be two distributions with zero average  defined on $ H^s \left( \T^d \right) $ and $  H^t \left( \T^d \right)  $ respectively, with $ s+t>0, \ s, t <d/2 $, then 
\begin{equation*}
\left\| u\cdot v \right\|_{H^{s+t-d/2} \left( \T^d \right)} \leqslant C_{s, t}\left\| u \right\|_{H^s \left( \T^d \right)}\left\| v \right\|_{H^t \left( \T^d \right)}.
\end{equation*}
\end{lemma} 
As in classical isotropic spaces (see \cite{adams}) if $ s>1/2 $ the space $ H^s \left( \T^1_v \right) $ is a Banach algebra. Combining this fact with the above lemma we deduce the following result which we shall apply all along the paper
\begin{lemma}\label{lem:prod_Sob_anisotropic}
Let $ u\in H^{s_1, s'}, \ v\in H^{s_2, s'} $ distributions with zero horizontal average with $ s_1 + s_2 > 0, \ s_1, s_2 <1 $ and $ s' > 1/2 $, then $ u\cdot v \in H^{s_1+s_2 -1, s'} $ and the following bound holds true
\begin{equation*}
\left\| u \cdot v  \right\|_{H^{s_1+s_2 -1, s'}} \leqslant C \left\| u \right\|_{H^{s_1, s'}} \left\| v \right\|_{H^{s_2 , s'}}.
\end{equation*}
\end{lemma}

Let us recall as well the definition of the anisotropic Lebesgue spaces, we denote with $L^p_hL^q_v$ the space $L^p\left( \mathbb{T}^2_h; L^q\left(\mathbb{T}^1_v\right)\right)$, defined by the norm:
$$
\left\| f\right\|_{L^p_hL^q_v}=\left\| \left\| f\left( x_h,\cdot\right) \right\|_{L^q\left(\mathbb{T}^1_v\right)} \right\|_{L^p\left( \mathbb{T}^2_h\right)}=
\left( \int_{\mathbb{T}^2_h}\left( \int_{\mathbb{T}^1_v} \left| f\left( x_h,x_3\right)\right|^q \d x_3\right) ^\frac{p}{q} \d x_h \right)^\frac{1}{p},
$$
in a similar way we define the space $L^q_vL^p_h$. It is well-known that the order of integration is important as it is described in the following lemma
\begin{lemma} 
Let $1\leqslant p \leqslant q$ and $f: X_1\times X_2 \to \mathbb{R}$ a function belonging to $L^p\left( X_1; L^q\left( X_2\right) \right)$ where $\left( X_1; \mu_1\right),\left(X_2;\mu_2\right)$ are measurable spaces, then $f\in L^q\left( X_2; L^p\left( X_1 \right)\right)$ and we have the inequality
$$
\left\|f\right\|_{L^q\left( X_2; L^p\left( X_1 \right)\right)}\leqslant \left\|f\right\|_{L^p\left( X_1; L^q\left( X_2\right) \right)}
$$
\end{lemma}

In the anisotropic setting the H\"older inequality becomes;
$$
\left\| fg\right\|_{L^p_hL^q_v} \leqslant \left\|f\right\|_{L^{p'}_hL^{q'}_v}\left\| g\right\|_{L^{p''}_hL^{q''}_v},
$$
where $1/p=1/p'+1/p'',1/q=1/q'+1/q''$.

\subsection{Results.}

We recall at first the result of local existence and uniqueness of solutions for Navier-Stokes equations without vertical viscosity  and periodic initial conditions.
\begin{theorem}\label{thm:local_ex_strong_solutions}
Let $s>1/2$ and $V_0 \in \cPHs \left( \mathbb{T}^3\right)$  a divergence-free vector field. Then there exists a time $T>0$ independent of $\varepsilon$ and a unique solution $V^\varepsilon$ for the system \eqref{primitive equations} which belongs to the space
\begin{align*}
V^\varepsilon \in & \ \mathcal{C}\left(\left[0,T\right]; \cPHs\right),
 &
  \nh V^\varepsilon \in &\ L^2\left(\left[0,T\right]; \cPHs\right),
\end{align*}
moreover $ \left( V^\varepsilon \right)_{\varepsilon>0} $ is uniformly bounded (in $ \varepsilon $) in the space
\begin{align*}
V^\varepsilon \in & \ L^\infty \left(\R_+; \cPLtwo\right),
 &
  \nh V^\varepsilon \in &\ L^2\left(\R_+ ; \cPLtwo \right).
\end{align*}
\end{theorem}

The existence part of Theorem \ref{thm:local_ex_strong_solutions} has been proved in \cite{chemin_et_al}, while the uniqueness (in the same energy space) has been proved in \cite{iftimie_uniqueness_NS_anisotropic}.

\begin{rem}
 We want to point out that, as  it has been proved by M. Paicu in the work \cite{paicu_NS_periodic} (see Proposition \ref{propagation norms NS horizontal}) the maximal lifespan does not depend on the regularity of the initial data, as long as $ V_0 \in \cPHs, s>1/2 $.\fine
\end{rem}

In order to state the result of convergence for the sequence of solutions of \eqref{primitive equations} and the global existence results for the limit system we briefly introduce the "filtering operator" and the limit system. Let $\mathcal{L}\left( \tau \right)$ be the semigroup generated by $\mathbb{P} \mathcal{A}$, where $\mathbb{P}$ is the Leray projector on the divergence-free vector fields on the first three components, while leaves unchanged the fourth. In particular the Leray projector in three dimensions  is given by the formula $\mathbb{P}^{(3)}=1-\mathcal{R}^{(3)}\otimes \mathcal{R}^{(3)}$, where $\mathcal{R}^{(3)}$ is the three dimensional Riesz transform given by
$$
\mathcal{R}^{(3)}=\left(
\begin{array}{ccc}
\frac{\partial_1}{\sqrt{-\Delta}},& \frac{\partial_2}{\sqrt{-\Delta}},& \frac{\partial_3}{\sqrt{-\Delta}}
\end{array}\right),
$$
while $\mathcal{A}$ is the matrix defined in \eqref{matrici}. In the same way we define the operators $ \Lambda=\sqrt{-\Delta}, \Lh=\sqrt{-\Delta_h}, \Lv=  \left| \partial_3 \right| $.
\\
 Let $\mathcal{L}(t) V_0$ be the unique global solution of 
$$
\left\lbrace
\begin{array}{l}
\partial_t V_{\text{L}}+ \mathbb{P} \mathcal{A}\ V_{\text{L}}=0\\
V_{\text{L}}\bigr|_{t=0}=V_0.
\end{array}\right.
$$
Let us further define $U^\varepsilon=\Lminus  V^\varepsilon$. We will denote $U^\varepsilon$ as the sequence of filtered solutions, we define
\begin{align*}
\mathcal{Q}^\varepsilon \left( U,V\right)=& \Lminus \mathbb{P} \left[ \Lplus U \cdot \nabla \Lplus V \right], &
\mathbb{D}^\varepsilon U =& \Lminus \mathbf{D} \Lplus U,
\end{align*}
where \textbf{D} is defined in \eqref{matrici}, and we consider their limits $\mathcal{Q},\mathbb{D}$ in $\mathcal{D}'$ (we shall see that these limit exists). We introduce the following system which we will denote as limit system
\begin{equation}\tag{S}\label{Slim}
\left\lbrace
\begin{array}{l}
\partial_t U +\mathcal{Q}\left( U, U \right) -\mathbb{D}U = 0\\
\dive u =0\\
\left. U\right|_{t=0}=V_0,
\end{array}
\right.
\end{equation}
We stress out the fact that the limit system \eqref{Slim} is similar to a 3D \NS\ system. We use on purpose the ambiguous word "similar" since the bilinear form $ \mathcal{Q} $ shares many similarities with the transport form, but it is substantially different as far as concerns product laws. This improved regularity is what will allow us to prove that \eqref{Slim} is globally well posed in some suitable critical space of low regularity.

Since we are working in the periodic setting resonant effect may play an important role. Other authors (see, for instance,  \cite{BMNresonantdomains} and \cite{paicu_rotating_fluids} ) have already dealt with this kind of problem.\\
\begin{definition}\label{resonance set}
The resonant set $\mathcal{K}^\star$ is the set of frequencies such that
\begin{align*}
\mathcal{K}^\star = &
\left\lbrace \left( k,m,n \right)\in \mathbb{Z}^9 \left|\hspace{3mm} \omega^a(k)+ \omega^b(m)=\omega^c(n) \text{ with } k+m=n, \hspace{3mm} \left( a,b,c\right) \in \left\lbrace -,+ \right\rbrace
 \right.\right\rbrace,\\
 = &
\left\lbrace \left( k,n \right)\in \mathbb{Z}^6 \left|\hspace{3mm} \omega^a(k)+ \omega^b(n-k)=\omega^c(n), \hspace{3mm} \left( a,b,c\right) \in \left\lbrace -,+ \right\rbrace
 \right.\right\rbrace, 
\end{align*}
where $\omega^a, \ a=\pm$ are the eigenvalues of a suitable operator (see Section \ref{filtering operator} for further details) which in particular take the following form
\begin{align*}
i\;\omega^\pm (n) = & \pm \frac{i}{F} \frac{\sqrt{{\left|\check{n}_h\right|^2+ {F}^2 \check{n}_3^2}} }{\left|\check{n}\right|}.
\end{align*}
We may as well associate a resonant  space to a determinate frequency $n$, in this case we define 
$$
\mathcal{K}^\star_n = 
\left\lbrace \left( k,m \right)\in \mathbb{Z}^6 \left|\hspace{3mm} \omega^a(m)+ \omega^b(k)=\omega^c(n) \text{ with } k+m=n, \hspace{3mm} \left( a,b,c\right) \in \left\lbrace -,+ \right\rbrace
 \right.\right\rbrace.
$$
\end{definition}
\begin{definition}
We say that the torus $\mathbb{T}^3$ is non-resonant if $\mathcal{K}^\star=\emptyset$.
\end{definition}

Tori which are non-resonant, are, generally, a better choice since the oscillating part of the solution satisfies a linear equation (see  \cite{gallagher_schochet}). Indeed though a generic torus may as well present resonant effects. For this reason we introduce the following definition
\begin{definition}\label{def:condition_P}
We say that a torus $ \T^3\subset \R^3 $ satisfies the condition \setword{$ \left( \mathcal{P} \right) $}{cP} if either one or the other of the following situation is satisfied:
\begin{enumerate}
\item $ \T^3 $ is non-resonant.
\item \label{item:second_cond_P} If $ \T^3 $ is resonant, the Froude number $ F^2 $ is rational, and either
    \begin{itemize}
    \item $ a_3^2/a_1^2 \in \mathbb{Q} $ and $ a_3^2/a_2^2 $ is not algebraic of degree smaller or equal than four.
    \item$ a_3^2/a_2^2 \in \mathbb{Q} $ and $ a_3^1/a_2^2 $ is not algebraic of degree smaller or equal than four.
    \end{itemize}
\end{enumerate}
\end{definition} 
\begin{rem}
The above definition (Definition \ref{def:condition_P})  is motivated  in Section \ref{sec:propagation_horizontal_average}. In fact the point \ref{item:second_cond_P} ensures us that even with resonant effects  we can propagate the horizontal average of the initial data, thing that, generally, is not true for \NS\ equations.\fine
\end{rem}

Although \eqref{Slim} is an hyperbolic system in the vertical variable we are able to prove that there exist weak (in the sense of distributions) global solutions. This has been first remarked by M. Paicu in \cite{paicu_rotating_fluids} and it due to the fact that the limit bilinear form $ \mathcal{Q} $ has in fact better product rules than the standard bilinear form in \NS\ equations (see as well Lemma \ref{product rule}). The complete statement of the theorem is the following one.
\begin{theorem}\label{thm:Leray_sol}
Let $\mathbb{T}^3$ be a 3-dimensional torus in $\mathbb{R}^3$ and let $ F\neq 1 $, for each divergence-free vector field $V_0\in \cPLtwo$ and $ \Omega_0 = -\partial_2  v^1_0 + \partial_1  v^2_0 -F\partial_3 T_0 \in \cPLtwo $ there exists a distributional solution to the system
\begin{equation} \tag{\ref{Slim}}
\label{lim syst}
\left\lbrace
\begin{array}{l}
\partial_t U+ \mathcal{Q}\left(U,U\right)-\mathbb{D}U= 0\\
\dive u=0\\
\bigl. U \bigr|_{t=0}= V_0,
\end{array}
\right.
\end{equation}
in the space $ \mathcal{D}' \left( \R_+\times\T^3 \right) $ 
which moreover belongs to the space 
\begin{align*}
U\in & L^\infty \left( \mathbb{R}_+; \cPLtwo\right) & \nh U \in & L^2 \left( \mathbb{R}_+; \cPLtwo\right),
\end{align*}
and satisfies the following energy estimate
$$
\left\| U(t)\right\|_\cPLtwo^2 +c \int_0^t \left\|\nh U(s)\right\|_\cPLtwo^2 \d s \leqslant  C  \left\| U_0\right\|_\cPLtwo^2,
$$
where the constant $c=\min \left\{\nu_h, \nu'_h\right\}>0$.\\
\end{theorem}

\noindent
We remark that Theorem \ref{thm:Leray_sol} holds for any three-dimensional torus. We do not require the condition \ref{cP} to hold.\\

\noindent
A natural question we address to is whether the system \eqref{primitive equations} converges (even in a weak sense) to the limit system \eqref{lim syst} as $ \varepsilon \to 0 $. This is the scope of the following theorem:

\begin{theorem}\label{thm:weak_conv_limit_system_cP}
Let be the initial data $ V_0 $ be as in Theorem \ref{thm:Leray_sol}, then defining the operator 
\begin{equation*}
\mathcal{L} \left( \tau \right) = e^{-\tau \PA},
\end{equation*}
and denoting as $ U $ the distributional solution of the limit system \eqref{lim syst} identified in Theorem \ref{thm:Leray_sol} the following convergence holds in the sense of distributions
\begin{equation*}
V^\varepsilon - \Lplus U \to 0.
\end{equation*}
Moreover the unknown $ U $ satisfying the evolution of the limit system \eqref{Slim} can be described as the superposition of the evolution of $U=U_{\QG}+U_{\osc}= V_{\QG}+U_{\osc}$ where $V_{\QG}$ satisfies the following system
 \begin{equation*}
 \left\lbrace
 \begin{array}{l}
 \partial_t V_{\QG} +a_{\QG}\left( D_h\right)V_{\QG}=- \left(
 \begin{array}{c}
 \nh^\perp\\
 0\\
 -F \partial_3 
 \end{array}\right)
 \Delta^{-1}_F\left( v_{\QG}^h \cdot \nh \Omega\right),\\
 \diveh v^h_{\QG}= \dive  v_{\QG}=0\hfill\\
 V_{\QG}\Bigr|_{t=0}=V_{{\QG},0},
  \end{array}
 \right.
 \end{equation*}
and $U_{\osc}$ satisfies 
\begin{equation*}
\left\lbrace
\begin{array}{rcl}
\partial_t U_{\osc}+\mathcal{Q}\left(V_{\QG},U_{\osc}\right)+
\mathcal{Q}\left(U_{\osc},V_{\QG}\right) + \mathcal{Q}\left(U_{\osc},U_{\osc}\right) +a_{\osc}\left(D\right)U_{\osc}&=&0\\
\dive u_{\osc} =0\hfill\\
U_{\osc}\Bigr|_{t=0}=U_{{\osc},0}=\left(V_0\right)_{\osc}.\hfill
\end{array}
\right.
\end{equation*}
\end{theorem}

In particular $ a_{\QG} $ and $ a_{\osc} $ are Fourier symbols such that there exist two positive constants $ 0<c_1 \leqslant c_2 <\infty $ such that, for any tempered distribution $ u $
$$
c_1 \left| \xi_h \right|^2 \leqslant \left| a_{\QG} \left( \xi \right) \hat{u} \left( \xi \right)\right|, \ \left| a_{\osc} \left( \xi \right) \hat{u} \left( \xi \right)\right| \leqslant c_2 \left| \xi_h \right|^2,
$$
and $\mathcal{Q}$ is a bilinear form which shares  many aspects with the more classical bilinear form of \NS\ equations, but has better properties as far as the
regularity of the product is concerned (see Section \ref{bil form limit Section} for the product rules and \eqref{limit quadratic} for the proper definition).\\

Thanks to some a priori estimates on the solutions of the limit system \eqref{Slim} we can state the following improvement of the above theorem, at the cost of having well prepared initial data and tori which satisfy Condition \ref{cP}. 
We say that a data $V_0$ is well prepared if it has zero horizontal mean, i.e. $\int_{\mathbb{T}^2_h}V_0\left( x_h, x_3\right)\d x_h=0$. This property is conserved by the limit system \eqref{Slim}  for almost every torus in $\mathbb{R}^3$ (see Lemma \ref{propagation horizontal mean VQG}, \ref{propagation horizontal mean osc part}). Moreover we ask as well that the potential vorticity, defined as
\begin{equation}
\label{prima definizione Omega}
\Omega \left(t,x\right)= -\partial_2 U^{1}\left(t,x\right) +\partial_1 U^{2}\left(t,x\right) -F\partial_3 U^4 \left(t,x\right).
\end{equation}
 is an $\cPHs$ function at time $t=0$, i.e. $\Omega_0\in\cPHs$.

\begin{theorem}\label{GWP2}
Let $\mathbb{T}^3 $ satisfy the condition \ref{cP} and consider a vector field $U_0\in \cPHs$ with zero horizontal average and $\Omega_0\in \cPHs$, for $s\geqslant 1$ such that $\dive u_0=0$,  $ F\neq 1 $ the limit system \eqref{Slim} admits a global solution 
\begin{align*}
U\in & L^\infty\left(\mathbb{R}_+; \cPHs\right), &
\nh U\in & L^2 \left(\mathbb{R}_+; \cPHs\right),
\end{align*}
such that satisfies the following energy bound
$$
\left\|U(t)\right\|^2_\cPHs + c \int_0^t \left\|\nh U(s)\right\|^2_\cPHs \leqslant \mathcal{E}\left( \left\| U_0 \right\|_\cPHs^2 \right),
$$
where $ \mathcal{E} $ is a suitable (bounded on compact sets) function. Moreover the solution $ U $ is unique in the space $ L^\infty \left( \R_+; H^{0,\sigma} \right)\cap L^2 \left( \R_+, H^{1,\sigma} \right) $ for $ \sigma\in [-1/2, s) $.
\end{theorem}

This refined regularity is be crucial in order to prove the convergence result to solutions of \eqref{primitive equations}.
\begin{rem}\label{rk_difference_with_paicu}
Compared to the work of M. Paicu \cite{paicu_rotating_fluids} the author requires only $ s>1/2 $. This discrepancy is due to the fact that in the present work the limit system is well-posed only for $ s\geqslant 1 $, indeed we are able to propagate $ \cPHs, s\geqslant 0 $ norms for the potential vorticity $ \Omega $, and, as explained in Lemma \ref{higher regularity VQG}, $ \left\| V_{\QG} \right\|_{H^{0,s+1}}\lesssim \left\| \Omega \right\|_\cPHs $.
\end{rem}

The main idea in the propagation of regularity stated in Theorem \ref{GWP2} is that we can recover the missing viscosity in the vertical direction using the fact that the vector field $ u $ is divergence-free. We can in fact observe that in the nonlinear term the vertical derivative is always multiplied by the third component $ u^3 $ of the vector field considered (i.e. terms of the form $ u^3 \partial_3 $). We hence remark the fact that the term $\partial_3 u^3$ is more regular thanks to the relation $ -\partial_3 u^3 = \diveh u^h $, and due to the fact that the horizontal viscosity has a regularizing effect on the derivatives in the horizontal variable $ x_h $.

We can at this point state the theorem of convergence that connects the solution of \eqref{primitive equations} with the  solutions of the limit system \eqref{Slim}.
\begin{theorem}\label{conv theorem}
Let $\mathbb{ T}^3\subset\mathbb{R}^3 $ satisfy the condition \ref{cP}, $\Omega_0= -\partial_2 v^1_0 + \partial_1 v^2_0 - F\partial_3 T_0\in \cPHs $, $ U_0\in \cPHs $ with zero horizontal average. Let $V_0\in \cPHs$ for $s>1$ a divergence free vector field,  let $ V^\varepsilon $ be a local solution of \eqref{primitive equations} and $ U $ be the unique global solution of the limiti system \eqref{Slim}. Then the following convergences take place
\begin{align*}
\lim_{\varepsilon\to 0} \left( V^\varepsilon-\Lplus U\right)&=0 & \text{in } &\mathcal{C}\left(\mathbb{R}_+; H^{0,\sigma}\right) \\
\lim_{\varepsilon\to 0} \nh\left( V^\varepsilon-\Lplus U\right)&=0 & \text{in } & L^2\left(\mathbb{R}_+; H^{0,\sigma}\right)
\end{align*}
for $\sigma \in \left[ 1, s \right) $.
\end{theorem}

The paper is divided as follows
\begin{itemize}
\item In Section \ref{preliminaries} we introduce some mathematical  tools that will be useful in the development of the paper.

\item  Section \ref{filtering operator} we provide a careful analysis of the spectral properties of the linear system whose evolution is determined by the operator $ \PA $. 
 In Subsection \ref{sec:global_splitting_bilinear} we refer to some results proved in the works\cite{embid_majda},  \cite{embid_majda2} and \cite{BMN-P} which prove how the limit bilinear interaction behaves along the the eigendirections identified in Section \ref{filtering operator} for smooth functions.
 
 \item In Section \ref{existence limit system} we prove Theorem \ref{thm:Leray_sol}. Such result is not a straightforward application of Leray Theorem since, due to the lack of the vertical diffusivity, the solutions are bounded in the space $ L^2_{\loc} \left( \R_+ ; H^{1, 0} \right) $ only, which is \textit{not compactly embedded} in $ L^2_{\loc} \left( \R_+ ; L^2 \right) $. Such observation prevents us to use standard compactness theorems in functional spaces such as Aubin-Lions lemma (see \cite{Aubin63}). Nonetheless using Fujiwara near-optimal bound (see \cite{Fujiwara_bound}) we can transform a vertical derivative $ \partial_3 $ in a multi-index of the form $ C\varepsilon \left( \partial_1^{N_1}, \partial_2^{N_2} \right) $, where $ N_1, N_2 $ may as well be large, while $ \varepsilon $ is positive and small. The system \eqref{primitive equations} has a non-zero diffusive effects in the horizontal directions, and hence we can finally prove that bilinear interactions of weakly converging (in the sense that converge w.r.t. a Sobolev topology of negative index) converge in the sense of distributions to some limiting element.

 \item In Section   \ref{sec:weak_conv} we prove Theorem \ref{thm:weak_conv_limit_system_cP}. The approach is twofold:
 \begin{itemize}
 \item Thanks to a topological argument we prove that the sequence $ \left( V^\varepsilon \right)_{\varepsilon>0} $ is compact in some weak space,
 \item A careful analysis of the bilinear interactions in the limit $ \varepsilon \to 0 $ gives us the explicit form of the bilinear limit interactions.
 \end{itemize}
 Next in Subsection \ref{sec:propagation_horizontal_average} we prove that, under some suitable geometric conditions (see Definition \ref{def:condition_P}) the limit system \eqref{lim syst} propagates globally-in-time the horizontal average of the initial data.

 \item In Section \ref{sec:prop_H0s_regularity} we prove that the limit system propagates globally-in-time $ \cPHs $ norm, at the price of having well prepared (in the sense of zero-horizontal average) initial data and domains which satisfy the Definition \ref{def:condition_P}. Hence we prove Theorem \ref{GWP2}.

\item Lastly in Section \ref{app S. method} we prove Theorem \ref{conv theorem}, i.e. that we can approximate the solutions of \eqref{primitive equations} as $\varepsilon\to 0$ with the (global) solutions of \eqref{Slim} in some topology of strong type.
\item In the Section \ref{Section energy estimates} various technical results are proved.
\end{itemize}

\section{Preliminaries.}\label{preliminaries}
This section is devoted to introduce the mathematical tools that will be used all along the paper and which are necessary to understand the contents described in the following pages. 

\subsection{Elements of Littlewood-Paley theory.}\label{elements LP}

A tool that will be widely used all along the paper is the anisotropic theory of Littlewood--Paley, which consists in doing a dyadic cut-off of the vertical frequencies.\\
Let us define the (non-homogeneous) vertical truncation operators as follows:
\begin{align*}
{\cPtv} u= & \sum_{n\in\mathbb{Z}^3} \hat{u}_n \varphi \left(\frac{\left|\check{n}_3\right|}{2^q}\right) e^{i\check{n} \cdot x}&\text{for }& q\geqslant 0\\
\triangle^v_{-1}u=& \sum_{n\in\mathbb{Z}^3} \hat{u}_n \chi \left( \left|\check{n}_3 \right| \right)e^{i\check{n} \cdot x}\\
{\cPtv} u =& 0 &\text{for }& q\leqslant -2
\end{align*}
where $u\in\mathcal{D}'\left(\mathbb{T}^3 \right)$ and  $\hat{u}_n$ are the Fourier coefficients of $u$. The functions $\varphi$ and $\chi$ represent a partition of the unity in $\mathbb{R}$, which means that are smooth functions with compact support such that
\begin{align*}
\text{Supp}\; \chi \  \subset \hspace{1mm} & B \left(0,\frac{4}{3}\right), & \text{Supp}\; \varphi\  \subset\hspace{1mm}& \mathcal{C}\left( \frac{3}{4},\frac{8}{3}\right),
\end{align*}
Moreover for each $ t\in \R $ the sequence $ \left( \chi \left( \cdot \right), \varphi \left( 2^{-q}\cdot \right) \right)_{q\in\mathbb{N}} $ is a partition of the unity. 
Let us define further the vertical cut-off operator as 
$
S^v_q u= \sum_{q'\leqslant q-1}\cPTv u.
$

\subsection{Anisotropic paradifferential calculus.}\label{paradifferential calculus}

The dyadic decomposition turns out to be very useful also when it comes to study the product between two distributions. We can in fact, at least formally, write for two distributions $u$ and $v$
\begin{align}\label{decomposition vertical frequencies}
u=&\sum_{q\in\mathbb{Z}}{\cPtv} u ; &
v=&\sum_{q'\in\mathbb{Z}}\cPTv v;&
u  v = & \sum_{\substack{q\in\mathbb{Z} \\ q'\in\mathbb{Z}}}{\cPtv} u  \cPTv v
\end{align}

We are going to perform a Bony decomposition in the vertical variable (see  \cite{bahouri_chemin_danchin_book}, \cite{Bony1981}, \cite{chemin_book} for the isotropic case and \cite{chemin_et_al},\cite{iftimie_NS_perturbation} for the anisotropic one).
\\
Paradifferential calculus is  a mathematical tool for splitting the above sum in three parts 

$$
u  v = T^v_u v+ T^v_v u + R^v\left(u,v\right),
$$
where
\begin{align*}
T^v_u v=& \sum_q S^v_{q-1} u {\cPtv} v, &
 T^v_v u= & \sum_{q'} S^v_{q'-1} v \cPTv u, &
 R^v\left( u,v \right) = & \sum_k \sum_{\left| \mu\right| \leqslant 1} \triangle^v_k u \triangle^v_{k+\mu} v.
\end{align*}

In particular the following almost orthogonality properties hold
\begin{align*}
{\cPtv} \left( \cPSvq a\; \cPTv b\right)=&0 & \text{if }& \left|q-q'\right|\geqslant 5\\
{\cPtv} \left( \cPTv a\; \triangle^v_{q'+\mu}b\right)=&0 & \text{if }& q'< q-4, \; \left| \mu \right|\leqslant 1
\end{align*}
and hence we will often use the following relation
\begin{align}
{\cPtv}\left( u  v \right)= &\sum_{\left| q -q'\right| \leqslant 4} {\cPtv}\left(S^v_{q'-1} v\; \cPTv u\right) +
\sum_{\left| q -q'\right| \leqslant 4} {\cPtv}\left(S^v_{q'-1} u \;\cPTv v\right)+
\sum_{q'\geqslant q-4}\sum_{|\mu|\leqslant 1}{\cPtv}\left(  \cPTv u\; \triangle^v_{q'+\mu}v\right)\nonumber \\
=& \sum_{\left| q -q'\right| \leqslant 4} {\cPtv}\left(S^v_{q'-1} v\; \cPTv u\right) + \sum_{q'>q-4} {\cPtv}\left( S^v_{q'+2} u \;\cPTv v\right).\label{Paicu Bony deco}
\end{align}

In the paper \cite{chemin_lerner} J.-Y. Chemin and N. Lerner introduced the following asymmetric decomposition, which was first used by J.-Y. Chemin et al. in \cite{chemin_et_al} in its anisotropic version. This particular decomposition turns out to be very useful in our context
\begin{equation}
\label{bony decomposition asymmetric}
{\cPtv} \left(uv\right) = S^v_{q-1} u {\cPtv} v
 +\sum_{|q-q'|\leqslant 4} \left\lbrace\left[ {\cPtv}, \cPSvq u \right]\cPTv v + \left( S^v_q u-\cPSvq u \right) {\cPtv}\cPTv v\right\rbrace + \sum_{q'>q-4} {\cPtv} \left(  S^v_{q'+2} v \cPTv u\right),
\end{equation}
where the commutator $\left[{\cPtv}, a\right]b$ is defined as 
$
\left[{\cPtv}, a\right]b= {\cPtv} \left( ab \right) - a {\cPtv} b.
$

All along the following we shall denote as $ \left( b_q \right)_{q\geqslant-1} $ any sequence which is summable that may depend on different parameters such that $ \sum_q b_q \leqslant 1 $ . In the same way we shall denote as $ \left( c_q \right)_q\in \ell^2 \left( \mathbb{Z} \right) $ any sequence  such that $ \sum_q c^2_q \leqslant 1 $ . As well $ C $ is a (large) positive constant independent of any parameter and $ c $ a small one, these two constants may differ implicitly from line to line. We remark that the regularity of a function can be rephrased in the following way: we say that $ u \in \cPHs $
only if there exists a sequence $ \left( c_q \right)_q $ depending on $ u $ such that
\begin{equation}
\label{eq:reg_dyadic_blocks}
\left\| \triangle_q^v u \right\|_\cPLtwo \leqslant C \ c_q \left( u \right) 2^{-qs} \left\| u \right\|_\cPHs.
\end{equation}

\subsection{Dyadic blocks and commutators as convolution operators.}
\label{sec:dy_convolution}
The dyadic blocks and the low-frequencies truncation operators can be seen as convolution operators, in particular if we denote as $ h= \mathcal{F}^{-1}\varphi $ and $ g= \mathcal{F}^{-1}\chi $ we have
\begin{align}
& {\cPtv} u= \varphi \left( 2^{-q}D \right) u = 2^{q}\int_{\mathbb{T}} h\left( 2^q y \right) u\left( x-y \right)\d y,\label{eq:dyadic_as_convolution}\\
& S_q^v u= \chi \left( 2^{-q}D \right) u = 2^{q}\int_{\mathbb{T}} g\left( 2^q y \right) u\left( x-y \right)\d y.\nonumber
\end{align}
This is due to the fact that $ {\cPtv} u \left( x \right) = \left( \mathcal{F}^v \right)^{-1}\left( \varphi \left( \cdot \right) \hat{u} \left( \cdot \right) \right) \left( x \right) $.
We introduce this alternative way to consider commutators and truncations because we need it in Section \ref{Section energy estimates}. In particular we want to express a commutator as a convolution operator, since a commutator is defined as 
$$
\left[{\cPtv}, a\right]b\left( x \right)= {\cPtv} \left( ab \right)\left( x \right) - a\left( x \right) {\cPtv} b\left( x \right),
$$
and we apply to the right hand side of the above equation the relation in \eqref{eq:dyadic_as_convolution} we obtain in fact that
\begin{equation*}
\left[{\cPtv}, a\right]b\left( x \right)= 2^{q}\int_{\mathbb{T}} h \left( x_h, x_3-y_3 \right) \left( a\left(x_h , y_3 \right)-a\left( x_h, x_3 \right) \right)b \left(x_h, y_3 \right) \d y_3.
\end{equation*}
Thanks to Taylor expansion with reminder in Cauchy form we know that
$$
a\left(x_h , y_3 \right)-a\left( x_h, x_3 \right) = \partial_3 a \left( x_h, x_3 + \tau \left( x_3-y_3 \right) \right) \left( x_3-y_3 \right),
$$
for some $ \tau\in \left( 0,1 \right) $, hence we can write the commutator as
\begin{equation}\label{eq:commutator_as_convolution}
\left[{\cPtv}, a\right]b\left( x \right)= 2^{q}\int_{\mathbb{T}}  \left( x_3-y_3 \right) h \left( x_h, x_3-y_3 \right) \partial_3 a \left( x_h, x_3 + \tau \left( x_3-y_3 \right) \right)b \left(x_h, y_3 \right) \d y_3.
\end{equation}

\subsection{Some basic estimates.}\label{basic estimates}
The interest in the use of the dyadic decomposition is that the derivative in the vertical direction of a function localized in vertical frequencies of size $2^q$ acts like the multiplication of a  factor $2^q$ (up to a constant independent of $q$ ). In our setting (periodic case) a Bernstein type inequality holds. For a proof of the following lemma we refer to the work \cite{iftimie_NS_perturbation}.

\begin{lemma}\label{bernstein inequality}
Let $u$ be a function such that $\textnormal{supp}\;\mathcal{F}^vu \subset \mathbb{T}^2_h\times 2^q\mathcal{C}$, where $\mathcal{F}^v$ denotes the Fourier transform in the vertical variable. For all integers $k$, $ p\in \left[1,\infty\right] $, $ 1\leqslant r' \leqslant r \leqslant \infty $, the following relations hold
\begin{align*}
2^{qk}C^{-k}\left\| u \right\|_{L^p_h L^r_v}\leqslant & \left\| \partial^k_{x_3}u \right\|_{L^p_h L^r_v} \leqslant 2^{qk}C^{k}\left\| u \right\|_{L^p_h L^r_v},\\
2^{qk}C^{-k}\left\| u \right\|_{ L^r_vL^p_h}\leqslant & \left\| \partial^k_{x_3}u \right\|_{ L^r_vL^p_h} \leqslant 2^{qk}C^{k}\left\| u \right\|_{ L^r_vL^p_h}.
\end{align*}

Let now $\infty \geqslant r\geqslant r' \geqslant 1$ be real numbers. Let  $\text{supp}\;\mathcal{F}^vu \subset \mathbb{T}^2_h\times 2^q B$, then
\begin{align*}
\left\| u \right\|_{L^p_h L^r_v}\leqslant & C  2^{q\left( \frac{1}{r'}-\frac{1}{r}\right)}\left\| u \right\|_{L^p_h L^{r'}_v}\\
\left\| u \right\|_{ L^r_vL^p_h}\leqslant & C  2^{q\left( \frac{1}{r'}-\frac{1}{r}\right)}\left\| u \right\|_{ L^{r'}_vL^p_h}
\end{align*}
\end{lemma}

The following are inequalities of Gagliardo-Niremberg type,  we will avoid to give the proofs of such tools since they are already present in \cite{paicu_rotating_fluids}.

\begin{lemma}
There exists a constant $C$ such that for all periodic vector fields $u$ on $\mathbb{T}^3$ with zero horizontal average ($\int_{\mathbb{T}^2_h}u\left( x_h, x_3 \right)\d x_h=0$) we have
\begin{equation}\label{GN type ineq}
\left\|u\right\|_{L^2_v L^4_h}
\leqslant C_1 \left\|u\right\|_{\Hud}
\leqslant C_2  \left\| u\right\|^{1/2}_{\cPLtwo}\left\| \nh u \right\|^{1/2}_{\cPLtwo}.
\end{equation}
\end{lemma} 

From Lemma \ref{bernstein inequality} and \eqref{GN type ineq} we can deduce the following result

\begin{cor}
Let $ u $ be a periodic vector field such that $ \textnormal{Supp} \ \mathcal{F}^v u \subset \T^2_h \times 2^q B $, then
\begin{equation}\label{eq:LinfL2embedding}
\left\| u \right\|_{L^\infty_v L^2_h} \leqslant C 2^{q/2} \left\| u \right\|_{\cPLtwo},
\end{equation}
moreover if $ u $ has zero horizontal average
\begin{equation}\label{eq:LinfL4embedding}
\left\| u \right\|_{L^\infty_v L^4_h} \leqslant C 2^{q/2} \left\| u \right\|_{\cPLtwo}^{1/2} \left\| \nh u \right\|_{\cPLtwo}^{1/2}
\end{equation}

\end{cor}

\begin{lemma}\label{lemma:sobolev_embedding_zero_average}
Let $s$ be a real number and $\mathbb{T}^3$ a three dimensional torus. For all vector fields $u$ with zero horizontal average, the following inequality holds
\begin{equation}
\label{GN type ineq2}
 \left\|u\right\|_{H^{1/2, s}}
\leqslant C  \left\|u\right\|_{H^{0,s}}^{1/2} \left\|\nh u\right\|_{H^{0,s}}^{1/2}
\end{equation}
\end{lemma}

\begin{cor}\label{L4Linf embedding}
Let $s>1/2$. There exists a constant $C$ such that the inequality
$$
\left\| u\right\|_{L^\infty_v L^2_h}\leqslant C\left\|u\right\|_{H^{0,s}}
$$
holds. Moreover if $u$ is of zero horizontal average we have
$$
\left\| u\right\|_{ L^\infty_v L^4_h}\leqslant C  \left\|u\right\|^{1/2}_{H^{0,s}} \left\|\nh u\right\|^{1/2}_{H^{0,s}}
$$
\end{cor}

Finally we state a lemma that shows that the commutator with the truncation operator in the vertical frequencies is a regularizing operator. The proof of such lemma can be found in \cite{paicu_NS_periodic}.
\begin{lemma}\label{estimates commutator}
Let $\mathbb{T}^3$ be a 3D torus and $p,r,s$ real positive numbers such that $\infty \geqslant r',s',p,r,s\geqslant 1 $ $\frac{1}{r'}+\frac{1}{s'}=\frac{1}{2}$ and $\frac{1}{p}=\frac{1}{r}+\frac{1}{s}$. There exists a constant $C$ such that for all vector fields $u$ and $v$ on $\mathbb{T}^3$ we have the inequality
$$
\left\| \left[ {\cPtv} , u\right] v \right\|_{L^2_vL^p_h}\leqslant C  2^{-q}\left\|\partial_3 u \right\|_{L^{r'} _v L^r_h} \left\|v\right\|_{L^{s'}_vL^s_h}
$$
\end{lemma}

\subsection{Preliminary results on the \NS\ equations with zero vertical diffusivity.}

A primary tool in the study of the convergence of the primitive equations \eqref{primitive equations}  to the limit system \eqref{lim syst} will be a careful study of the Navier-Stokes equation with only horizontal diffusion
\begin{equation}
\tag{NS$_h$}\label{NS horozontal}
\left\lbrace
\begin{array}{lll}
\partial_t v+ v\cdot \nabla v -\nu_h \Delta_h v+\nabla p=0& \text{in} & \mathbb{R}_+\times\mathbb{T}^3\\
\dive v =0\\
\left. v\right|_{t=0}=v_0
\end{array}
\right.
\end{equation}

This equation in the case of the periodic data on $\mathbb{T}^3$ has been carefully studied in \cite{paicu_NS_periodic}, hence we will refer to this work as we go along.\\

Indeed the equation satisfied by $U^\varepsilon$, i.e. \eqref{filtered systemP} is a Navier-Stokes equation with zero vertical diffusion and hence can be well described by the system \eqref{NS horozontal}. Here we start giving the following energy estimate for three-dimensional anisotropic Navier-Stokes equations
\begin{prop}\label{propagation norms NS horizontal}
Let $s\geqslant s_0 >\frac{1}{2}$ and $v$ a solution of \eqref{NS horozontal} belonging to the space $\mathcal{C}\left( \left[0,T\right];  H^{0,s}\right)$ whose horizontal gradient $\nh v \in L^2\left( \left[0,T\right]; H^{0,s}\right)$. Let us suppose moreover that $v=\underline{v}+\tilde{v}$ where $\underline{v}$ is the horizontal average of $v$ and $\tilde{v}$ has zero horizontal mean. Suppose moreover that $\left\|\underline{v}\left(t\right)\right\|_{H^{s_0}_v}\leqslant c a_3^{-1}\nu_h$ in $\left[0,T\right]$, then for $t\in \left[0,T\right]$
\begin{equation*}
\left\| v(t)\right\|_{H^{0,s}}^2+\nu_h\int_0^t \left\| \nh v\left(\tau \right)\right\|_{H^{0,s}}^2\d \tau
\leqslant \left\| v_0\right\|_{H^{0,s}}^2 \exp \left( C \int_0^t \left\| \nh v\left(\tau \right)\right\|_{H^{0,s_0}}^2\d \tau+C \int_0^t 
\left\| v\left(\tau \right)\right\|_{H^{0,s_0}}^2
\left\| \nh v\left(\tau \right)\right\|_{H^{0,s_0}}^2 \right)
\end{equation*}
\end{prop}

\begin{rem}
Proposition \ref{propagation norms NS horizontal} has been proved by M. Paicu for  $s\geqslant s_0 >\frac{1}{2}$. Indeed in \cite{paicu_rotating_fluids} the limit system was a coupling between a 2d \NS\  system and the oscillating part. Indeed the 2d \NS\  system is globally well posed if the initial data depends on $ x_h $ only and it is in $ \cPHs $ for $s\geqslant 0$. The oscillating part instead is globally well posed in   $ \cPHs $ for $ s>1/2 $. In our case though the limit flow is the sum of $ V_{\QG} $ satisfying  \eqref{quasi geostrophic equation} and the oscillating part $ U_{\osc} $ which are two three-dimensional vector fields. Now, $ U_{\osc} $ is globally well posed in $ \cPHs $ for $ s>1/2 $ (see Proposition \ref{propagation H0s norms for oscillating part}), but $ V_{\QG} $ is globally well posed in $ \cPHs $ for $ s\geqslant 1 $ (see Proposition \ref{propagation H0s norms for Omega} and Lemma \ref{higher regularity VQG}). This is why in the following as long as we are required to apply Proposition \ref{propagation norms NS horizontal} we shall use the index $ s_0>1 $ instead that $ s_0>1/2 $.\fine
\end{rem}

For a proof of Proposition \ref{propagation norms NS horizontal} we refer to the works \cite[Proposition 3.1]{paicu_rotating_fluids} and \cite{paicu_NS_periodic}.
\\

Given any vector field $ A $ we denote 
\begin{equation*}
\underline{A} \left( x_3 \right)=\frac{1}{\left| \T^2_h \right|}\int_{\T^2_h} A\left( y_h, x_3 \right)\d y_h,
\end{equation*}
and
\begin{equation*}
\tilde{A}\left( x_h, x_3 \right)= A \left( x_h, x_3 \right) - \underline{A} \left( x_3 \right).
\end{equation*}

\begin{prop}\label{uniqueness anisotropic NS}
Let $s>\frac{1}{2}$ and $\mathbb{T}^3$ an arbitrary torus  and $w\in \mathcal{C} \left( \left[0,T\right]; H^{0,s}\right), \nh w\in L^2\left(\left[0,T\right]; H^{0,s} \right)$ a solution of the problem
\begin{equation}
\left\lbrace
\begin{array}{l}
\partial_t w + w\cdot \nabla w +u\cdot \nabla w +w\cdot \nabla u -\nu_h\Delta_h w +\nabla p= f\\
\dive w = 0\\
\left.w\right|_{t=0}=w_0,
\end{array}
\right.
\end{equation}
where $u\in \mathcal{C} \left(\left[0,T\right]; H^{0,s}\right), \nh u\in L^2\left(\left[0,T\right]; H^{0,s} \right)$ a divergence-free vector field  such that its horizontal average satisfies $\left\|\underline{u}\left(t\right)\right\|_{H^s_v}\leqslant c a_3^{-1}\nu_h$ for all $t\in \left[0,T\right]$ and $f= \underline{f}+\tilde{f}$ is such that
\begin{align*}
\underline{f}\in & L^1\left( \left[0,T\right]; H_v^{-\frac{1}{2}} \right), \\
\tilde{f} \in & L^2\left(\left[0,T\right]; H^{-1,-\frac{1}{2}}\right).
\end{align*}
Then there exists a constant $C>0$ such that we have for all $t\in \left[0,T\right]$
\begin{multline*}
\left\|w(t)\right\|^2_{H^{0,-\frac{1}{2}}} +\nu_h \int_0^t \left\|\nh w (s)\right\|^2_{H^{0,-\frac{1}{2}}} \d s \\ 
\leqslant
C \left( \left\|w_0\right\|^2_{H^{0,-\frac{1}{2}}} +\int_0^t \left\|\tilde{f}(s)\right\|^2_{H^{-1,-\frac{1}{2}}}\d s +\int_0^t \left\| \underline{f}(s)\right\|_{H^{-\frac{1}{2}}_v}\d s
\right) \\
\times \exp\left\lbrace \int_0^t \left\|\underline{f}(s)\right\|_{H_v^{-\frac{1}{2}}}\d s + \int_0^t \left( 1+ \left\| w (s)\right\|^2_{H^{0,s}}\right)\left\|\nh w (s)\right\|^2_{H^{0,s}} \d s\right.
 +
\left.\int_0^t \left( 1+ \left\| u (s)\right\|^2_{H^{0,s}}\right)\left\|\nh u (s)\right\|^2_{H^{0,s}} \d s
\right\rbrace .
\end{multline*}
\end{prop}
\begin{proof}
\cite[Proposition 3.2, p. 182]{paicu_rotating_fluids}
\end{proof}

\begin{rem}
Let us remark the fact that we impose two different kind of regularities on the exterior force. In order to obtain global results in time we shall apply this proposition for bulk forces which are $f\in L^1\left( \mathbb{R}_+, H^{-1,-1/2} \right)\cap L^2\left( \mathbb{R}_+, H^{-1,-1/2} \right)$.\fine
\end{rem}

\section{The  filtering operator $\mathbb{P}\mathcal{A}$.}\label{filtering operator}
Let us consider the following linear equation
\begin{equation}\label{eq:linear_system}
\left\lbrace
\begin{array}{l}
\partial_t V_{\text{L}}+ \PA \ V_{\text{L}}=0\\
\left. V_{\text{L}} \right|_{t=0}= V_0
\end{array}
\right.,
\end{equation}
where $\mathbb{P}$ is the Leray projection onto the divergence free vector fields, without changing $V_{\text{L}}^4$. The Fourier multiplier associated to $\mathbb{P}$ has the following form
\begin{equation}
\label{eq:Leray_projector}
\mathbb{P}_n=1-\frac{1}{|\check{n}|^2}\left(
\begin{array}{cccc}
\frac{n_1^2}{a_1^2} & \frac{n_1 n_2}{a_1 a_2} & \frac{n_1n_3}{a_1a_3} & 0\\
\frac{n_2 n_1}{a_2 a_1} & \frac{n_2^2}{a_2^2} & \frac{n_2 n_3}{a_2 a_3} &0\\
\frac{n_3 n_1}{a_3 a_1} & \frac{n_3 n_2}{a_3 a_2} &  \frac{n_2^3}{a_2^3} &0\\
0&0&0&0
\end{array}  
\right),
\end{equation}
where $\left|\check{n}\right|^2=\sum_j\frac{n_j^2}{a_j^2}$ and $1$ is the identity matrix on $\mathbb{C}^4$. The operator $ \mathcal{A} $ was defined in \eqref{matrici}. The solution to the linear equation is indeed $ V_{\text{L}}(\tau)= e^{-\tau \PA} V_0$. We denote the propagator operator $ e^{-\tau \PA}  $ as $ \mathcal{L}\left( \tau \right)  $. 
One can compute the matrix $\mathbb{P}_n\mathcal{A}$
\begin{equation}
\label{eq:PnA}
\mathbb{P}_n\mathcal{A}=
\left(
\begin{array}{cccc}
 -\frac{{n}_1 {n}_2}{\left| n \right|^2} & -1+\frac{{n}_1^2}{n^2} & 0 & -\frac{{n}_1 {n}_3}{F \left| n \right|^2} \\[3mm]
 1-\frac{{n}_2^2}{\left| n \right|^2} & \frac{{n}_1 {n}_2}{\left| n \right|^2} & 0 & -\frac{{n}_2 {n}_3}{F \left| n \right|^2} \\[3mm]
 -\frac{{n}_2 {n}_3}{\left| n \right|^2} & \frac{{n}_1 {n}_3}{\left| n \right|^2} & 0 & \frac{1}{F} \left( {1-\frac{{n}_3^2}{\left| n \right|^2}} \right) \\[3mm]
 0 & 0 & -\frac{1}{F} & 0
\end{array}
\right)
, 
\end{equation}
whose eigenvalues are 
\begin{align}
\omega^0(n)&=0, & i\;\omega^\pm(n)=& \pm  \frac{i}{F} \frac{\sqrt{{\left|\check{n}_h\right|^2}+F^2\check{n}_3^2}}{\left|\check{n}\right|},\label{eigenvalues}
\end{align}
where the eigenvalue $\omega^0$ has multiplicity 2, and we can write $\omega^\pm=\pm \omega$. The associated normalized eigenvectors are 
\begin{equation}
\begin{aligned}
e^0(n)=& \; \frac{1}{\left| \check{n} \right|_F}
 \left(
\begin{array}{c}
{-\check{n}_2}\\
{\check{n}_1}\\
0\\
{-F\check{n}_3}
\end{array}
\right), &
e^\pm(n)=&\; \frac{1}{{\left( 1+F^2 \left| \omega \left( n \right) \right|^2 \right) \left| \check{n}_h \right|^2 \left| \check{n} \right|^2}}
\left(
\begin{array}{c}
{- F\check{n_3}\left(  \check{n_2}\mp i \check{n_1} \omega \left( n \right)  \right)}\\
 {F\check{n}_3 \left( \check{n}_1 \pm i \check{n}_2 \omega \left( n \right) \right)}\\
 {\mp i F \omega(n)\left|\check{n_h}\right|^2}
\\
{\left|\check{n_h}\right|^2}
\end{array}
\right),\label{cP:eigenvectors}
\end{aligned}
\end{equation}
where $ \left| \check{n} \right|_F= \sqrt{ \check{n}_1^2 + \check{n}_2^2 + F^2 \check{n}_3^2 } $,  if $\left|n_h\right|, n_3\neq 0$, otherwise, respectively
\begin{align}
\label{eigenvectors zero}
e^\pm(0, n_3)=&\frac{1}{\sqrt{2}}\left(
\begin{array}{c}
\pm i \\
1\\
0\\
0
\end{array}
\right),
&
e^\pm(n_h, 0)=&\frac{1}{\sqrt{2}}\left(
\begin{array}{c}
0\\
0\\
\pm i\\
1
\end{array}
\right)
\end{align}
The eigenvalue $ \omega^0 $ has algebraic multiplicity 2, but there is only one eigenvector related to it, namely $ e^0 $. Indeed the matrix $ \PA $ has a nontrivial $ 2\times 2 $ Jordan block structure associated to the eigenvalue 0, hence the fourth is a generalized eigenvector $ \tilde{e}^0 $. This though is not divergence-free, hence it shall play no role in the evolution of the system \eqref{primitive equations}, for this reason it is omitted.   For a more detailed discussion on the spectral properties of the linear system we refer the reader to the papers \cite{embid_majda} and \cite{embid_majda2}. 
\\
Once we have  introduced the eigenvectors in \eqref{cP:eigenvectors} we can consider a generic divergence-free vector field $ V $ as direct sum  of the elements belonging to $ \C e^0 $ and $ \C e^-\oplus \C e^+ $. We shall call the projection of $ V $ onto $ \C e^0 $ the {\qg} part, while the projection onto $ \C e^-\oplus \C e^+ $ the oscillating part. The projection can be explicitly defined as follows
\begin{equation}
\begin{aligned}\label{eq:QG-osc}
V_{\QG} = & \mathcal{F}^{-1} \left( \left(\left. \hat{V}_n \right| e^0 \left( n \right)  \right)_{\mathbb{C}^4} e^0\left( n \right) \right), & \hspace{1cm}
V_{\osc} = & \sum_{i=\pm} \mathcal{F}^{-1} \left( \left(\left. \hat{V}_n \right| e^i \left( n \right)  \right)_{\mathbb{C}^4}  e^i\left( n \right) \right).
\end{aligned}
\end{equation}
The element $ V_{\osc} $ is called oscillating because is the only part of the initial vector field $ V_0 $ which is affected in the evolution of the system \eqref{eq:linear_system}, $ V_{\QG} $ stays still being in the kernel of $ \PA $.\\
We would like to point out the following relevant fact, the non-oscillating eigenspace $ \mathbb{C} e^0 $ is orthogonal to the oscillating eigenspace $ \mathbb{C}e^- \oplus \mathbb{C}e^+ $, whence in particular it is always true that $ V_{\QG}\perp V_{\osc} $.
\\

\noindent
In the following we shall denote as $e^a(n)$ the eigenvector of $\mathbb{P}_n\mathcal{A}$ associated with the eigenvalue $i \ \omega^a$, i.e.  
$$e^{\tau \mathbb{P}_n\mathcal{A}} \ \left( e^{in\cdot x}  e^a(n) \right)=\exp\left\{in \cdot x+ \ \tau \omega^a(n) \right\}e^a(n).$$
Let us define $U^\varepsilon=\Lminus V^\varepsilon$, we want to reformulate \eqref{primitive equations} in terms of the new unknown $ U^\varepsilon $. A straightforward computation shows that the vector field $ U^\varepsilon $ satisfies the following equation
\begin{equation}
\label{filtered systemP}\tag{FS$_\varepsilon$}
\left\lbrace
\begin{array}{l}
\partial_t U^\varepsilon+\mathcal{Q}^\varepsilon \left(U^\varepsilon,U^\varepsilon\right)-
\mathbb{D}^\varepsilon U^\varepsilon= 0\\
\dive v^\varepsilon =0\\
\bigl. U^\varepsilon\bigr|_{t=0}=V_0
\end{array}
\right.
\end{equation}
where
\begin{align}
\mathcal{Q}^\varepsilon \left(U^\varepsilon,U^\varepsilon\right)=& \Lminus \mathbb{P} \left[ \left( \Lplus U^\varepsilon \cdot \nabla\right) \Lplus U^\varepsilon\right]\label{Qeps}\\
\mathbb{D}^\varepsilon U^\varepsilon=& \Lminus\mathbf{D}\Lplus U^\varepsilon.\label{Deps}
\end{align}
We shall call the system \eqref{filtered systemP} the filtered system.\\

Before using the above results to find the limit of \eqref{filtered systemP} we introduce the "potential vorticity"
\begin{equation}
\label{potential vorticity}
\Omega^\varepsilon = -\partial_2  U^{1,\varepsilon}+ \partial_1 U^{2,\varepsilon}-F\partial_3U^{4,\varepsilon}.
\end{equation}

The potential vorticity has been introduced by J.-Y. Chemin in \cite{chemin_prob_antisym} and it is now a well-known tool in the study of primitive equation (see \cite{charve_primitive},\cite{charve_ngo_primitive},\cite{gallagher_schochet},\cite{iftimie_primitive}). The diagonalization explained in \eqref{eq:QG-osc} can as well be obtained by writing $U^\varepsilon=U^\varepsilon_{\QG}+ U_{\osc}^\varepsilon$, with
\begin{equation}
U^\varepsilon_{\QG}=\left(
\begin{array}{cccc}
-\partial_2\Delta^{-1}_F\Omega^\varepsilon, & \partial_1 \Delta^{-1}_F\Omega^\varepsilon, & 0, & -\partial_3 F\Delta^{-1}_F\Omega^\varepsilon
\end{array}\right),
\end{equation}
where $\Delta_F^{-1}$ denotes the Fourier multiplier
$$
-\Delta^{-1}_F u =\mathcal{F}^{-1}\left( \left(  \frac{1}{\check{n}_1^2+\check{n}_2^2+F^2\check{n}_3^2} \  \hat{u}_n \right)_n \right).
$$
We remark the fact that since $U^\varepsilon_{\QG}$ belongs to the kernel of $\mathbb{P}\mathcal{A}$  we obtain indeed that $U^\varepsilon_{\QG}= V^\varepsilon_{\QG}$.\\

One of the major problem is to understand exactly which is the limit for $\varepsilon\to 0$ of the forms $\mathcal{Q}^\varepsilon, \mathbb{D}^\varepsilon$ and, if possible, how to give a closed formulation for it. 
 To do so we use the explicit formulation of $\mathcal{Q}^\varepsilon, \mathbb{D}^\varepsilon$ given in equation \eqref{Qeps} and \eqref{Deps}. Let us decompose divergence-free vector field $U$  as:
$$
\mathcal{F} U (n)= \sum_{a\in \{-,0,+\}} U^a \left( n \right)= \sum_{a\in \{-,0,+\}} \left( \left. \mathcal{F}U \left( n \right) \right| e^a \left( n \right) \right)_{\mathbb{C}^4} e^a(n),
$$
and after some  computations we  obtain that;
\begin{equation}
\label{forma bilin epsilon non zero}
\mathcal{F}\left( \mathcal{Q}^\varepsilon \left( U, V \right)\right)(n)
= \sum_{a,b,c\in \{-,0,+\}} e^{-i\frac{t}{\varepsilon}\left( \omega^a(k)+\omega^b(n-k)-\omega^c(n)\right)} 
 \left( \left.  \sum_{j=1,2,3}\left( n_j-k_j\right)  U^{a,j}(k) V^{b}(n-k)\right| e^c(n)\right)_{\mathbb{C}^4} e^c(n) .
\end{equation}
In the following we will write $\omega^{a,b,c}_{k,n-k,n}=\omega^a(k)+\omega^b(n-k)-\omega^c(n)$ for the sake of conciseness, as well as $\omega^{a,b}_n=\omega^a(n)+\omega^b(n)$. With $ U^{a,j} $ we denote the $ j $-th component of the vector $ U^a = \left(\left. \hat{U} \right| e^a  \right)_{\C^4} \ e^a $ for $ a=0, \pm $.\\
Similar calculations give us that
\begin{equation}
\label{laplaciano epsilon non zero}
\mathbb{D}^\varepsilon U= \mathcal{F}^{-1}\left( \sum_{a,b\in \{-,0,+\}} e^{-i\frac{t}{\varepsilon}\omega^{a,b}_n} \left( \left. \mathbf{D} (n) U^{b}(n)\right| e^a(n)\right)_{\mathbb{C}^4} e^a(n)\right),
\end{equation}
where $ \mathbf{D} (n) $ is the Fourier symbol associated to the second-order differential operator $ \mathbf{D} $, see \eqref{matrici}.\\
Letting $\varepsilon\to 0$ we only have to use the non stationary phase theorem ( see, for instance \cite{AlinhacGerard}, \cite{bahouri_chemin_danchin_book}, \cite{Stein93}) to obtain that, if $U,V$ are smooth functions;
\begin{align}
\mathcal{Q}\left( U,V\right)= & \; \mathcal{F}^{-1}\left(
\mathbb{P}_n
 \sum_{\omega^{a,b,c}_{k,n-k,n}=0} \left(\left. \sum_{j=1}^3 \left( n_j-k_j \right) U^{a,j}(k) V^{b}(n-k)\right| e^c(n)\right)_{\mathbb{C}^4} e^c(n)\right)\label{limit quadratic},\\
\mathbb{D}U= &\; \mathcal{F}^{-1}\left( \sum_{\omega_n^{a,b}=0} \left( \left.  \mathbf{D} (n) U^{b}(n) \right| e^a(n)\right)_{\mathbb{C}^4} e^a(n)\right).\label{limit linear}
\end{align}
Here we implicitly define as $ \mathbf{D} \left( n \right) $ the Fourier symbol assciated to the matrix $ \bf D $ deifned in \eqref{matrici}.\\

\subsection{The global splitting of the limit bilinear form $ \mathcal{Q} $.}\label{sec:global_splitting_bilinear}

This section is aimed to explain how the bilinear interaction $ \mathcal{Q} $ defined in \eqref{limit quadratic} behaves along non-oscillating and oscillating subspaces $ \C e^0 $ and $ \C e^-\oplus \C e^+ $. Such kind of result is very well known in the theory of singular perturbation problems in periodic domains, and the results that we present here have been already proved by several authors in \cite{BMN-P}, \cite{embid_majda} and \cite{embid_majda2}, for this reason we will not prove them but instead we will refer to the works mentioned and references therein.\\
The results presented in the present section derive from the geometrical properties of vector decomposed as in \eqref{eq:QG-osc} and from the localization in the frequency space of the limit bilinear form $ \mathcal{Q} $, localization which reads as
\begin{equation*}
\set{ \left( k,n \right)\in \mathbb{Z}^6 \  \Big| \ \omega^a\left( k \right) + \omega^b \left( n-k \right)= \omega^c \left( n \right), \ a,b,c \in \set{ 0, \pm }\   },
\end{equation*}
where the eigenvalues are defined in \eqref{eigenvalues}.\\

\noindent In this section we will always consider \textit{smooth} vector fields, in particular given a smooth vector field $ W $ we define
\begin{align*}
\Omega_W = & \ -\partial_2 W^2 +\partial_1 W^2 -F \partial_3 W^4,&
W_{\QG} = & \ \left( 
\begin{array}{c}
-\partial_2\\
\partial_1\\
0\\
-F \partial_3
\end{array}
 \right) \Delta_{F}^{-1} \Omega_{W}
 =  \ \left( w_{\QG}, W^4_{\QG} \right), &
 W_{\osc} = & \ W-W_{\QG}.
\end{align*}
Obviously $ W_{\QG} $ and $ W_{\osc} $ are respectively the projections of $ W $ onto the non-oscillating and oscillating subspaces defined in \eqref{eq:QG-osc}.\\

\begin{lemma}\label{lem:splitting_bilinear_QG}
The following identity holds true
\begin{equation*}
\mathcal{F}^{-1} \left( \left(\left. \mathcal{F \ Q} \left( W, W \right) \right| \left| n \right|_F e^0  \right)_{\C^4} \right) = w_{\QG}\cdot \nabla \Omega_{W},
\end{equation*}
where $ \mathcal{Q} $ is defined in \eqref{limit quadratic} and $ e^0 $ is the non-oscillating eigenvector defined in \eqref{cP:eigenvectors}.
\end{lemma}
\begin{cor}\label{cor:splitting_bilinear_QG}
The following identity holds true
\begin{equation*}
\mathcal{F}^{-1} \left( \left(\left. \mathcal{F \ Q} \left( W, W \right) \right|  e^0  \right)_{\C^4} \right) = \left( 
\begin{array}{c}
-\partial_2\\
\partial_1\\
0\\
-F \partial_3
\end{array}
 \right) \Delta_{F}^{-1} \left( w_{\QG}\cdot \nabla \Omega_{W} \right).
\end{equation*}
\end{cor}

\noindent
For a proof of Lemma \ref{lem:splitting_bilinear_QG} we refer the reader to \cite{embid_majda} and \cite{embid_majda2}. What has to be retained is the facts that the projection of $  \mathcal{ Q} \left( W, W \right) $ onto the potential non-oscillating subspace does not presents interactions of the oscillating part of the vector field.\\

\begin{lemma}\label{lem:splitting_bilinear_osc}
Let $ W $ be a smooth vector field, then the following identity holds true
\begin{equation*}
\left( \mathcal{Q} \left( W_{\QG}, W_{\QG} \right) \right)_{\osc}=0.
\end{equation*}
\end{lemma}
\begin{proof}
 Considering the explicit formulation of the limit bilinear form $ \mathcal{Q} $ we deduce 
 \begin{equation}\label{boh6}
 \left( \mathcal{Q} \left( W_{\QG}, W_{\QG} \right) \right)_{\osc}
  = \mathcal{F}^{-1} \left( 
\sum_{\substack{ k+m=n \\ \omega^{0,0,\pm}_{k,m,n}=0 }} \left(\left. n\cdot  \left( W^{0} \left( k \right) \otimes \ W^{0}\left( m \right)  \right)\right| e^\pm \left( n \right)  \right) _{\C^4} \ e^\pm \left( n \right)
  \right).
 \end{equation}
 Let us consider hence the equation $ \omega^{0,0,\pm}_{k,m,n}=0 $, thanks to the explicit expression of the eigenvalues in \eqref{eigenvalues} then it is equivalent to the equation
$$
\left| n_h \right|^2 + {F^2}{n_3^2}=0,
$$
which is true only if $ n=0 $, and in this case the contributions arising in \eqref{boh6} are null, concluding. 
\end{proof}

\begin{cor}\label{cor:splitting_bilinear_osc}
The projection of the limit bilinear form $ \mathcal{Q} $ onto the oscillating subspace can be written as
\begin{equation*}
\left( \mathcal{Q} \left( W, W \right) \right)_{\osc}
\\
= \left( \mathcal{Q} \left( W_{\QG}, W_{\osc} \right) \right)_{\osc} + \left( \mathcal{Q} \left( W_{\osc}, W_{\QG} \right) \right)_{\osc}
+ \left( \mathcal{Q} \left( W_{\osc}, W_{\osc} \right) \right)_{\osc},
\end{equation*}
thanks to the decomposition \eqref{eq:QG-osc}.
\end{cor}

\section{Proof of Theorem \ref{thm:Leray_sol}.}

\label{existence limit system}

\begin{rem}
As the reader may have noted Theorem \ref{thm:Leray_sol} states the existence of  \textit{\`a la Leray}-type solutions. This can seem to be unexpected since, generally, Leray solutions are constructed thanks to compactness methods. In system \eqref{lim syst} we cannot apply any compactness method since we do not have any second-order vertical derivative $ \partial_3^2 $ and $ L^2 $ is not compactly embedded in $ H^{1, 0} $. Nonetheless the bilinear form $ \mathcal{Q} $ has better product rules than the standard bilinear form in the \NS\ equations, this will allow us to make sense (distributionally) of the term $ \mathcal{Q}\left( U,U \right) $. Moreover we require the initial potential vorticity $ \Omega_0 $ to be $ \cPLtwo $, which, roughly speaking, is "almost as" requiring the initial velocity field to be $ H^1 $. 
\end{rem}

\textit{Proof of Theorem \ref{thm:Leray_sol}} Before starting the proof we point out the following fact, \NS\ equations preserve the global average of the unknown function. This happens as well for the system \eqref{primitive equations}, whence we can consider data with zero horizontal average. Thanks to this property homogeneous and non-homogeneous Sobolev spaces are equivalent, we shall use this  constantly in the present proof. In particular they will be always non-homogeneous. This fact concerns \textit{only} the isotropic spaces $ H^s \left( \R^3 \right) $.\\
The proof is standard application of Galerkin's approximation. We define the truncation operator 
$$
J_n u =  \sum_{\left\lbrace \left. k\in \mathbb{Z}^3 \right| \left| k\right| \leqslant n\right\rbrace } \hat{u}_n e^{i \check{k}\cdot x},
$$
and consequently the approximated system
\begin{equation}
\label{approx lim syst}
\left\lbrace
\begin{array}{l}
\partial_t U_n+ J_n\mathcal{Q}\left(U_n,U_n\right)+\mathbb{D}U_n= 0\\
\dive u_n=0\\
\bigl. U_n \bigr|_{t=0}=J_n U_0,
\end{array}
\right.
\end{equation}
in the unknown $U_n$. 
We recall that for a fixed $n$, $J_n$ maps continuously any $ H^k $ space to any $ H^{k+h} $ space for $ h\geqslant0 $ thanks to Bernstein inequality. Thus \eqref{approx lim syst} is a differential equation in the space
$$
L^2_n\left( \mathbb{T}^3\right)= \left\lbrace \left. u\in \cPLtwo \right| \hat{u}_k=0 \text{ if } \left|k\right|>n\right\rbrace.
$$
Since the support of the Fourier transform of $U_n \in L^2_n\left( \mathbb{T}^3\right)$ is included in the ball of center 0 and radius $n$ and the support of $\mathcal{F} \left( U_n \otimes U_n\right)$ is included in $B_{2n}(0)$ we obtain easily that $J_n \mathcal{Q}\in \mathcal{C}\left(L^2_n\left( \mathbb{T}^3\right)\times L^2_n\left( \mathbb{T}^3\right); L^2_n\left( \mathbb{T}^3\right) \right)$. Hence Cauchy-Lipschitz theorem gives the existence of a unique solution to \eqref{approx lim syst} on a maximal interval of time $\left[ 0, T_n\right)$ taking values in $L^2_n\left( \mathbb{T}^3\right)$.\\
Moreover since 
$$
\mathcal{Q}\left(A,B\right)=\lim_{\varepsilon\to 0} \Lminus\mathbb{P} \left[ \left( \Lplus A  \cdot \nabla\right) \Lplus B \right],
$$
it is clear that $\left( \left. J_n\mathcal{Q}\left(U_n,U_n\right) \right| U_n \right)_\cPLtwo=\left( \left. \mathcal{Q}\left(U_n,U_n\right) \right| U_n \right)_\cPLtwo=0$ since $\dive u_n=0$. Hence by a standard energy estimate on the parabolic-hyperbolic equation \eqref{approx lim syst} we get
$$
\frac{1}{2}\left\| U_n(t)\right\|_\cPLtwo^2 +c \int_0^t \left\| \nh U_n(s)\right\|_\cPLtwo^2 \d s \leqslant \frac{1}{2} \left\| U_0 \right\|_\cPLtwo^2,
$$
from which for all $t \in \left[ 0, T_n \right)$ we have $\left\| U_n(t)\right\|_\cPLtwo^2\leqslant \left\| J_n U_0 \right\|_\cPLtwo^2\leqslant \left\|  U_0 \right\|_\cPLtwo^2$. We deduce that $T_n=\infty$ and for all $t>0$ $U_n(t)$ satisfies
$$
\left\| U_n(t)\right\|_\cPLtwo^2 +2c \int_0^t \left\|\nh U_n(s)\right\|_\cPLtwo^2 \d s \leqslant \left\| U_0\right\|_2^2.
$$
Considered the relation $\left\|U_n\right\|_{L^2\left(\left(0,t\right);\cPLtwo\right)}\leqslant \sqrt{t} \left\|U_n\right\|_{L^\infty\left(\left(0,t\right);\cPLtwo\right)}\leqslant \sqrt{t} \left\|U_0 \right\|_\cPLtwo$ we can say that the sequence $U_n$ is bounded in $L^\infty \left( \mathbb{R}_+; \cPLtwo\right)\cap L^2_{\text{loc}}\left( \mathbb{R}_+; H^{1,0}\right)$. By the  structure of \eqref{approx lim syst} we obtain easily that $\partial_tU_n$ is bounded in $L^2_\text{loc}\left(\mathbb{R}_+; H^{-N}\right)$ for $N$ sufficiently big (the proof of such fact is identical as the proof of Proposition \ref{prop:topological_convergence}), hence $\left( \partial_t U_n\right)_n$ is a sequence of uniformly bounded functions in $L^2_\text{loc}\left(\mathbb{R}_+; H^{-N}\right)$. We can infer via Aubin-Lions lemma \cite{Aubin63} obtaining that $U_n\to U$ in $L^2_\text{loc}\left( \mathbb{R}_+; H^{-\varepsilon}\left( \mathbb{T}^3\right) \right)$ where $ \varepsilon\in \left(0, N \right) $ up to (non-relabeled) subsequences.\\

\noindent
Since the sequence $ \left( U_n \right)_n $ converges in $ L^2_\text{loc}\left( \mathbb{R}_+; H^{-\varepsilon}\left( \mathbb{T}^3\right) \right) $ only, and products of $ H^{-\varepsilon} $ functions are, a priori, not well defined we introduce a diagonalization method which allows us to split \eqref{approx lim syst} in two systems which we will be able to handle.\\
\noindent
We rely on a diagonalization method introduced introduced by P. Embid and A. Majda in \cite{embid_majda}, in detail, we define 
\begin{align}
\Omega_n = & -\partial_2 U_n^1 + \partial_1 U_n^2 -F\partial_3 U_n^4, \label{def Omegan}\\
V_{{\QG},n}=U_{{\QG},n}= & \left(
\begin{array}{c}
\nhp\\
0\\
-F\partial_3 
\end{array}\right)\Delta^{-1}_F \Omega_n ,\label{def VQGn}\\
U_{{\osc},n}= & U_n-U_{{\QG},n}.\nonumber
\end{align}

\noindent
Applying Lemma \ref{lem:splitting_bilinear_QG} on the smooth vector field $ U_n $ we deduce that
\begin{equation*}
\left(\left. \mathcal{F} J_n\mathcal{Q}\left(U_n,U_n\right) \right| \left| n \right|_F \ e^0 \left( n \right)  \right)_{\C^4}= \mathcal{F} \left( J_n\left( v_{{\QG},n}^h\cdot\nh \Omega_n\right) \right).
\end{equation*}

\noindent
Whence the projection of the element $ J_n\mathcal{Q}\left(U_n,U_n\right) $ onto the potential space defined by the potential vorticity is the \qg\ transport $ J_n\left( v_{{\QG},n}^h\cdot\nh \Omega_n\right) $. The proof of such result is omitted in the present work, but it relies on a careful analysis of the cancellation properties induced by the limit bilinear form $ \mathcal{Q} $.\\
Applying Corollary \ref{cor:splitting_bilinear_osc} we deduce:
\begin{equation*}
\left( J_n\mathcal{Q}\left(U_n,U_n\right) \right)_{\osc}
 = \left( J_n\mathcal{Q}\left(V_{{\QG},n},U_{{\osc},n}\right) \right)_{\osc}+
\left( J_n\mathcal{Q}\left(U_{{\osc},n},V_{{\QG},n} \right) \right)_{\osc}
+ \left( J_n\mathcal{Q}\left(U_{{\osc},n},U_{{\osc},n} \right) \right)_{\osc}.
\end{equation*}

\noindent
Projecting hence \eqref{approx lim syst} onto the oscillating subspace and the potential nonoscillating subspace we obtain the following global splitting for the first equation of  \eqref{approx lim syst}:
\begin{equation}
\label{eq:splitting_approx_system}
\begin{aligned}
\partial_t\Omega_n+J_n\left( v_{{\QG},n}^h\cdot\nh \Omega_n\right)+ a_{\QG}\left(D_h \right)\Omega_n= 0, \\
\begin{multlined}
\partial_t U_{{\osc}, n}+\left( J_n\mathcal{Q}\left(V_{{\QG},n},U_{{\osc},n}\right) \right)_{\osc}+
\left( J_n\mathcal{Q}\left(U_{{\osc},n},V_{{\QG},n} \right) \right)_{\osc}
\\
+ \left( J_n\mathcal{Q}\left(U_{{\osc},n},U_{{\osc},n} \right) \right)_{\osc}
 +a_{\osc}\left(D_h\right)U_{{\osc},n}=0.
 \end{multlined}
\end{aligned}
\end{equation}
The operators $ a_{\QG} $ and $ a_{\osc} $ are nothing but the projection of the operator $ -\mathbb{D} $ onto the potential space defined by $ \Omega $ and the oscillating subspace. We avoid to give a detailed description of such operators now (see Section \ref{sec:weak_conv}), what has to be retained is that they are symbols such that there exists a positive constant $ c $ such that $ \left| a_{\QG} \left( \xi \right) \right|, \left| a_{\osc} \left( \xi \right) \right| \geq c \left| \xi_h \right|^2 $. \\
On the splitting \eqref{eq:splitting_approx_system} we can apply the same procedure as above  to obtain that $\Omega_n\to \Omega$ in $L^2_\text{loc}\left( \mathbb{R}_+; H^{-\varepsilon}\left( \mathbb{T}^3\right) \right)$, and defining $\definizioneVQG$ for $\Omega$ the limit of the sequence $\left( \Omega_n\right)_n$, and since 
$$\left(
\begin{array}{c}
\nhp\\
0\\
-F\partial_3 
\end{array}\right)\Delta^{-1}_F \in   \mathcal{L}\left( {H}^\alpha, {H}^{\alpha +1}\right), \  \alpha\in\mathbb{R}$$
 we obtain as well that 
$$
V_{{\QG},n}\to V_{\QG} \text{ in } L^2_\text{loc}\left( \mathbb{R}_+; {H}^{1-\varepsilon}\right),
$$
and $\left( V_{{\QG},n} \right)_n $ uniformly (in $ n $) bounded in $L^\infty \left( \mathbb{R}_+, H^1 \right)$ .\\
Combining the definitions \eqref{def Omegan} and \eqref{def VQGn} we can hence rewrite $V_{{\QG},n}$ as 
$$
V_{{\QG},n}= \left(
\begin{array}{c}
-\partial_2 \\
\partial_1 \\
0\\
-F \partial_3 
\end{array}\right)\Delta^{-1} _F
\left(
\begin{array}{cccc}
-\partial_2 , &
\partial_1 ,& 0, -F\partial_3
\end{array}
\right)\cdot U_n
= \Pi_{\QG} U_n,
$$
with $\Pi_{\QG}$ Fourier multiplier of order zero, hence $\Pi_{\QG}\in\mathcal{L}\left( H^\alpha \left( \mathbb{T}^3\right)\right)$ for each $\alpha\in\mathbb{R}$. This implies in particular that, defining $U_{\osc} = U-V_{\QG}$
\begin{align*}
\left\| U_{{\osc},n}-U_{\osc} \right\|_{H^{-\varepsilon}}= & \left\| \left( U_n -V_{{\QG},n}\right)-\left( U -V_{{\QG}}\right)\right\|\\
= & \left\| \left(1-\Pi_{\QG} \right) \left( U_n-U\right) \right\|_{H^{-\varepsilon}}\\
\leqslant & C \left\| U_n-U\right\|_{H^{-\varepsilon}}.
\end{align*}
This implies in particular that $  U_{{\osc},n}\to U_{\osc} $ in $L^2_\text{loc}\left( \mathbb{R}_+; H^{-\varepsilon}\left( \mathbb{T}^3\right) \right)$. The same idea can be applied to show that $ \left( U_{{\osc},n} \right)_n $ is bounded in $L^\infty \left( \mathbb{R}_+; \cPLtwo\right)$ and $ \left( \nh U_{{\osc}, n} \right)_n $ is bounded in $ L^2\left( \mathbb{R}_+; \cPLtwo\right)$.

At this point we can project $\mathcal{Q}\left( U_n, U_n \right)$ on the spaces $\mathbb{C} e^0 , \mathbb{C} e^- \oplus \mathbb{C} e^+  $ (see \eqref{cP:eigenvectors}) obtaining, thanks to the results of Corollary \ref{cor:splitting_bilinear_QG} and \ref{cor:splitting_bilinear_osc}:
\begin{equation*}
\begin{aligned}
\mathcal{Q}\left( U_n, U_n \right)= & \ \mathcal{Q}\left( U_n, U_n \right)_{\QG}+
\mathcal{Q}\left( U_n, U_n \right)_{\osc}\\
= & \
\left( -\partial_2, \partial_1, 0, -F\partial_3 \right)^\intercal
\Delta_F^{-1}\left(  v_{{\QG},n}^h\cdot \nh \Omega_n \right)
\\
& + \left( \mathcal{Q}\left( V_{{\QG},n}, U_{{\osc},n}\right) \right)_{\osc}
+\left(  \mathcal{Q}\left( U_{{\osc},n}, V_{{\QG},n}\right) \right)_{\osc}
+ \left( \mathcal{Q}\left( U_{{\osc},n}, U_{{\osc},n}\right) \right)_{\osc}.
\end{aligned}
\end{equation*}
It is matter of standard energy bounds with classical product rules in Sobolev spaces to prove that
\begin{align*}
\left( -\partial_2, \partial_1, 0, -F\partial_3 \right)^\intercal
\Delta_F^{-1}\left(  v_{{\QG},n}^h\cdot \nh \Omega_n \right) \to &
\;
\left( -\partial_2, \partial_1, 0, -F\partial_3 \right)^\intercal
\Delta_F^{-1}\left(  v_{{\QG}}^h\cdot \nh \Omega \right),\\
\left( \mathcal{Q}\left( V_{{\QG},n}, U_{{\osc},n}\right) \right)_{\osc} \to & \;
\left( \mathcal{Q}\left( V_{{\QG}}, U_{{\osc}}\right) \right)_{\osc},\\
\left( \mathcal{Q}\left( U_{{\osc},n}, V_{{\QG},n}\right) \right)_{\osc} \to & 
\;
\left( \mathcal{Q}\left( U_{{\osc}}, V_{{\QG}}\right) \right)_{\osc},
\end{align*}
in the sense of distributions as $ n\to \infty $. 
The limit of the product of terms of the form $ U_{{\osc}, n} $ is, in general, not well defined. Indeed system \eqref{Slim} lacks of vertical dissipation, hence the best we know is that $ U_{{\osc}, n}\to U_{\osc} $ in $ L^2_{\loc} \left( \R_+; H^{-\varepsilon} \right) $, but generally a product between $ H^{-\varepsilon} $ elements is not well-defined. Is in this context in fact that we shall use the improved regularity in the product which is characteristic of the bilinear form $ \mathcal{Q} $. We claim that 
\begin{equation}
\label{eq:conv_bilin_osc_existence}
\left( \mathcal{Q}\left( U_{{\osc}, n},U_{{\osc}, n} \right) \right)_{\osc}
 \xrightarrow[n\to\infty]{ \mathcal{D}' \left( \R_+\times\T^3 \right)} 
 \left( \mathcal{Q}\left( U_{{\osc}},U_{{\osc}} \right) \right)_{\osc},
\end{equation}
 The proof of \eqref{eq:conv_bilin_osc_existence} is postponed.
 Whence we finally proved that $\mathcal{Q}\left( U_n, U_n \right) \to \mathcal{Q}\left( U, U \right)$ in $\mathcal{D}' \left(\R_+\times\T^3\ \right)$, concluding.
\hfill$ \Box $

\subsection{Proof of \eqref{eq:conv_bilin_osc_existence}.}\label{sec:conv_bilin_osc_existence}
As we already stated M. Paicu in \cite{paicu_rotating_fluids} proved a similar result. We shall prove \eqref{eq:conv_bilin_osc_existence} using different techniques.\\
Defining $ \mathcal{Q}\left( A,B \right)= \dive \widetilde{\mathcal{Q}}\left( A,B \right) $, i.e.
\begin{align*}
\widetilde{\mathcal{Q}} \left( A, B \right) = & \sum_{\mathcal{K}} \left(\left. \hat{A}^a \left( k \right) \hat{B}^b\left( m \right) \right| e^c\left( n \right)  \right)_{\mathbb{C}^4}e^c\left( n \right),\\
= &  \sum_{\mathcal{K}}  \hat{A}^a\left( k \right) \hat{B}^{b,c}\left( m,n \right)
\end{align*}
where $ \hat{A}^a \left( k \right)= \left( \left. \hat{A}\left( k \right) \right| e^a \left( k \right) \right) e^a\left( k \right) $, $ \hat{B}^{b,c}\left( m,n \right)= \left( \left. \hat{B}^b \left( m \right) \right| e^c \left( n \right) \right) e^c \left( n \right) $. It suffice in fact to prove that
$$
\widetilde{\mathcal{Q}}\left( U_{{\osc}, j}-U_{\osc}, U_{{\osc}, j}+ U_{\osc} \right)\to 0,
$$
in $ \mathcal{D}' \left( \R_+\times\T^3 \right) $ as $ j\to \infty $ to conclude. To do so we consider a $ \phi\in   \mathcal{D} $ and, by Plancherel theorem
\begin{multline}
\label{eq:Leraysolbilinoscprimaeq}
\int_{\R_+\times\T^3} \phi \left( t, x \right) \widetilde{\mathcal{Q}} \left( U_{{\osc}, j}-U_{\osc}, U_{{\osc}, j}+ U_{\osc} \right) \left( t ,  x \right) \dx \d t
\\
\begin{aligned}
 = & \ 
 \int_{\R_+} \sum_{n\in\mathbb{Z}^3} \sum_{\mathcal{K}^\star_n} \hat{\phi}\left(t,  n \right)  \widehat{\left( U_{{\osc}, j}-U_{\osc} \right)}^a\left(t,  k \right)  \widehat{\left( U_{{\osc}, j}+ U_{\osc} \right)}^{b,c} \left(t,  m,n \right) \d t \\
= & \ \int_{\R_+}  \sum_{n,k_h, m_h} \hat{\phi}_n \left( t \right) \sum_{\left\{ k_3 : \left( k, \left( m_h, n_3-k_3 \right), n \right)\in \mathcal{K}^\star \right\}} \widehat{\left( U_{{\osc}, j}-U_{\osc} \right)}^a\left(t,  k \right)  \widehat{\left( U_{{\osc}, j}+ U_{\osc} \right)}^{b,c}\left(t,  m_h, n_3-k_3,n \right)\d t.
\end{aligned}
\end{multline}
We make a couple of remarks in order to simplify the notation. Since we considered the eigenvectors as normalized all along the paper the following relations are easy to deduce
\begin{align*}
&\left| \hat{U}^{b,c} \left( m,n \right) \right|\lesssim \left| \hat{U}^{b} \left( m \right) \right|\lesssim \left| \hat{U} \left( m\right) \right|.
\end{align*}
Hence from now on the terms  $ \widehat{\left( U_{{\osc}, j}-U_{\osc} \right)}^a\left(t,  k \right) $ and $ \widehat{\left( U_{{\osc}, j}+ U_{\osc} \right)}^{b,c}\left(t,  m_h, n_3-k_3,n \right) $ shall be substituted respectively to $ \widehat{\left( U_{{\osc}, j}-U_{\osc} \right)}_k $ and $ \widehat{\left( U_{{\osc}, j}+ U_{\osc} \right)}_{\left( m_h, n_3-k_3 \right)} $. Here we chose to make implicit the dependence on the variable $ t $.  We want to stress out the fact that this choice is made only to simplify the notation.
Indeed we have that
\begin{equation}
\widehat{\left( U_{{\osc}, j}-U_{\osc} \right)}_k  \widehat{\left( U_{{\osc}, j}+ U_{\osc} \right)}_{\left( m_h, n_3-k_3 \right)}
=
 k^{-\varepsilon/2}\widehat{\left( U_{{\osc}, j}-U_{\osc} \right)}_k k^{\varepsilon/2} \widehat{\left( U_{{\osc}, j}+ U_{\osc} \right)}_{\left( m_h, n_3-k_3 \right)}.\label{eq:Leraysolbilinoscsecondaeq}
\end{equation}
The set $ \left\{ k_3 : \left( n,k \right)\in \mathcal{K}^\star \right\} $ is indeed finite and, in particular, it is composed by the $ k_3 $ which satisfy the following equation
\begin{multline*}
\left(F^2\left(k_3\right){}^2+\left(k_h\right){}^2\right){}^{1/2}\left( \left(m_h\right){}^2+ \left(n_3-k_3\right){}^2\right){}^{1/2}\\
= \left( \left(k_h\right){}^2+ \left(k_3\right){}^2\right){}^{1/2}\left( \left(m_h\right){}^2+ \left(n_3-k_3\right){}^2\right){}^{1/2} 
 - \left( \left(k_h\right){}^2+ \left(k_3\right){}^2\right){}^{1/2}\left(F^2\left(n_3-k_3\right){}^2+\left(m_h\right){}^2\right){}^{1/2}.
\end{multline*}
Expanding the above equation and collecting term by term in the powers of $ k_3 $ give us the following polynomial equation
$$
\wp \left( k_3 \right)=\sum_{i=0}^8 A_i \left( k_h, m_h, n \right) k_3^i=0,
$$
where the $ A_i $ take the following form
\begin{align*}
A_8= & \left(1-4 F^2\right)\\
 A_7 = & 4 (-1 + 4 F^2)  n_3\\
A_6 = & -6  \left(F^2 k_h^2+F^2 m_h^2+\left(-1+4 F^2\right) n_3^2\right)\\
A_5 = & 4  n_3 \left(6 F^2 k_h^2+3 F^2 m_h^2+\left(-1+4 F^2\right) n_3^2\right)\\
A_4 = &
- \left(F^2 \left(-4+F^2\right) k_h^4+F^2 \left(-4+F^2\right) m_h^4 \right.
 \left.-6 F^2 m_h^2 n_3^2+\left(1-4 F^2\right) n_3^4-2 k_h^2 \left(\left(3+2 F^2+F^4\right) m_h^2+18 F^2 n_3^2\right)\right)
\end{align*}
\begin{align*}
 A_3 = & 4  k_h^2 n_3 \left(-F^2 \left(-4+F^2\right) k_h^2+\left(3+2 F^2+F^4\right) m_h^2+6 F^2 n_3^2\right)\\
 A_2 = & -2  k_h^2 \left(\left(2+F^2\right) m_h^4+\left(3+2 F^2+F^4\right) m_h^2 n_3^2\right.
 \left.+3 F^2 n_3^4+k_h^2 \left(\left(2+F^2\right) m_h^2-3 F^2 \left(-4+F^2\right) n_3^2\right)\right)\\
 A_1 = & 4  k_h^4 n_3 \left(\left(2+F^2\right) m_h^2-F^2 \left(-4+F^2\right) n_3^2\right)\\
 A_0= &-k_h^4 \left(3 m_h^4+2 \left(2+F^2\right) m_h^2 n_3^2-F^2 \left(-4+F^2\right) n_3^4\right).
\end{align*}
Although we have been giving the explicit expression of the $ A_i $'s we outline the fact that the explicit expression by itself is irrelevant, the only thing that matters is that the $ A_i $'s are  polynomials in the variables $ k_h,m_h,n $. We can hence apply the following result
 which bounds the modulus of a root of a complex root of a polynomial in terms of its coefficients, the following proposition is known as Fujiwara near-optimal bound.
\begin{prop}\label{Fujiwara_bound}
Let $ P(z)= \sum_{k=0}^n a_n z^k$ a polynomial $P\in \mathbb{C}\left[ z\right]$, let $\zeta$ be one of the $n$ complex roots of $P$, then
\begin{align*}
\left| \zeta \right| \leqslant 2 \max \left\lbrace
\left| \frac{a_{n-1}}{a_n}\right|,\left| \frac{a_{n-2}}{a_n}\right|^{1/2},\ldots ,\left| \frac{a_{1}}{a_n}\right|^{{1}/{\left( n-1 \right)}},\left| \frac{a_{0}}{a_n}\right|^{1/n}
\right\rbrace.
\end{align*}
\end{prop}
We shall omit to prove Proposition \ref{Fujiwara_bound} and refer the reader to the work \cite{Fujiwara_bound} instead.\\
Proposition \ref{Fujiwara_bound} applied on $ \wp \left( k_3 \right) $ tells us that
$$
\left| k_3 \right| \lesssim \left| n \right|^{\alpha_1} \left| m_h \right|^{\alpha_2} \left| k_h \right|^{\alpha_3},
$$
where $ k_3 $ is any root of $ \wp $, hence
$$
\left| k \right|^{\varepsilon/2}\lesssim \left| k_h \right|^{\varepsilon/2} + \left( \left| n \right|^{\alpha_1} \left| m_h \right|^{\alpha_2} \left| k_h \right|^{\alpha_3} \right)^{\varepsilon/2} .
$$
by concavity on the function $ h_\varepsilon \left( x \right) = x^{\varepsilon/2} $,
with $ \alpha_1+\alpha_2+\alpha_3 < N $ for some large and finite $ N $. Coming back to \eqref{eq:Leraysolbilinoscprimaeq} and \eqref{eq:Leraysolbilinoscsecondaeq} this means that
\begin{multline*}
\left|\int_{\R_+}  \int_{\T^3} \phi \left( x \right) \widetilde{\mathcal{Q}} \left( U_{{\osc}, j}-U_{\osc}, U_{{\osc}, j}+ U_{\osc} \right) \left( x \right) \dx \d t  \right|\\
\begin{aligned}
\lesssim  & \ 
\int_{\R_+} \sum_{n,k_h, m_h}\left| \hat{\phi}_n \right| \sum_{\left\{ k_3 : \left( k, \left( m_h, n_3-k_3 \right), n \right)\in \mathcal{K}^\star \right\}}
\left| k_h \right|^{\varepsilon/2}
\left| k\right|^{-\varepsilon/2}  \left|\widehat{\left( U_{{\osc}, j}-U_{\osc} \right)}_k \right| \\
& \
 \hspace{7cm}\times \left| \widehat{\left( U_{{\osc}, j}+ U_{\osc} \right)}_{\left( m_h, n_3-k_3 \right)} \right| \d t 
\\
 & + \int_{\R_+} \sum_{n,k_h, m_h}\left| \hat{\phi}_n \right| \sum_{\left\{ k_3 : \left( k, \left( m_h, n_3-k_3 \right), n \right)\in \mathcal{K}^\star \right\}}
\left| k\right|^{-\varepsilon/2} \left|\widehat{\left( U_{{\osc}, j}-U_{\osc} \right)}_k \right|\\
& \ 
\hspace{4cm}\times \left(  \left| n \right|^{\alpha_1} \left| m_h \right|^{\alpha_2} \left| k_h \right|^{\alpha_3} \right)^{\varepsilon/2} \left| \widehat{\left( U_{{\osc}, j}+ U_{\osc} \right)}_{\left( m_h, n_3-k_3 \right)} \right| \d t
\\
= & \ 
\int_{\R_+} \sum_{n,k_h, m_h}\left| \hat{\phi}_n \right| \sum_{\left\{ k_3 : \left( k, \left( m_h, n_3-k_3 \right), n \right)\in \mathcal{K}^\star \right\}}
\left| k_h \right|^{\varepsilon/2}
\left| k\right|^{-\varepsilon/2} \left|\widehat{\left( U_{{\osc}, j}-U_{\osc} \right)}_k \right|\\
& \ 
\hspace{7cm}\times  \left| \widehat{\left( U_{{\osc}, j}+ U_{\osc} \right)}_{\left( m_h, n_3-k_3 \right)} \right| \d t
  \\
   & \  +
\int_{\R_+} 
\sum_{n,k_h, m_h} \left| n \right|^{\frac{\alpha_1 \varepsilon}{2}} \left| \hat{\phi}_n \right| \sum_{\left\{ k_3 : \left( k, \left( m_h, n_3-k_3 \right), n \right)\in \mathcal{K}^\star \right\}}
\left| k_h \right|^{\frac{\alpha_3 \varepsilon}{2}} \left| k \right|^{-\varepsilon/2}\left|\widehat{\left( U_{{\osc}, j}-U_{\osc} \right)}_k \right|\\
& \ 
 \hspace{6cm} \times  \left| m_h \right|^{\frac{\alpha_2 \varepsilon}{2}} 
   \left| \widehat{\left( U_{{\osc}, j}+ U_{\osc} \right)}_{\left( m_h, n_3-k_3 \right)} \right|\d t\\
   = & \ I_{1,j}+I_{2,j}.
   \end{aligned}
\end{multline*}
We prove  that $ I_{2,j}\to 0 $ as $ j\to \infty $. In order to prove that $ I_{1,j}\to 0 $ the procedure is very similar (and actually simpler) to the one we are going to perform now, for this reason is omitted.
We start remarking that
\begin{align*}
\left| k_h \right|^{\frac{\alpha_3 \varepsilon}{2}} \left| k\right|^{-\varepsilon/2}\left|\widehat{\left( U_{{\osc}, j}-U_{\osc} \right)}_k \right|=  
\left( \left| k \right|^{-\varepsilon} \left|\widehat{\left( U_{{\osc}, j}-U_{\osc} \right)}_k \right| \right)^{1/2}
\left( \left| k_h \right|^{{\alpha_3 \varepsilon}} \left| \widehat{\left( U_{{\osc}, j}-U_{\osc} \right)}_k \right| \right)^{1/2},
\end{align*}
hence
\begin{multline}\label{eq:Leraysolbilinoscterzaeq}
I_{2,j}
\lesssim 
\int_{\R_+} 
\sum_{n,k_h, m_h} \left| n \right|^{\frac{\alpha_1 \varepsilon}{2}} \left| \hat{\phi}_n \right| \sum_{\left\{ k_3 : \left( k, \left( m_h, n_3-k_3 \right), n \right)\in \mathcal{K}^\star \right\}}
 \left( \left| k \right|^{-\varepsilon} \left|\widehat{\left( U_{{\osc}, j}-U_{\osc} \right)}_k \right| \right)^{1/2}\\
\times \left( \left| k_h \right|^{{\alpha_3 \varepsilon}} \left| \widehat{\left( U_{{\osc}, j}-U_{\osc} \right)}_k \right| \right)^{1/2}
   \left| m_h \right|^{\frac{\alpha_2 \varepsilon}{2}} 
   \left| \widehat{\left( U_{{\osc}, j}+ U_{\osc} \right)}_{\left( m_h, n_3-k_3 \right)} \right| \d t.
\end{multline}
Applying Lemma \ref{product rule}  we obtain
\begin{equation*}
I_{2,j}
\lesssim
\left\| \phi \right\|_{L^\infty_\loc\left( \R_+; H^{\frac{1}{2}+ \frac{\alpha_1\varepsilon}{2}}\right)}
 \left\| U_{{\osc},j}+U_{{\osc}} \right\|_{L^2_\loc\left( \R_+;H^{\frac{1}{2}+ \frac{\alpha_2\varepsilon}{2}, 0}\right)}
\left\| U_{{\osc}, j}-U_{\osc} \right\|_{L^2_\loc\left( \R_+; H^{\alpha_3\varepsilon, 0}\right) }^{1/2}
 \left\| U_{{\osc}, j}-U_{\osc} \right\|_{L^\infty_\loc\left( \R_+; H^{-\varepsilon}\right) }^{1/2}.
\end{equation*}
Both $ U_{{\osc},j}, U_{\osc} $ belong to $ L^\infty\left( \R_+; L^2 \right) $ and $ L^2 \left( \R_+; \dot{H}^{1,0} \right) $, and hence to  $ L^2_\loc \left( \R_+; L^2 \right) $ and by interpolation to $ L^2_\loc \left( \R_+; \dot{H}^{\sigma,0} \right) $ for $ \sigma\in \left( 0,1 \right) $. This means that is $ \varepsilon $ is sufficiently small the quantities $  \left\| U_{{\osc},j}+U_{{\osc}} \right\|_{L^2_\loc\left( \R_+;H^{\frac{1}{2}+ \frac{\alpha_2\varepsilon}{2}, 0}\right)}$, $
\left\| U_{{\osc}, j}-U_{\osc} \right\|_{L^2_\loc\left( \R_+; H^{\alpha_3\varepsilon, 0}\right) }^2 $ are bounded, while since 
$$
\left\| U_{{\osc}, j}-U_{\osc} \right\|_{L^\infty_\loc\left( \R_+; H^{-\varepsilon}\right) }^2 \xrightarrow{j\to\infty}0,
$$
we proved that $ I_{2,j}\to 0 $ distributionally. This implies hence that $ \mathcal{Q} \left( U_{{\osc},j}, U_{{\osc}, j} \right) \to \mathcal{Q} \left( U_{{\osc}}, U_{{\osc}} \right)  $ in a distributional sense.

\section{Weak convergence in the weak limit as $ \varepsilon\to 0 $.}\label{sec:weak_conv}

\label{the limit subsection}

In the present section we prove Theorem \ref{thm:weak_conv_limit_system_cP}.\\

Introducing the filtered system \eqref{filtered systemP} allows us to deal with a system of equations which has a closer form to the classical \NS\  system. In particular we can not have any uniform bound, in $\varepsilon$, for the norm $\left\| \partial_t V^\varepsilon\right\|_{H^s\left(\mathbb{T}^3\right)}$, but this is possible for the system \eqref{filtered systemP}. 
We recall that we denoted $ \left| n \right|_F= \sqrt{n_1^2+n_2^2+F^2n_3^2} $.\\

\noindent It is natural to ask ourselves if in the limit $ \varepsilon \to 0 $ the filtered system \eqref{filtered systemP} converges to the limit system \eqref{lim syst}.

\begin{prop}\label{prop:topological_convergence}
Let $ U_0\in \cPHs $ and $ U^\varepsilon $ be a local strong solution identified by Theorem \ref{thm:local_ex_strong_solutions} of \eqref{filtered systemP}, then the sequence $ \left( U^\varepsilon \right)_{\varepsilon > 0} $ has the following regularity uniformly in $ \varepsilon $ 
\begin{align}\label{eq:Ue_in_energy_space}
U^\varepsilon \in L^\infty \left( \R_+ ; \cPLtwo \right), &&
\nh U^\varepsilon \in L^2 \left( \R_+ ; \cPLtwo \right),
\end{align}
and is compact in the space
\begin{equation*}
L^2_{\loc} \left( \R_+ ; H^{-\eta} \left( \T^3 \right) \right),
\end{equation*}
for some $ \eta >0 $ (possibly small).
\end{prop}

\begin{proof}
The proof of \eqref{eq:Ue_in_energy_space} is merely an $ \cPLtwo $ energy estimate on the filtered system \eqref{filtered systemP}, hence is omitted.\\
We prove now that $ \left( \partial_t U^\varepsilon \right)_{\varepsilon} $ is bounded, uniformly in $ \varepsilon $, in $ L^2_{\loc} \left( \R_+; H^{-N} \right) $ where $ N $ is large.\\
The only thing to prove is to control the bilinear interaction $ \mathcal{Q}^\varepsilon \left( U^\varepsilon, U^\varepsilon \right) $ in the $ L^2_{\loc} \left( \R_+; H^{-N} \right) $ space. Let $ \phi $ be a test function:
\begin{align*}
\left| \int_{\R_+\times \T^3} \mathcal{Q}^\varepsilon \left( U^\varepsilon, U^\varepsilon \right) \cdot \phi \ \dx \ \d t \right| = & 
\ \left| \int_{\R_+\times \T^3} \left[ \Lplus U^\varepsilon \otimes \Lplus U^\varepsilon \right] : \nabla \phi
\  \dx \ \d t \right|\\
\leqslant & \
 \left\| U^\varepsilon \right\|_{L^\infty  \left( \R_+; L^2 \right)}
 \left\| U^\varepsilon \right\|_{L^2_{\loc} \left( \R_+; H^{1,0} \right)} \left\| \nabla \phi \right\|_{L^2 \left( \R_+;  L^\infty_v L^2_h \right)}
\end{align*}
Indeed \eqref{eq:Ue_in_energy_space} assures us that $ U^\varepsilon \in L^2_{\loc}\left( \R_+ ; H^{1,0} \left( \T^3 \right) \right) $ uniformly in $ \varepsilon $, whence, by density,  we proved that  $ \left( \partial_t U^\varepsilon \right)_{\varepsilon} $ is bounded, uniformly in $ \varepsilon $, in $ L^2_{\loc} \left( \R_+; H^{-N} \right) $ where $ N $ is large.
It suffice hence to apply Aubin-Lions lemma (see \cite{Aubin63}) to deduce the claim.
\end{proof}

\noindent
Proposition \ref{prop:topological_convergence} asserts hence that (up to subsequences, not relabeled):
\begin{equation*}
U^\varepsilon = U + r^\varepsilon,
\end{equation*}
where $ r^\varepsilon $ is an $ L^2_{\loc} \left( \R_+ ; H^{-\eta} \left( \T^3 \right) \right) $ perturbation and $ U $ is a non-highly-oscillating state. In what follows we denote as $ V_{{\QG}} $ the projection onto the non-oscillating space defined in \eqref{eq:QG-osc} of the limit non-highly-oscillating state $ U $, similarly $ U_{{\osc}} $ is the projection of $ U $ onto the oscillating subspace. The element $\Omega$ is indeed defined as $\Omega= -\partial_2U^1+\partial_1 U^2 -F \partial_3 U^4$.\\

\noindent
First of all we have to make sense of a convergence of the from
\begin{equation*}
\mathcal{Q}^\varepsilon \left( U^\varepsilon, U^\varepsilon \right)\to
\mathcal{Q} \left( U, U \right), 
\end{equation*}
where $ U $ is a weak solution of the limit system \eqref{lim syst} of which we can say at best that it belongs to the space 
\begin{align}
U \in L^\infty \left( \R_+ ; \cPLtwo \right), &&
\nh U \in L^2 \left( \R_+ ; \cPLtwo \right),
\end{align}
thanks to Theorem \ref{thm:Leray_sol}, and $ \left( U^\varepsilon \right)_{\varepsilon > 0} $ a (not relabeled) sequence of local strong solutions of \eqref{filtered systemP} which satisfy \eqref{eq:Ue_in_energy_space} uniformly in $ \varepsilon $ and that converge to a limit element $ U $ in $ L^2_{\loc } \left( \R_+; H^{-\eta} \right) $ for some $ \eta >0 $. In fact in order to define $ \mathcal{Q} $ in \eqref{limit quadratic} we applied the nonstationary phase theorem for smooth function. This is obviously not the case but by mollification we can deduce the same result.

\begin{lemma}\label{lem:NSPT}
Let $ \left( U^\varepsilon \right)_{\varepsilon > 0} $ a (not relabeled) sequence of local strong solutions of \eqref{filtered systemP} which satisfy \eqref{eq:Ue_in_energy_space} uniformly in $ \varepsilon $ and that converges to a limit element $ U $ in $ L^2_{\loc } \left( \R_+; H^{-\eta} \right) $ for some $ \eta >0 $. Then the following limit holds in the sense of distributions
  \begin{equation*}
\mathcal{Q}^\varepsilon \left( U^\varepsilon, U^\varepsilon \right)\to
\mathcal{Q} \left( U, U \right).
\end{equation*}
\end{lemma}

\begin{proof}
Let us define the mollifications
\begin{align*}
U^\varepsilon_\alpha = \mathcal{F}^{-1} \left( 1_{\set{\left| n \right|\leqslant \frac{1}{\alpha}}} \hat{U^\varepsilon} \right),
&&
U _\alpha = \mathcal{F}^{-1} \left( 1_{\set{\left| n \right|\leqslant \frac{1}{\alpha}}} \hat{U} \right).
\end{align*}
Indeed
\begin{equation}\label{eq:SF1}
\begin{aligned}
\mathcal{Q}^\varepsilon \left( U^\varepsilon, U^\varepsilon \right)-
\mathcal{Q} \left( U, U \right)
= & \ \mathcal{Q}^\varepsilon \left( U^\varepsilon, U^\varepsilon \right)
- \mathcal{Q}^\varepsilon \left( U^\varepsilon_\alpha, U^\varepsilon_\alpha \right)\\
& + \mathcal{Q}^\varepsilon \left( U^\varepsilon_\alpha, U^\varepsilon_\alpha \right) - \mathcal{Q} \left( U _\alpha, U _\alpha \right)\\
& + \mathcal{Q} \left( U _\alpha, U _\alpha \right) - \mathcal{Q} \left( U , U  \right),
\end{aligned}
\end{equation}
and 
\begin{equation}
\label{eq:SF2}
\begin{aligned}
\ \mathcal{Q}^\varepsilon \left( U^\varepsilon, U^\varepsilon \right)
- \mathcal{Q}^\varepsilon \left( U^\varepsilon_\alpha, U^\varepsilon_\alpha \right) \xrightarrow{\alpha \to 0} & \ 0,\\
 \mathcal{Q} \left( U _\alpha, U _\alpha \right) - \mathcal{Q} \left( U , U  \right)  \xrightarrow{\alpha \to 0} & 0,
\end{aligned}
\end{equation}
weakly since $ U^\varepsilon_\alpha\xrightarrow{\alpha\to 0} U^\varepsilon $, $ U _\alpha\xrightarrow{\alpha\to 0} U $  in $ L^\infty_{\loc} \left( \R_+; L^2 \left( \T^3 \right) \right) $. Being the space domain $ \T^3 $ compact we do not require a passage to subsequences on the parameter $ \alpha $ but the convergence holds true for the entire sequence. Next we can say that
\begin{align*}
 \mathcal{Q}^\varepsilon \left( U^\varepsilon_\alpha, U^\varepsilon_\alpha \right) - \mathcal{Q} \left( U _\alpha, U _\alpha \right)
= & \
\left(  \mathcal{Q}^\varepsilon \left( U^\varepsilon_\alpha, U^\varepsilon_\alpha \right) -
\mathcal{Q}^\varepsilon \left( U _\alpha, U _\alpha \right) \right)\\
&
+
\left( \mathcal{Q}^\varepsilon \left( U _\alpha, U _\alpha \right) 
 - \mathcal{Q} \left( U _\alpha, U _\alpha \right) \right),
\end{align*}
and again, for $ \alpha >0 $ fixed 
\begin{equation}\label{eq:SF3}
\mathcal{Q}^\varepsilon \left( U^\varepsilon_\alpha, U^\varepsilon_\alpha \right) -
\mathcal{Q}^\varepsilon \left( U _\alpha, U _\alpha \right) \xrightarrow{\varepsilon \to 0} 0,
\end{equation}
weakly since $ U^\varepsilon_{\alpha}\xrightarrow{\varepsilon\to 0} U_\alpha $ in $ L^2_{\loc} \left( \R_+; H^{-\eta} \right) $ due to the topological argument performed in Proposition \ref{prop:topological_convergence}, while finally we can apply the nonstationary phase theorem on $ \mathcal{Q}^\varepsilon \left( U _\alpha, U _\alpha \right) 
 - \mathcal{Q} \left( U _\alpha, U _\alpha \right) $ deducing that
 \begin{equation}\label{eq:SF4}
 \mathcal{Q}^\varepsilon \left( U _\alpha, U _\alpha \right) 
 - \mathcal{Q} \left( U _\alpha, U _\alpha \right) \xrightarrow{\varepsilon \to 0} 0,
 \end{equation}
 in the sense of distributions for  $ \alpha >0 $ fixed. Whence \eqref{eq:SF1}--\eqref{eq:SF4} imply that, fixed a (possibly small) positive $ \alpha >0 $, considering a $ \phi \in \mathcal{D} \left( \R_+\times \T^3 \right) $,  there exists a $ c_\alpha= c_\alpha \left( \phi \right) > 0 $ such that $ c_\alpha \to 0 $ as $ \alpha \to 0 $ and such that
 \begin{equation}\label{eq:SF5}
\lim _{\varepsilon\to 0} \left| \int_{\R_+\times \T^3} \left( \mathcal{Q}^\varepsilon \left( U^\varepsilon, U^\varepsilon \right)-
\mathcal{Q} \left( U, U \right) \right) \cdot \phi \ \d x \ \d t \right| \leqslant c_\alpha.
 \end{equation}
 The left-hand side of \eqref{eq:SF5} is indeed independent from the parameter $ \alpha $, whence
 \begin{equation*}
 \lim _{\varepsilon\to 0} \left| \int_{\R_+\times \T^3} \left( \mathcal{Q}^\varepsilon \left( U^\varepsilon, U^\varepsilon \right)-
\mathcal{Q} \left( U, U \right) \right) \cdot \phi \ \d x \ \d t \right| \leqslant
\lim _{\alpha \to 0} \  c_\alpha=0.
 \end{equation*}
\end{proof}

\noindent
We underline the fact that the following calculations are an adaptation of the ones present in the work \cite{gallagher_schochet} to the case of anisotropic viscosity. For this reason many calculations shall not be carried out in detail, or we shall directly refer to the work \cite{gallagher_schochet} and references therein.\\

Once the convergence for the bilinear interactions is formalized we focus to understand how the global splitting introduced in Section \ref{sec:global_splitting_bilinear} can be applied on bilinear interactions of elements which are not smooth.\\
P. Embid and A. Majda proved the following lemma in \cite{embid_majda}:
 \begin{lemma}\label{lemma:limit_QG_part_bilinear_form}
 $\mathcal{F}^{-1} \left( \left(\mathcal{F}\mathcal{Q}^\varepsilon \left(U^\varepsilon, U^\varepsilon\right)\left|\left|n\right|_F e^0(n)\right.\right)_{\mathbb{C}^4} \right)\xrightarrow{\varepsilon\to 0}v_{{\QG}}\cdot\nabla\Omega$. The limit holds in the sense of distributions.
 \end{lemma}
 \begin{proof}
Let us compute 
\begin{multline*}
\mathcal{F}^{-1} \left( \left(\mathcal{F}\mathcal{Q}^\varepsilon \left(U^\varepsilon, U^\varepsilon\right)\left|\left|n\right|_F e^0(n)\right.\right)_{\mathbb{C}^4} \right) -v_{{\QG}}\cdot\nabla\Omega
\\
\begin{aligned}
= & \ \mathcal{F}^{-1} \left( \left(\mathcal{F}\mathcal{Q}^\varepsilon \left(U^\varepsilon, U^\varepsilon\right)\left|\left|n\right|_F e^0(n)\right.\right)_{\mathbb{C}^4} \right) -
\mathcal{F}^{-1} \left( \left(\mathcal{F}\mathcal{Q}^\varepsilon  \left(U^\varepsilon_{\alpha}, U^\varepsilon_{\alpha}\right)\left|\left|n\right|_F e^0(n)\right.\right)_{\mathbb{C}^4} \right)\\
& + \mathcal{F}^{-1} \left( \left(\mathcal{F}\mathcal{Q}^\varepsilon  \left(U^\varepsilon_{\alpha}, U^\varepsilon_{\alpha}\right)\left|\left|n\right|_F e^0(n)\right.\right)_{\mathbb{C}^4} \right)
-v_{\QG, \alpha}\cdot \nabla \Omega_{\alpha}\\
&+v_{\QG, \alpha}\cdot \nabla \Omega_{\alpha}
-v_{\QG}\cdot \nabla \Omega.
\end{aligned}
\end{multline*}
The element 
\begin{equation*}
 \mathcal{F}^{-1} \left( \left(\mathcal{F}\mathcal{Q}^\varepsilon \left(U^\varepsilon, U^\varepsilon\right)\left|\left|n\right|_F e^0(n)\right.\right)_{\mathbb{C}^4} \right) -
\mathcal{F}^{-1} \left( \left(\mathcal{F}\mathcal{Q}^\varepsilon  \left(U^\varepsilon_{\alpha}, U^\varepsilon_{\alpha}\right)\left|\left|n\right|_F e^0(n)\right.\right)_{\mathbb{C}^4} \right)
\xrightarrow[\alpha \to 0]{\mathcal{D}' \left( \R_+\times \T^3 \right)} 0,
\end{equation*}
since $ U^\varepsilon_{\alpha}\xrightarrow{\alpha\to 0} U^\varepsilon $ in $ L^\infty_{\loc} \left( \R_+ ; L^2 \right) $. Next applying the nonstationary phase theorem and Lemma \ref{lem:splitting_bilinear_QG} we can say that
\begin{equation*}
\mathcal{F}^{-1} \left( \left(\mathcal{F}\mathcal{Q}^\varepsilon  \left(U^\varepsilon_{\alpha}, U^\varepsilon_{\alpha}\right)\left|\left|n\right|_F e^0(n)\right.\right)_{\mathbb{C}^4} \right)
-v_{\QG, \alpha}\cdot \nabla \Omega_{\alpha} \to 0,
\end{equation*}
as $ \varepsilon\to 0 $ in the sense of distributions. Lastly again we can argue as above in order to state that
\begin{equation*}
v_{\QG, \alpha}\cdot \nabla \Omega_{\alpha}
-v_{\QG}\cdot \nabla \Omega \xrightarrow[\varepsilon\to 0]{\mathcal{D}' \left( \R_+\times \T^3 \right)} 0,
\end{equation*}
since $ v_{\QG, \alpha}\to v_{\QG} $ and $ \Omega_{\alpha}\to \Omega $ in $ L^\infty_{\loc} \left( \R_+; L^2 \right) $, concluding.
 \end{proof}
We want to understand which are the projections of $\mathbb{D}^\varepsilon U$ on the oscillatory and non oscillatory space as $\varepsilon \to 0$. This is easily done if we consider the formulation of the limit form as it is given in \eqref{limit linear}. Let us consider the projection of the limit linear form onto the potential space defined by $ \Omega = \mathcal{F}^{-1} \left( \left(  \mathcal{F}U \left| \left| n \right|_F e^0 \right. \right)_{\mathbb{C}^4} \right) $,
$$
\left(  \mathcal{F}\mathbb{D}U \left| \left| n \right|_F e^0\right. \right)_{\mathbb{C}^4}=  \sum_{\omega_n^{a,b}=0} \left( \left. \mathbf{D} (n) U^{b}(n)\right| e^a(n)\right)_{\mathbb{C}^4} \left( \left.  e^a(n) \right|  \left| n \right|_F e^0 \right)_{\mathbb{C}^4} .
$$ 
As it has been pointed out above $ e^0 \perp e^\pm $, hence   $a=0$. On the other hand if we consider the limit set $\omega_n^{a,b}=0$ with the fact that $a=0$ we easily obtain that $\omega^b(n)\equiv 0$, whence $b=0$ as well, hence we obtained that
\begin{equation*}
\begin{aligned}
\left(\left.-\mathcal{L}\left(-\frac{t}{\varepsilon}\right) \mathbf{D} 
\mathcal{L}\left(\frac{t}{\varepsilon}\right)U^\varepsilon \right|
\mathcal{F}^{-1} \left( \left| n_F \right| e^0 \right)\right)_{\mathbb{C}^4} 
\xrightarrow{\varepsilon\to 0}& \  a_{\text{QG}}\left(D_h\right)\Omega
 = & \
\mathcal{F}^{-1}\left(  
\frac{\nu\left( n_1^2 + n_2^2 \right)+ \nu 'F^2 n_3^2 }{n_1^2 + n_2^2 +F^2 n_3^2} \left( n_1^2+n_2^2 \right) \widehat{\Omega}_n \right)
.
\end{aligned} 
\end{equation*}
In the same way, defining  $ U^a= \left( \left. \mathcal{F} U \right| e^a \right) e^a $
\begin{align*}
-\lim_{\varepsilon\to 0} \mathcal{L}\left(-\frac{t}{\varepsilon}\right) \mathbf{D} 
\mathcal{L}\left(\frac{t}{\varepsilon}\right)U^\varepsilon_\text{osc}= & \;
a_\text{osc}\left(D_h\right)U^\varepsilon_\text{osc} 
=  \;
\mathcal{F}^{-1}
\left( \sum_{\substack{\omega^{a,b}_n=0\\a,b=\pm}} \left( \left. \mathbf{D} (n) U^{b}(n) \right| e^a \left( n \right) \right)_{\mathbb{C}^4} e^a\left( n \right) \right).
\end{align*}

We want now to understand which form assumes the limit as $ \varepsilon \to 0 $ of the projection of $ \mathcal{Q}^\varepsilon\left(U^\varepsilon, U^\varepsilon\right) $ onto the oscillatory subspace $ \C e^- \oplus \C e^+ $. In particular the following result holds true:
\begin{lemma}\label{lemma:bilinear_part_osc}
For every three-dimensional torus $ \mathbb{T}^3 $ we have
 \begin{equation}
 \label{eq:bilinear_part_osc}
 \mathcal{Q}^\varepsilon\left(U^\varepsilon, U^\varepsilon\right)_{{\osc}}
\xrightarrow{\varepsilon\to 0} \left( \mathcal{Q}\left(V_{\QG},U_{\osc}\right) \right)_{\osc}+
\left( \mathcal{Q}\left(U_{\osc},V_{\QG}\right) \right)_{\osc} + \left( \mathcal{Q}\left(U_{\osc},U_{\osc}\right) \right)_{\osc}. 
 \end{equation}
\end{lemma}

\begin{proof}
We avoid to give a detailed proof of such result since the proof is very similar to the one performed in Lemma \ref{lemma:limit_QG_part_bilinear_form} but using Corollary \ref{cor:splitting_bilinear_osc} instead of Lemma \ref{lem:splitting_bilinear_QG}.
\end{proof}

The above lemmas hence states that in the limit  $ \varepsilon \to 0 $ there is no bilinear interaction of kernel elements in the equation describing the evolution of $ U_{\osc} $.\\

Whence the filtered system \eqref{filtered systemP} can be described, as $\varepsilon\to 0$, thanks to the following two systems:
\begin{align}
&\left\lbrace\label{nonosc limit}
\begin{array}{rcl}
\partial_t\Omega+v_{\QG}^h\cdot\nh \Omega+ a_{\QG}\left(D_h \right)\Omega & = & 0\\
\diveh v^h_{\QG}= \dive v_{\QG}=0\hfill\\
\Omega\Bigr|_{t=0}=\Omega_0\hfill 
\end{array}
\right.
\\
&\left\lbrace\label{osc limit}
\begin{aligned}
&
\begin{multlined}
\partial_t U_{\osc}+\left( \mathcal{Q}\left(V_{\QG},U_{\osc}\right) \right)_{\osc}+
\left( \mathcal{Q}\left(U_{\osc},V_{\QG}\right) \right)_{\osc}
\\
 + \left( \mathcal{Q}\left(U_{\osc},U_{\osc}\right) \right)_{\osc} +a_{\osc}\left(D_h\right)U_{\osc}=0
\end{multlined}\\
& \dive u_{\osc} =0\hfill\\
& U_{\osc}\Bigr|_{t=0}=U_{{\osc},0}=\left(V_0\right)_{\osc}.\hfill
\end{aligned}
\right.
\end{align}
The system \eqref{nonosc limit} represents the projection of the limit system onto the non-oscillatory potential subspace defined by $ \Omega $, and \eqref{osc limit} represents the projection onto $ \mathbb{C}e^- \oplus \mathbb{C}e^+$.\\
 
 It is easy to deduce from \eqref{nonosc limit} that if $\definizioneVQG$ then\
 \begin{equation}
 \left\lbrace
 \begin{array}{l}
  \label{quasi geostrophic equation}
 \partial_t V_{\QG} +a_{\QG}\left( D_h\right)V_{\QG}=- \left(
 \begin{array}{c}
 \nh^\perp\\
 0\\
 -\partial_3 F
 \end{array}\right)
 \Delta^{-1}_F\left( v_{\QG}^h \cdot \nh \Omega\right),\\
 \diveh v^h_{\QG}= \dive v_{\QG}=0\hfill\\
 V_{\QG}\Bigr|_{t=0}=V_{{\QG},0}= \left( \nhp, 0, -F\partial_3 \right)^\intercal \Delta_{F}^{-1}\Omega_{0}.
  \end{array}
 \right.
 \end{equation}
 We remark that in the equation \eqref{osc limit} the term $ \mathcal{Q} \left( U_{\osc}, U_{\osc} \right) $ represents a bilinear interaction between highly oscillating modes, i.e. we are taking into account some potentially resonant effect such as in  \cite{paicu_rotating_fluids}.

 The following lemma gives a connection in terms of regularity between the solutions of \eqref{nonosc limit} and \eqref{quasi geostrophic equation}, and will result to be extremely useful in the energy estimates for the global well posedness of the limit system.

\begin{lemma}\label{higher regularity VQG}
Let $\Lh^s \Lv^{s'}\Omega\in\cPLtwo$, with $\definizioneVQG$. Let $\sigma\in [0,1]$, then there exists a uniformly finite (in $ \sigma $) constant $C_\sigma$ depending only on $\sigma$ such that
\begin{align*}
\left\| \Lh^{s+\sigma}\Lv^{s'+\left( 1-\sigma \right)} v_{\QG} \right\|_{\cPLtwo} \leqslant C_\sigma \left\|\Lh^s \Lv^{s'} \Omega \right\|_{\cPLtwo}.
\end{align*}
\end{lemma}

 \subsection{Propagation of the horizontal average.}\label{sec:propagation_horizontal_average}

In the following lemmas we identify some conditions which suffice to guarantee that the horizontal average of $U=U_{\QG}+U_{\osc}$ solution of the limit system \eqref{osc limit}-\eqref{quasi geostrophic equation} is preserved for each time $t>0$. This turns out to be very important since we are dealing with periodic functions, hence, generally we cannot use inequalities such as the one stated in \eqref{GN type ineq} or Corollary \ref{L4Linf embedding} unless the horizontal mean of the function considered is zero. It is in this setting that the condition \ref{cP} shall play a fundamental role.

\begin{lemma}\label{propagation horizontal mean VQG}
Let $V_{\QG}$ the solution of \eqref{quasi geostrophic equation}, if we define 
$$\underline{V_{\QG}}(t, x_3)=
\frac{1}{\left| \T^2_h \right|}
\int_{\mathbb{T}^2_h} V_{\QG} \left(t,  y_h,x_3\right) \d y_h,
$$
then
$$
\partial_t \underline{V_{\QG}}(t, x_3) =0.
$$
\end{lemma}
\begin{proof}
It suffices to remark that $$ \left( -\partial_2, \partial_1, 0, -F \partial_3 \right)^\intercal \Delta^{-1}_F \left( v^h_{\QG} \cdot \nh \Omega \right) = \left( -\partial_2, \partial_1, 0, -F \partial_3 \right)^\intercal \Delta^{-1}_F \diveh \left( v^h_{\QG}  \Omega \right) . $$
\end{proof}

 \begin{lemma}\label{propagation horizontal mean osc part}
 Suppose that the limit system \eqref{nonosc limit}--\eqref{quasi geostrophic equation} is well posed. Then, setting $U=V_{\QG}+U_{\osc}$  
 $$
 \partial_t \int_{\mathbb{T}^2_h} U(t, x_h,x_3)\d x_h =0,
 $$
 for almost every torus $\mathbb{T}\subset \mathbb{R}^3$.
 \end{lemma}
 \begin{proof}
Taking in consideration the oscillatory part described by equation \eqref{osc limit} it suffices to prove that
\begin{align*}
\int_{\mathbb{T}^2_h} \mathcal{Q}\left( V_{\QG}, U_{\osc}\right) \d x_h=
 \int_{\mathbb{T}^2_h} \mathcal{Q}\left( U_{\osc}, V_{\QG} \right) \d x_h = \int_{\mathbb{T}^2_h} \mathcal{Q}\left( U_{\osc}, U_{\osc} \right) \d x_h &=0,
\end{align*}
 we consider at first the term $\int_{\mathbb{T}^2_h} \mathcal{Q}\left( V_{\QG}, U_{\osc}\right) \d x_h$. To do so we consider 
 $$
 \mathcal{F}\int_{\mathbb{T}^2_h} \mathcal{Q}\left( V_{\QG}, U_{\osc}\right) \d x_h=
  \sum_{\substack{ \omega^{0,b,c}_{k,m,\left( 0,n_3 \right)}=0\\ b, c=\pm\\k+m=\left( 0,n_3\right)} }
  \left( n_3  \hat{v}^3_{\QG} \left( k\right) \right) \hat{U}_{\osc} \left( m\right).
 $$
If we look what the term $ \hat{v}^3_{\QG} \left( k \right) $ is we can easily deduce that $  \hat{v}^3_{\QG} \left( k \right)= \hat{V}_{\QG} \left( k \right)\cdot e^0\left( k \right) e^{0,3}\left( k \right) $, where $ e^0 $ is defined in \eqref{eigenvalues} and $ e^{0,3} $ is the third component of $ e^0 $. Looking at \eqref{eigenvalues} we immediately notice that $ e^{0,3} \equiv 0 $, and hence the above value is null.\\
Next we consider the following term 
\begin{equation}
\label{eq:probl_term_propagation_horizontal_average}
\mathcal{F}\int_{\mathbb{T}^2_h}\mathcal{Q}\left(U_{\osc},V_{\QG}\right)\d x_h= 
\sum_{\substack{ \omega^{a,0,c}_{k,m,\left( 0,n_3 \right)}=0\\ a,c=\pm\\k+m=\left( 0,n_3\right)} } \left(\left. \left( \left( n_3  \hat{u}^3_{\osc} \left( k\right) \right) \hat{V}_{\QG} \left( m \right)\right) \right|  e^c\left( n \right)  \right)_{\mathbb{C}^4} e^c\left( n \right),
\end{equation}
to show that the above quantity is zero we have to study the summation set.
 Recall that the eigenvalues are given by formula \eqref{eigenvalues}, the right hand side of the above equation has been evaluated   explicitly thanks to the explicit formulation of the bilinear form $ \mathcal{Q} $. The formulation of the summation set turns out to be quite simple thanks to the relation $ n_h\equiv 0 $, writing down in fact explicitly the relation $ \omega^{a,0,c}_{k,\left( 0,n_3 \right)-k,\left( 0,n_3 \right)}=0 $
 \footnote{ $ \left( 0,n_3 \right)-k=m $ and we recover the same summation set as in \eqref{eq:probl_term_propagation_horizontal_average}. }
  we deduce that we are considering the following modes:
\begin{align*}
\mathcal{K}_{\pm} = & \left\lbrace \left. k\in \mathbb{Z}^3 \right| \omega^\pm \left( k \right)=1 \right\rbrace.
\end{align*}
 The equation $ \omega^\pm \left( k \right)=1 $ characterizing $ \mathcal{K}_{\pm} $ reads as 
$$
\frac{\left(F^2\check{k}_3^2 + { \left| \check{k}_h \right|^2} \right)^{1/2}}{\left| \check{k} \right|}=\pm F,
$$
which is equivalent to
$$
\left( F^2-1 \right)\left| k_h \right|^2=0.
$$
It is trivial that this relation is satisfied only if $ k_h\equiv 0 $, but let us consider now in detail what the element $ \left. \hat{u}^3_{\osc}\left( k \right)\right|_{k=\left( 0,k_3 \right)} $ appearing in \eqref{eq:probl_term_propagation_horizontal_average} is. By definition $ \hat{u}^3_{\osc}\left( k \right)= \left(\left. \mathcal{F} U \left( k \right) \right| e^\pm\left( k \right)  \right) e^{\pm,3}\left( k \right) $, where $ e^{\pm,3}\left( k \right) $ is the third component of the oscillating eigenvectors defined in \eqref{eigenvectors zero}, i.e. $  e^{\pm,3}\left( k \right) \equiv 0 $. Whence $  \hat{u}^3_{\osc}\left( 0,k_3 \right) \equiv 0 $ and this implies that the contribution in \eqref{eq:probl_term_propagation_horizontal_average} is zero.\\
Next we shall deal with the more complex term, namely the term
$$
\int_{\mathbb{T}^2_h} \left( \mathcal{Q}\left( U_{\osc}, U_{\osc} \right) \right)_{\osc} \d x_h,
$$
being the deduction for the other ones appearing a matter of straightforward computations.
In this term there are present interactions between perturbations which do not live in the kernel of the penalized operator. In this context the resonance set defined in Definition \ref{resonance set} shall play a fundamental role. Let us consider the explicit expression of the above term
\begin{equation*}
\int_{\mathbb{T}^2_h} \left( \mathcal{Q}\left( U_{\osc}, U_{\osc} \right) \right)_{\osc} \d x_h
 =  \; \mathcal{F}^{-1}\left(
\sum_{\substack{ \mathcal{K}^\star_{\left( 0,\check{n}_3 \right)}}} \left(\left.\sum_{j=1,2,3} U^{a,j}\left( k \right) m_j U^b\left( m \right)  \right| e^c \left( 0, n_3 \right)  \right)_{\mathbb{C}^4} e^c \left(0, n_3 \right)\right).
\end{equation*}
We prove that the above quantity is zero by proving that $ \mathcal{K}^\star_{\left( 0,\check{n}_3 \right)}=\emptyset $. Since $\check{n}_h=0 $ and we have the convolution constraint $ \check{k}+\check{m}=\check{ n} $ we immediately understand $ \check{k}_h+ \check{m}_h=0 $, i.e. $ \left| \check{k}_h \right| = \left| \check{m}_h \right|= \lambda $. Writing down the resonant equation we obtain the following equality
\begin{align*}
\frac{\left({F^2}\check{k}_3^2 +  \lambda^2 \right)^{1/2}}{\left( \lambda^2+ \check{k}_3^2 \right)^{1/2}} \pm
 \frac{\left({F^2}\check{m}_3^2 + { \lambda^2} \right)^{1/2}}{\left( \lambda^2+ \check{m}_3^2 \right)^{1/2}} = \pm 1.
\end{align*}
Taking square (twice) and after some algebraic manipulation we obtain that the above equation is equivalent to
\begin{align*}
\left(\lambda ^4+F^2 \lambda ^2 \check{m}_3^2+\check{ k}_3^2 \left(-\left(-2+F^2\right) \lambda ^2+\check{m}_3^2\right)\right){}^2
=
4 \left(\lambda ^2+\check{k}_3^2\right){}^2 \left(\lambda ^2+\check{m}_3^2\right) \left(\lambda ^2+F^2 \check{m}_3^2\right).
\end{align*}
We multiply the above equation for $ a_3^8 $, obtaining the new equality in the unknown $ \mu^2= \lambda^2 a_3^2 $
\begin{equation}\label{polynomial resonance propagation horizontal average}
\left(\mu ^4+F^2 \mu ^2 m_3^2+k_3^2 \left(-\left(-2+F^2\right) \mu ^2+m_3^2\right)\right){}^2
=4 \left(\mu ^2+k_3^2\right){}^2 \left(\mu ^2+m_3^2\right) \left(\mu ^2+F^2 m_3^2\right),
\end{equation}
and 
$$
\mu^2= \lambda^2 a_3^2= \left( \frac{a_3}{a_1} \right)^2 k_1^2 + \left( \frac{a_3}{a_2} \right)^2 k_2^2 = \mu_1 k_1^2 + \mu_2 k_2^2.
$$
Since the torus satisfies the Condition \ref{cP} we know that $ F=r_1/r_2\in \mathbb{Q} $, hence we can transform the expression in \eqref{polynomial resonance propagation horizontal average} into an equation of the form $ P\left( \mu \right)= 0 $, with $ P\in \mathbb{Z} \left[ \mu\right] $. Whence by the definition of Condition \ref{cP} given in Definition \ref{def:condition_P} we argue that
\begin{itemize}
\item If $ \mu_1= a_3^2/a_1^2 \in \mathbb{Q} $ the \eqref{polynomial resonance propagation horizontal average} can be rewritten as $ \tilde{P} \left( \mu_2 \right)= 0 $ where $ \deg \tilde{P}= 4 $, hence by hypothesis in Definition \ref{def:condition_P} we have that $ \mu_2 $ is not algebraic of degree smaller or equal than four, this implies that the equation $ \tilde{P} \left( \mu_2 \right)= 0 $ has no solution, concluding.
\item If $ \mu_2= a_3^2/a_2^2 \in \mathbb{Q} $ the procedure is  the same as above, but symmetric (see Definition \ref{def:condition_P}).
\end{itemize}
\end{proof}

We have hence identified some conditions under such we can say that the horizontal mean of the limit function $U=\displaystyle \lim_{\varepsilon\to 0}\Lminus V^\varepsilon$ is preserved. Hence if we consider initial data with zero horizontal average we can use freely  \eqref{GN type ineq} and moreover the following Poincar\'e inequality
$
\left\|U\right\|_{{\cPLp}}\leqslant C \left\|\nh U \right\|_{{\cPLp}},
$ 
 holds.

\section{Propagation of $ \cPHs $ regularity.}
\label{sec:prop_H0s_regularity}

\subsection{The quasi-geostrophic part.}
\label{sec:QG_propagation_regularity}

Subsection \ref{existence limit system} ensures us that there exists a solution $U$ for the limit system \eqref{lim syst} which is 
\begin{align*}
U\in & L^\infty \left( \mathbb{R}_+; \cPLtwo \right)&
\nh U \in & L^2 \left( \mathbb{R}_+; \cPLtwo \right).
\end{align*}
The scope of the present and following section though is to prove if, under suitable initial conditions, the equations \eqref{nonosc limit} and \eqref{osc limit} propagate  $ \cPHs $ regularity.

\begin{prop}
Let $\Omega$ be a solution of \eqref{nonosc limit}. Then if $\Omega_0\in \cPLtwo$
 $\Omega \in L^\infty\left(\mathbb{R}_+;\cPLtwo\right)$ , $\nh \Omega \in L^2\left(\mathbb{R}_+; \cPLtwo\right)$, and in particular for each $ t>0 $ the following bound holds true
 $$
 \left\| \Omega(t)\right\|^2_\cPLtwo+ 2c\int_0^t\left\|\nh \Omega(\tau)\right\|_\cPLtwo^2\d \tau \leqslant C \left\| \Omega_0 \right\|_\cPLtwo^2.
 $$
\end{prop}
 This is a standard $L^2$ energy estimate on the parabolic equation \eqref{nonosc limit} which has been already proved in Theorem \ref{thm:Leray_sol}.  
  
  \begin{prop}\label{propagation H0s norms for Omega}
  Let $\Omega$ be the solution of \eqref{nonosc limit} and let $\Omega_0\in H^{0,s}$ for some $s>0$. Then for all $t\in\mathbb{R}$  we have that $\Omega\in \anuno $ and $\nh \Omega \in \andue$, and in particular the following estimates hold:
\begin{equation}\label{stima spazi anisotropici per omega}
\left\| \Omega(t)\right\|_{\cPHs}^2 + c\int_0^t\left\| \nh \Omega(\tau)\right\|_{\cPHs}^2  \d \tau
 \leqslant
C  \left\|\Omega_0\right\|_{\cPHs} \exp\left\lbrace \frac{2C}{c} 
\left( 1+ \left\| \Omega_0 \right\|_\cPLtwo^2 \right)\left\| \Omega_0 \right\|_\cPLtwo^2
\right\rbrace
\end{equation}
 \end{prop}
\begin{proof}
Applying the vertical truncation ${\cPtv}$ on both sides of equation \eqref{nonosc limit}, multiplying both sides for ${\cPtv}\Omega$ and taking the scalar product in $\cPLtwo$ we obtain
$$
\frac{1}{2}\frac{\d}{\d t}\left\|{\cPtv}\Omega \right\|_{\cPLtwo}^2+c\left\|{\cPtv} \nh \Omega\right\|_{\cPLtwo}^2 \leqslant 
\left| \left( \left.{\cPtv} \left( v^h_{\QG} \cdot \nh \Omega\right) \right| {\cPtv} \Omega\right)_{\cPLtwo}\right|.
$$

By use of Cauchy-Schwartz inequality and \eqref{stima forma bilineare omega} we obtain
\begin{multline}\label{ancora de molt e sommare}
\frac{1}{2}\frac{\d}{\d t}\left\|{\cPtv}\Omega \right\|_{\cPLtwo}^2+c\left\|{\cPtv} \nh \Omega\right\|_{\cPLtwo}^2 \\
\leqslant 
  C\;2^{-2qs} b_q(  t)
 \left[
 \left\| \Omega \right\|_{\cPLtwo}^{1/2} \left\| \nh\Omega \right\|_{\cPLtwo}^{1/2} \left\| \Omega \right\|_{\cPHs}^{1/2} \left\| \nh \Omega \right\|_{\cPHs}^{3/2}\right.
  + \left. \left\| \nh \Omega \right\|_\cPLtwo  \left\|\Omega \right\|_{\cPHs} \left\|\nh \Omega \right\|_{\cPHs}\right]
\end{multline}
We recall that in \eqref{ancora de molt e sommare} $ \left( b_q \right)_q $ is a $ \ell^1 \left( \mathbb{Z} \right) $ positive sequence which depends on $ \Omega $ and such that $ \sum_q b_q \left( t \right)\leqslant 1 $. 
Multiplying equation \eqref{ancora de molt e sommare} on both sides for $2^{2qs}$, summing on $q\in \mathbb{Z}$ and using the convexity inequalities $2ab \leqslant a^2+ b^2$ and $ab \leqslant \frac{1}{4}a^4 +\frac{3}{4}b^{4/3}$ we obtain
\begin{equation}\label{da applicare gronwall}
\frac{1}{2}\frac{\d}{\d t}\left\| \Omega\right\|_{\cPHs}^2 + c\left\| \nh \Omega\right\|_{\cPHs}^2
\leqslant \frac{c}{2}\left\| \nh \Omega\right\|_{\cPHs}^2
 +C \left( \left(1+\left\| \Omega \right\|_{\cPLtwo}^{2}\right) \left\| \nh\Omega \right\|_{\cPLtwo}^{2}  \right)\left\|\Omega\right\|_{\cPHs}^2
\end{equation} 
whence, applying Gronwall inequality to \eqref{da applicare gronwall} in $\left[0,t\right]$ we get the bound
\begin{equation*}
\left\| \Omega(t)\right\|_{\cPHs}^2 + c\int_0^t\left\| \nh \Omega(\tau)\right\|_{\cPHs}^2  \d \tau
 \leqslant
C  \left\|\Omega_0\right\|_{\cPHs} \exp\left\{ 2C \int_0^t \left( 1+ \left\|\Omega(s) \right\|^2_{\cPLtwo}\right)\left\|\nh \Omega (s)\right\|^2_{\cPLtwo}\d s\hfill
\right\}.
\end{equation*}
Hence, considering that $ \Omega $ is bounded in $ L^\infty \left( \R_+; \cPLtwo \right) $ and $ \nh \Omega $ is bounded in $ L^2\left( \R_+; \cPLtwo \right) $ we deduce the estimate \eqref{stima spazi anisotropici per omega}.
\end{proof}
\begin{rem}
In Proposition \ref{propagation H0s norms for Omega} we do not require the initial data to be of zero horizontal average in order to propagate $ H^{0,s} $ norms.\fine
\end{rem}

\subsection{The oscillatory part.}
\label{sec:osc_propagation_regularity}

 We can now turn our attention on the oscillatory part $U_{\osc}$ solution of the equation \eqref{osc limit}. Indeed the terms $ \mathcal{Q}\left( V_{\QG}, U_{\osc} \right) $ and $ \mathcal{Q}\left( U_{\osc}, V_{\QG} \right) $ present in \eqref{osc limit} should not present a problem in the propagation of regularity, being linear in $ U_{\osc} $. The term $  \mathcal{Q}\left( U_{\osc}, U_{\osc} \right)  $ though is a bilinear term of the form $ \sqrt{-\Delta} \left( U_{\osc} \otimes U_{\osc} \right) $. Fortunately as pointed out in Lemma \ref{product rule} the bilinear form $ \mathcal{Q} $ has better product rules than the standard \NS\ bilinear form, this will allow us to recover the global well posedness result for \eqref{osc limit} as well.
  
  \begin{lemma}
 Let $ U $ be the weak solution defined in Theorem \ref{thm:Leray_sol}, then $ U_{\osc}= U- V_{\QG} $ satisfies the energy bound
 $$
 \left\| U_{\osc} \left( t \right) \right\|_\cPLtwo^2 +c \int_0^t \left\| \nh U_{\osc} \left( \tau \right) \right\|_\cPLtwo^2\d \tau \leqslant  C  \left\| U_0\right\|_\cPLtwo^2.
 $$
  \end{lemma}
\begin{proof}
The proof stems from the fact that $ U_{\osc} = \Pi_{\osc} U $ where $ \Pi_{\osc}= 1 - \Pi_{\QG} $ is a pseudo-differential operator of order zero as it has been explained in the proof of Theorem \ref{thm:Leray_sol}.
\end{proof}
  
  \begin{prop}\label{propagation H0s norms for oscillating part}
  Let $U_{\osc}$ be the solution of \eqref{osc limit} and $\underline{V_{{\QG},0}}, \ \underline{U_{{\osc},0}}=0$. Let $ \T^3 $ satisfy the condition \eqref{cP} and $U_{{\osc},0},\Omega_0\in \cPHs$ for $s>1/2$, then $U_{\osc}\in \anuno$ and $\nh U_{\osc}\in \andue$ and the following bound holds
   \begin{multline*}
  \left\| U_{\osc} (t)\right\|_\cPHs^2 + c \int_0^t \left\| U_{\osc} (\tau)\right\|_\cPHs^2 \d \tau \\
  \leqslant C \left\| U_{{\osc},0} \right\|_\cPHs^2
   \exp \left\{
   \frac{2C}{c}  
   \left[   
\left\|\Omega_0\right\|_{\cPHs} \exp\left\lbrace \frac{2C}{c}   
\left( 1+ \left\| \Omega_0 \right\|_\cPLtwo^2 \right)\left\| \Omega_0 \right\|_\cPLtwo^2
\right\rbrace 
+ 
\left( 1 + \left\| U_0 \right\|_{\cPLtwo}^2\right) \left\| U_0 \right\|_{\cPLtwo}^2
\right]
 \right\}.
  \end{multline*}
  \end{prop}
  \begin{proof}
  As in the proof of Proposition \ref{propagation H0s norms for Omega} apply the vertical truncation ${\cPtv}$ on both sides of \eqref{osc limit} and taking scalar product in $\cPLtwo$ we obtain
  \begin{multline*}
  \frac{1}{2}\frac{\d}{\d t} \left\| {\cPtv} U_{\osc}\right\|_{\cPLtwo}^2+ c \left\| {\cPtv} \nh U_{\osc}\right\|_{\cPLtwo}^2\leqslant \left|\left( \left. {\cPtv} \mathcal{Q}\left( V_{\QG},U_{\osc}\right)\right| {\cPtv} U_{\osc}\right) \right|
  \\
  +
  \left|\left( \left. {\cPtv} \mathcal{Q}\left( U_{\osc}, V_{\QG}\right)\right| {\cPtv} U_{\osc}\right) \right|
  + \left| \left( \left. {\cPtv} \mathcal{Q}\left( U_{\osc}, U_{\osc} \right) \right| {\cPtv} U_{\osc} \right) \right|.
  \end{multline*}
  
  Taking moreover in account the estimates \eqref{prima stima termine bilineare parte oscillante} and \eqref{terza stima termine bilineare parte oscillante} the above inequality turns into
  \begin{multline}\label{energy est H0s osc part}
  \frac{1}{2}\frac{\d}{\d t} \left\| {\cPtv} U_{\osc}\right\|_{\cPLtwo}^2+ c \left\| {\cPtv} \nh U_{\osc}\right\|_{\cPLtwo}^2
  \\
\begin{aligned}
 \leqslant & \ C  b_q\left( t \right) 2^{-2qs}\left\|\nh\Omega\right\|_\cPHs \left\| \nh U_{\osc} \right\|_{\cPHs} \left\|  U_{\osc} \right\|_{\cPHs}\\
&+C b_q\left(  t \right) 2^{-2qs}\left\| \Omega \right\|^{1/2}_{\cPHs} \left\| \nh\Omega\right\|^{1/2}_{\cPHs}\left\|U_{\osc}\right\|_{\cPHs}^{1/2} \left\|\nh U_{\osc}\right\|_{\cPHs}^{3/2}\\
&+Cb_q\left(   t \right) 2^{-2qs}\left\| \nh  U_{\osc} \right\|_{\cPLtwo} \left\|  U_{\osc} \right\|_{\cPHs} \left\| \nh  U_{\osc} \right\|_{\cPHs} \\
&+  C b_q\left(   t \right) 2^{-2qs} \left\|  U_{\osc} \right\|_{\cPLtwo}^{1/2} \left\| \nh  U_{\osc} \right\|_{\cPLtwo}^{1/2} \left\|  U_{\osc} \right\|_{\cPHs}^{1/2} \left\| \nh  U_{\osc} \right\|_{\cPHs}^{3/2}.
\end{aligned}  
  \end{multline}
  We recall that $ \left( b_q \right)_q $ is a $ \ell^1 \left( \mathbb{Z} \right) $ positive sequence which depends on $ \Omega $ and $ U_{\osc} $ and such that $ \sum_q b_q \left( t \right)\leqslant 1 $. 
  Multiplying both sides of \eqref{energy est H0s osc part} for $2^{2qs}$, summing over $q\in\mathbb{Z}$, and using the inequalities $2ab \leqslant a^2+ b^2$ and $ab \leqslant \frac{1}{4}a^4 +\frac{3}{4}b^{4/3}$ we obtain
  \begin{equation}\label{still to integrate in time}
\frac{\d}{\d t} \left\|  U_{\osc}\right\|_{\cPHs}^2+c \left\|  \nh U_{\osc}\right\|_{\cPLtwo}^2
\leqslant 2C
\left(
\left(1+  \left\| \Omega \right\|^{2}_{\cPHs} \right) \left\| \nh\Omega\right\|^{2}_{\cPHs}
+\left( 1+ \left\| U_{\osc}  \right\|^{2}_{\cPLtwo}  \right)\left\| \nh U_{\osc}\right\|^{2}_{\cPLtwo} 
\right)
\left\| U_{\osc} \right\|_\cPHs^2
  \end{equation}
  applying Gronwall inequality to \eqref{still to integrate in time} we obtain
   \begin{multline*}
  \left\| U_{\osc} (t)\right\|_\cPHs^2 + c\int_0^t \left\| U_{\osc} (\tau)\right\|_\cPHs^2 \d \tau \\
  \leqslant C \left\| U_{{\osc},0} \right\|_\cPHs 
  \exp \left\lbrace 2C \int_0^t 
\left( 1+ \left\| \Omega (\tau) \right\|^{2}_{\cPHs}  \right)\left\| \nh\Omega(\tau)\right\|^{2}_{\cPHs}+ 
\left( 1+ \left\| U_{\osc} (\tau) \right\|^{2}_{\cPLtwo}  \right)\left\| \nh U_{\osc}(\tau)\right\|^{2}_{\cPLtwo} \d \tau \right\rbrace,
  \end{multline*}
  concluding.
  \end{proof}
  
  \subsection{Proof of Theorem \ref{GWP2}} 
  \label{sec:pf_thm}
  At this point it is very easy to prove Theorem \ref{GWP2} Let us consider a data $V_0\in\cPHs, \Omega_0\in H^{0,s}, s\geqslant 1$ and $V_0$ with zero horizontal average. Thanks to Proposition \ref{propagation H0s norms for Omega} we have that $\Omega\in \mathcal{C}\left( \mathbb{R}_+; H^{0,s-1}\right)\cap \mathcal{C}\left( \mathbb{R}_+; H^{0,s}\right)$, $\nh \Omega \in L^2\left( \mathbb{R}_+; H^{0,s-1}\right)\cap L^2\left( \mathbb{R}_+; H^{0,s}\right)$, which in particular implies, thanks to Lemma \ref{higher regularity VQG} that $\Lv^s V_{\QG} \in \mathcal{C}\left( \mathbb{R}_+; \cPLtwo\right)$, $\nh  \Lv^s V_{\QG} \in L^2\left( \mathbb{R}_+; \cPLtwo\right)$. 
 Since $ V_{\QG} $ is defined as $ V_{{\QG}}=\Pi_{\QG} U $ where $ \Pi_{\QG} $ is a Fourier multiplier of order zero which maps continuously any $ H^{s,s'} $ space to itself, this implies that $ V_{\QG} \in L^\infty \left( \R_+, L^2 \right) $, $ \nh V_{\QG} \in L^2 \left( \R_+, L^2 \right) $ since $ U $ is so thanks to Theorem \ref{thm:Leray_sol}, hence $V_{\QG} \in \mathcal{C}\left( \mathbb{R}_+; {H}^{0,s}\right),\nh V_{\QG} \in L^2\left( \mathbb{R}_+; {H}^{0,s}\right)$. For the oscillating part it suffices to apply Proposition \ref{propagation H0s norms for oscillating part} and the proof is complete.\\
  
We outline how to prove that solutions to the limit system are $H^{0,s'}$-stable, for $ s'\in \left[-1/2,s\right) $ globally with a continuous dependence of the initial data. To do so consider the two solutions $U_1,U_2$ to the limit system
  \begin{align}
  \label{limit 1}
  \left\lbrace
  \begin{array}{l}
  \partial_tU_1 + \mathcal{Q}\left( U_1,U_1 \right) -\mathbb{D} U_1 =0\\
  \dive u_1=0\\
  U_1\Bigr|_{t=0}=U_{1,0}  
  \end{array}
  \right.\\
  \label{limit 2} \left\lbrace
  \begin{array}{l}
  \partial_tU_2 + \mathcal{Q}\left( U_2,U_2 \right) -\mathbb{D} U_2 =0\\
  \dive u_2=0\\
  U_2\Bigr|_{t=0}=U_{2,0}  .
  \end{array}
  \right.
  \end{align}
 Subtracting \eqref{limit 2} from \eqref{limit 1} and setting $U=U_1-U_2$ we obtain the following system
 \begin{equation}
 \label{limit 1 - limit 2}
 \left\lbrace
  \begin{array}{l}
  \partial_tU + \mathcal{Q}\left( U_1,U \right) +\mathcal{Q}\left( U,U_2 \right) -\mathbb{D} U =0\\
  \dive u=0\\
  U\Bigr|_{t=0}=U_{0} = U_{1,0}-U_{2,0}.
  \end{array}
  \right.
 \end{equation}
 We apply now a stability result proved by M. Paicu in \cite{paicu_NS_periodic}, namely Proposition \ref{uniqueness anisotropic NS}, to the system \eqref{limit 1 - limit 2}. This gives the following estimate
 \begin{multline*}
 \left\| U \right\|_{H^{0,-\frac{1}{2}}}^2 + c \int_0^t \left\| \nh U \left( \tau \right) \right\|_{H^{0,-\frac{1}{2}}}^2  \d\tau\\
 \leqslant C \left\| U_0 \right\|_{H^{0,-\frac{1}{2}}}^2 \exp \left\lbrace
 \int_0^t \left( 1+ \left\| U\left( \tau \right) \right\|_\cPHs^2 \right)\left\| \nh  U\left( \tau \right) \right\|_\cPHs^2 \d\tau\right.
 +
 \int_0^t \left( 1+ \left\| U_1\left( \tau \right) \right\|_\cPHs^2 \right)\left\| \nh  U_1\left( \tau \right) \right\|_\cPHs^2 \d\tau\\
 +\left.
 \int_0^t \left( 1+ \left\| U_2\left( \tau \right) \right\|_\cPHs^2 \right)\left\| \nh  U_2\left( \tau \right) \right\|_\cPHs^2 \d\tau
 \right\rbrace.
 \end{multline*}
 The argument of the exponential is indeed uniformly bounded thanks to the estimates on the limit system performed above, whence if $ \left\| U_0 \right\|_{H^{0,-\frac{1}{2}}}^2 $ is small the whole right hand side of the above equation if small. Since moreover
 $$
 \left\| U \right\|_{H^{0,s}}^2 + c \int_0^t \left\| \nh U \left( \tau \right) \right\|_{H^{0,s}}^2  \d\tau \leqslant C
 \left( \left\| U_0 \right\|_\cPHs^2  \right)
 ,
 $$
 uniformly in $ t $ by interpolation we prove the assertion stated above.
 \hfill$\Box$
  
\section{Convergence of the system as $\varepsilon\to 0$.} \label{app S. method}

\begin{rem}
We point out the fact that  Proposition \ref{uniqueness anisotropic NS} can be applied as well to systems  with the form
\begin{align*}
\partial_t w + \mathcal{Q}^\varepsilon\left(w,w\right) +\mathcal{Q}^\varepsilon\left( u,w \right) - a_h\left( D \right) w =&f, & \dive w=& 0.
\end{align*}
\fine
\end{rem}

\begin{rem}
In the present section our aim is to use Proposition \ref{propagation norms NS horizontal} and \ref{uniqueness anisotropic NS} to the systems \eqref{filtered systemP} and \eqref{lim syst}. Let us compare these two systems with \eqref{NS horozontal}: the only structural difference  between these two is that in \eqref{filtered systemP} and \eqref{lim syst} the Poincar\'e semigroup couples velocity field and temperature $ v^\varepsilon, T^\varepsilon $ in a new variable $ U^\varepsilon $, but the structure itself of the equation is unchanged. For this reason Propositions  \ref{propagation norms NS horizontal} and \ref{uniqueness anisotropic NS} can be applied in the present case.
\end{rem}

We shall require as well the following result 
\begin{lemma}\label{lem:comparison_Sob}
Let $ f\in H^{s, s'}, \ s, s'\in \R $ such that the horizontal average $ \underline{f}\in H^{s'}_v $. Than
\begin{equation*}
\left\| \underline{f} \right\|_{H^{s'}_v} \leqslant  \left\| f \right\|_{H^{s, s'}}.
\end{equation*}
\end{lemma}
\begin{proof}
Since the element $ \underline{f} $ is the horizontal average of the function $ f $ we can indeed argue that
\begin{equation*}
\underline{f} \left( x_3 \right) = \mathcal{F}^{-1}_v \left( \left( \hat{f} \left( 0, n_3 \right) \right)_{n_3} \right),
\end{equation*}
at least in $ L^2 $.
Whence calculating explicitly the Sobolev norms
\begin{align}
\left\| \underline{f} \right\|_{H^{s'}_v}^2 = & \ \sum_{n_3\in \mathbb{Z}} \left( 1+n_3^2 \right)^{s'} \left| \hat{f} \left( 0, n_3 \right) \right|^2,\label{eq:Sob_vert}\\
\left\| f \right\|_{H^{s, s'}}^2 = & \ \sum_{n\in \mathbb{Z}^3}
 \left( 1+\left| n_h \right|^2 \right)^{s} \left( 1+n_3^2 \right)^{s'} \left| \hat{f} \left( n_h , n_3 \right) \right|^2, \label{eq:Sob_aniso}
\end{align}
Comparing the expressions in \eqref{eq:Sob_vert} and \eqref{eq:Sob_aniso} we remark that \eqref{eq:Sob_vert} is the restriction of \eqref{eq:Sob_aniso} on the fiber $ \set{ n_h=0 } $, concluding.
\end{proof}

\begin{rem}
Let us recall that Theorem \ref{thm:local_ex_strong_solutions}  implies that for each $ \varepsilon >0 $ fixed there exists a maximal time $ T^\star_\varepsilon \leqslant \infty $ such that for each $ T^\star<T^\star_{\varepsilon} $ and $ s>1/2 $ the function $ U^\varepsilon $ belongs to the space
\begin{align*}
U^\varepsilon \in L^\infty \left( \left[0,T^\star\right]; \cPHs \right),
&&
\nh U^\varepsilon \in L^2 \left( \left[0,T^\star\right]; \cPHs \right).
\end{align*}
\fine
\end{rem}

We prove that, given $V^\varepsilon_0\in\cPHs, s>1$, the solution of our filtered system \eqref{filtered systemP} converges to the solutions of the limit system \eqref{nonosc limit}, \eqref{osc limit}  in the sense that
\begin{align*}
\lim_{\varepsilon\to 0} \left( V^\varepsilon-\Lplus U\right)&=0 & \text{in } \mathcal{C}\left(\mathbb{R}_+; H^{0,\sigma}\right) \\
\lim_{\varepsilon\to 0} \nh\left( V^\varepsilon-\Lplus U\right)&=0 & \text{in } L^2\left(\mathbb{R}_+; H^{0,\sigma}\right)
\end{align*}
for $\sigma\in \left[1,s\right)$, where $U=U_{\osc}+U_{\QG}$ and $U_{\QG}=\definizioneVQG$ with $\Omega$ solution of \eqref{nonosc limit}. To do so we will use a method introduced by S. Schochet in \cite{schochet} in the framework of hyperbolic systems which allows us to deal with bulk forces which present  penalization. A suitable change of variable has to be performed so that the singular perturbations cancels among themselves. The same method has been studied in a wide generality by I. Gallagher in \cite{gallagher_schochet} in the generic context of parabolic (nonlinear) equations with singular, linear, skew-symmetric perturbation. We mention as well the works \cite{grenierrotatingeriodic} and \cite{paicu_rotating_fluids} in which such technique has been used.\\
We want to underline a major difference between the application of Schochet method in the present work and in the work  \cite{paicu_rotating_fluids}. In \cite{paicu_rotating_fluids} in fact the convergence takes place for the values of $ \sigma $ between $ 1/2 $ and $ s $. Indeed in our case $\sigma\in \left[1,s\right)$. This difference is motivated by the fact that our limit system is globally well posed in $ \cPHs, s>1 $ only. This is due to the fact that we have been proving the propagation of $ \cPHs, s>0 $ data for $ \Omega $ in Proposition \ref{propagation H0s norms for Omega} and hence we have applied Lemma \ref{higher regularity VQG} to state that $ \cPHs, s>1 $ data is propagated for $ V_{\QG} $.

Let us denote $T^\star_\varepsilon$ the maximal lifespan of $U^\varepsilon$ solution of \eqref{filtered systemP} in the space $\cPHs\left(\mathbb{T}^3\right)$ with $s>1$, which exists thanks to the work \cite{paicu_NS_periodic}. Then there exists a time $T^\star_\varepsilon\geqslant T >0$ such that $U^\varepsilon \in \mathcal{C}\left( [0,T]; \cPHs \right)$ and $\nh U^\varepsilon \in L^2 \left( [0,T]; \cPHs \right)$ uniformly in $\varepsilon$ small enough. Let us define $W^\varepsilon = U^\varepsilon -U$ defined on the interval $\left[0,T^\star_\varepsilon\right]$ taking values in $\cPHs$. We obtain that $W^\varepsilon$ satisfies the following equation
\begin{equation}\label{system Wvarepsilon}
\left\lbrace
\begin{aligned}
&
\begin{multlined}
 \partial_t W^\varepsilon + \mathcal{Q}^\varepsilon\left(W^\varepsilon,W^\varepsilon \right)+ \widetilde{\mathcal{Q}}^\varepsilon \left(U,W^\varepsilon\right)-\mathbb{D}^\varepsilon W^\varepsilon \\
= -\left( \mathbb{D}^\varepsilon-\mathbb{D}\right)U - \bigl( \mathcal{Q}^\varepsilon\left(U,U\right)- \mathcal{Q} \left(U,U\right) \bigr),
\end{multlined}
\\
& \dive w^\varepsilon=0,\\
& \bigl. W^\varepsilon \bigr|_{t=0}=0,
\end{aligned}
\right.
\end{equation}
where the form $\tilde{\mathcal{Q}}^\varepsilon$ is symmetric, bilinear  and defined via
$$
\widetilde{\mathcal{Q}}^\varepsilon\left(A,B\right)= \mathcal{Q}^\varepsilon\left(A,B\right) + \mathcal{Q}^\varepsilon\left(B,A\right).
$$

Let us define $R^\varepsilon_{\osc}\left( U\right)= \mathcal{Q}^\varepsilon\left(U,U\right)- \mathcal{Q} \left(U,U\right)$, where $\mathcal{Q}^\varepsilon\left(A,B\right)\xrightarrow[\mathcal{D}'\left(\mathbb{R}_+\times \mathbb{T}^3\right)] {\varepsilon\to 0} \mathcal{Q}\left(A,B\right)$. It is it a strongly oscillating in time function, given by the formula
$$
R^\varepsilon_{\osc} (U) = \mathcal{F}^{-1}\left( \sum_{\substack{\omega^{a,b,c}_{k,n-k,n}\neq 0\\
1\leqslant j\leqslant 3}} e^{i\frac{t}{\varepsilon}\omega^{a,b,c}_{k,n-k,n}} \left( \left.  U^{a,j}(k)\left(n_j-k_j\right) U^b\left( n-k \right)\right| e^c(n)\right)_{\mathbb{C}^4} e^c(n)
\right),
$$
where we have been using the notation $\omega^{a,b,c}_{k,n-k,n}= \omega^a(k) +\omega^b(n-k) - \omega^c(n)$, $a,b,c\in \left\lbrace \pm \right\rbrace$, $\omega^\pm (n)$ defined as in \eqref{eigenvalues},  $U^a(k)= \left( \left.\hat{U}(k)\right| e^a(k)\right)e^a(k)$ and $U^{a,j}$ is the $j$-th component of $U^a$.\\
As well the function $ S^\varepsilon_{\osc}= \left(\mathbb{D}^\varepsilon -\mathbb{D}\right)U$ is  a highly oscillating function given by the following formula
\begin{equation*}
S^\varepsilon_{\osc} \left( U\right) =\mathcal{F}^{-1}\left(\sum_{\omega^{a,b}_n \neq 0} e^{i\frac{t}{\varepsilon} \omega^{a,b}_n} \left( \left. \mathbf{D} (n) U^b (n)  \right| e^a(n)\right)_{\mathbb{C}^4} e^a(n)\right),
\end{equation*}
and as well as $R^\varepsilon_{\osc}$ even $S^\varepsilon_{\osc}\to 0$ as $\varepsilon\to 0$ only in $\mathcal{D}'$. For the rest of the section when we write the scalar product $ \left(\left. \cdot \right| \cdot  \right) $ we implicitly mean $ \left(\left. \cdot \right| \cdot  \right)_{\mathbb{C}^4} $.\\

We decompose $R^\varepsilon_{\osc}$ and $S^\varepsilon_{\osc}$ in high and low frequencies, i.e.
\begin{align*}
R^{\varepsilon,N}_{{\osc},\LF} (U) = & \mathcal{F}^{-1}\left(1_{\left\lbrace \left| n\right| \leqslant N\right\rbrace }
 \sum_{\substack{\omega^{a,b,c}_{k,n-k,n}\neq 0\\
1\leqslant j\leqslant 3}} e^{i\frac{t}{\varepsilon}\omega^{a,b,c}_{k,n-k,n}} 1_{\left\lbrace \left| k\right| \leqslant N\right\rbrace } \left( \left.   U^{a,j}(k)\left(n_j-k_j\right) U^b\left( n-k \right)\right| e^c(n)\right) e^c(n)
\right),\\
S^{\varepsilon,N}_{{\osc},\LF}\left( U\right)=& 
\mathcal{F}^{-1} \left(
1_{\left\lbrace \left| n\right| \leqslant N\right\rbrace }\sum_{\omega^{a,b}_n \neq 0}
e^{i\frac{t}{\varepsilon}\omega^{a,b}_n}
 \left( \left. \mathbf{D} (n) U^b (n) \right| e^a(n)\right) e^a(n)\right),
\end{align*}
and
\begin{align*}
R^{\varepsilon,N}_{{\osc},\HF} (U)=& R^\varepsilon_{{\osc}} (U) - R^{\varepsilon,N}_{{\osc},\LF} (U)\\
S^{\varepsilon,N}_{{\osc},\HF} (U) =& S^\varepsilon_{{\osc}}(U)-S^{\varepsilon,N}_{{\osc},\LF}(U).
\end{align*}
Indeed the subscript $ f_\HF $ stands for high frequencies and the subscript $ f_\LF $ stands for low frequencies. \\
Concerning the high frequencies terms the following lemma hold
\begin{lemma}\label{convergence zero hi-freq}
If \;$N\to \infty$ the terms $R^{\varepsilon,N}_{{\osc}, \HF} (U), S^{\varepsilon,N}_{{\osc}, \HF} (U)$ tend uniformly to 0 in $\varepsilon$ respectively in the space $L^p\left( [0,T]; H^{-1,-1/2}\right)$ and $L^p \left([0,T]; H^{-1,s}\right)$ for all $1\leqslant p\leqslant 2$, $s>1$.
\end{lemma}

\noindent
The proof of Lemma \ref{convergence zero hi-freq} is postponed to the end of the section for the sake of clarity.\\
The term $R^{\varepsilon, N}_{{\osc},\LF} (U)$ tends only weakly to zero. In order to absorb it in the following computations we introduce the following notation
\begin{align*}
\widetilde{R}^{\varepsilon, N}_{{\osc},\LF} ( U) =& \mathcal{F}^{-1}\left(1_{\left\lbrace \left| n\right| \leqslant N\right\rbrace }
 \sum_{\substack{\omega^{a,b,c}_{k,n-k,n}\neq 0\\
1\leqslant j\leqslant 3}} \frac{e^{i\frac{t}{\varepsilon}\omega^{a,b,c}_{k,n-k,n}}}{\omega^{a,b,c}_{k,n-k,n}}
 1_{\left\lbrace \left| n\right| \leqslant N\right\rbrace } \left( \left.  U^{a,j}( t , k)\left(n_j-k_j\right) U^b\left(t,  n-k \right)\right| e^c(n)\right) e^c(n)
\right)\\
\widetilde{S}^{\varepsilon, N}_{{\osc},\LF}\left( U\right)=& 
\mathcal{F}^{-1} \left(
1_{\left\lbrace \left| n\right| \leqslant N\right\rbrace }\sum_{\omega^{a,b}_n \neq 0}
\frac{e^{i\frac{t}{\varepsilon}\omega^{a,b}_n}}{i \omega^{a,b}_n}
 \left( \left. \mathbf{D} (n) U^b (n)\right| e^a(n)\right) e^a(n)\right).
\end{align*}

We do as well the following change of unknown
\begin{equation}\label{definition psi schochet method}
\Psi^{\varepsilon,N}_\LF = W^\varepsilon +\varepsilon \left( \widetilde{R}^{\varepsilon, N}_{{\osc},\LF} (U)+\widetilde{S}^{\varepsilon,N}_{{\osc},\LF}(U)\right).
\end{equation}
Considering the substitution defined in \eqref{definition psi schochet method} into \eqref{system Wvarepsilon}, and after some algebraic manipulation we obtain that  $\PLF$ satisfies the following equation
\begin{multline}\label{equazione soddisfatta da Psi}
\partial_t  \PLF 
+\frac{1}{2} \widetilde{\mathcal{Q}}\left( \PLF, \PLF -2\varepsilon \left( \RLF (U)+\SLF(U)\right)
+2U\right)
 - \mathbb{D}^\varepsilon \PLF \\
  = \Gamma^{\varepsilon,N}\left( U\right),
\end{multline}
where 
$$
\Gamma^{\varepsilon,N}= R^{\varepsilon,N}_{{\osc}, \HF}+ S^{\varepsilon,N}_{{\osc}, \HF} + \varepsilon \Gamma^\varepsilon_N ,
$$
and 
\begin{multline}
\label{term Gamma}
\Gamma^\varepsilon_N =  \mathbb{D}^\varepsilon  \left( \RLF (U)+\SLF(U)\right)
\\ + \frac{1}{2}\widetilde{\mathcal{Q}}\left( \left( \RLF (U)+\SLF(U)\right), \varepsilon  \left( \RLF (U)+\SLF(U)\right)-2U\right)+
\left( \widetilde{R}^{\varepsilon,N,t}_{{\osc},\LF} (U)+\widetilde{S}^{\varepsilon,N,t}_{{\osc},\LF}(U)\right)
\end{multline}
and respectively
\begin{align*}
\widetilde{R}^{\varepsilon,N,t}_{{\osc},\LF}= &
\mathcal{F}^{-1}\left(1_{\left\lbrace \left| n\right| \leqslant N\right\rbrace }
 \sum_{\substack{\omega^{a,b,c}_{k,n-k,n}\neq 0\\
1\leqslant j\leqslant 3}} \frac{e^{i\frac{t}{\varepsilon}\omega^{a,b,c}_{k,n-k,n}}}{\omega^{a,b,c}_{k,n-k,n}}
 1_{\left\lbrace \left| n\right| \leqslant N\right\rbrace } \partial_t \left[\left( \left.  U^{a,j}(t, k)\left(n_j-k_j\right) U^b\left(t,  n-k \right)\right| e^c(n)\right) e^c(n)\right]
\right)\\
\widetilde{S}^{\varepsilon,N,t}_{{\osc},\LF}=& 
\mathcal{F}^{-1} \left(
1_{\left\lbrace \left| n\right| \leqslant N\right\rbrace }\sum_{\omega^{a,b}_n \neq 0}
\frac{e^{i\frac{t}{\varepsilon}\omega^{a,b}_n}}{i \omega^{a,b}_n} \  \partial_t
\left[
 \left( \left. \mathbf{D} (n) U^b ( t, n)\right| e^a(n)\right) e^a(n)\right]\right).
\end{align*}

\begin{lemma}\label{boundedness low freq}
The term $\Gamma^\varepsilon_N$ given by the relation \eqref{term Gamma} is bounded uniformly in $\varepsilon$ by a constant $C(N)$ which depend solely on $N$  in the spaces $L^p\left([0,T]; H^{-1,-1/2}\right)$ for $1\leqslant p \leqslant 2$.
\end{lemma}
\begin{proof}
The result is due to the fact that we are considering functions localized in a ball of radius $ N $ in the frequency space, hence we can gain all the regularity that we want at the price of a constant which behaves like a power of $ N $, and, in particular if $\omega^{a,b}_n, \omega^{a,b,c}_{k,n-k,n}\neq 0$ implies that
$$
\frac{1}{\left| \omega^{a,b}_n \right|}, \frac{1}{\left| \omega^{a,b,c}_{k,n-k,n} \right|} \leqslant C(N).
$$
Whence we easily obtain that $\Gamma^\varepsilon_N$ belongs to the space $L^p\left( \mathbb{R}_+, H^{-1,-1/2}\right)$ and that is uniformly bounded by a constant $C(N)$.
\end{proof}

We remark that for $\varepsilon$ sufficiently small the term $U- \varepsilon\left( \RLF (U)+\SLF(U)\right)$ has a small horizontal mean in $H^s_v$, 
whence we can apply Proposition \ref{uniqueness anisotropic NS} to equation \eqref{equazione soddisfatta da Psi} in order to obtain, for all $t\in \left[ 0, T^\star_\varepsilon\right] $ the following bound
\begin{multline}\label{boh1}
\left\| \PLF(t)\right\|^2_{H^{0,-1/2}}+ c \int_0^t\left\| \nh \PLF(\tau)\right\|^2_{H^{0,-1/2}} \d \tau
\leqslant \mathcal{C} \left( \left\| U_0 \right\|_{H^{0,s_0}} \right)\\ 
\times\left(
\left\| \PLF(0)\right\|^2_{H^{0,-1/2}} +\int_0^t  \left\| \Gamma^{\varepsilon , N}(\tau)\right\|_{H^{-1,-1/2}} \d \tau +\int_0^t  \left\| \Gamma^{\varepsilon , N}(\tau)\right\|_{H^{-1,-1/2}}^2 \d \tau
\right)\\
\times \exp \left\{ \int_0^t  \left\| \Gamma^{\varepsilon , N}(\tau)\right\|_{H^{-1,-1/2}} \d \tau\right.
 +\left. \int_0^t \left( 1+ \left\| \PLF(\tau)\right\|^2_{H^{0,s_0}}\right)
\left\|\nh\PLF(\tau)\right\|^2_{H^{0,s_0}} \d \tau \right\}.
\end{multline}

Since we want to obtain global in time solutions it is important to have $\Gamma^{\varepsilon,N}$ at the same time in both spaces $L^1\left(\mathbb{R}_+; H^{-1,-1/2}\right)$ and $L^2\left(\mathbb{R}_+; H^{-1,-1/2}\right)$.\\

\begin{itemize}

\item We remark the fact that writing the estimate \eqref{boh1} we have been using implicitly the bound
$$
\int_0^t \left( 1+ \left\| U(\tau)\right\|^2_{H^{0,s_0}}\right) \left\| \nh U(\tau)\right\|^2_{H^{0,s_0}} \d \tau 
\leqslant \frac{ C}{c}\ \widetilde{\mathcal{C}} \left( \left\| U_0 \right\|_{H^{0,s_0}}\right),
$$
for $ s_0>1 $, 
and we denoted $ \mathcal{C} \left( \left\| U_0 \right\|_{H^{H^{0,s_0}}} \right)= \exp \left\{\frac{ C}{c}\ \widetilde{\mathcal{C}} \left( \left\| U_0 \right\|_{H^{0,s_0}}\right)\right\}$. 

\item We used Lemma \ref{lem:comparison_Sob} to deduce the inequality
\begin{equation*}
\left\| \underline{\Gamma^{\varepsilon, N}} \left( \tau \right) \right\|_{H^{-1/2}_v} \leqslant
\left\| {\Gamma^{\varepsilon, N}} \left( \tau \right) \right\|_{H^{-1, -1/2}},
\end{equation*}
which has consequently be applied in order to deduce \eqref{boh1}.

\end{itemize}

Considering Lemma \ref{convergence zero hi-freq} we can say that for each $\eta > 0$   there exits a large enough $N$ such that, setting $\mathcal{X}= L^1\left(\mathbb{R}_+; H^{-1,-1/2}\right) \cap L^2\left(\mathbb{R}_+; H^{-1,-1/2}\right)$, 
$$
\left\| R^{\varepsilon,N}_{{\osc}, \HF} + S^{\varepsilon,N}_{{\osc}, \HF} \right\|_\mathcal{X}\leqslant\frac{\eta}{2},
$$
and thanks to Lemma \ref{boundedness low freq} for $\varepsilon$ sufficiently small
$$
\varepsilon  \left\| \Gamma ^{\varepsilon}_N\right\|_\mathcal{X}\leqslant\varepsilon C(N) \leqslant \frac{\eta}{2},
$$
whence we obtain that
$$
\left\| \Gamma^{\varepsilon,N} \right\|_\mathcal{X}\leqslant \eta.
$$

Thanks to the definition \eqref{definition psi schochet method} we can argue that for each $\eta >0$ and $ t<T $ time of local existence of the solutions,  there exists a $ \varepsilon_1=\varepsilon_1 \left( \eta, T \right) $ such that for each $ \varepsilon\in \left( 0, \varepsilon_1 \right) $:
$$
\left\| \PLF (t)- W^\varepsilon (t)\right\|_\Hmud^2 + c \int_0^t 
\left\| \nh\PLF (\tau)-\nh W^\varepsilon (\tau)\right\|_\Hmud^2  \d \tau\leqslant \varepsilon C(N) \leqslant\frac{\eta}{2},
$$
in the same way we can write
$$
\left\| \PLF (0) \right\|_\Hmud = \varepsilon \left\| 
\RLF\left( U_0\right) + \SLF \left( U_0 \right) \right\|_\Hmud
\leqslant \varepsilon C(N)\left\|U_0\right\|^2_\cPHs \leqslant\frac{\eta}{2}.
$$
Whence for $\varepsilon$ sufficiently small and $t\in \left[ 0, T^\star_\varepsilon \right)$ we have
\begin{equation}\label{boh2}
\left\| \PLF (t) \right\|^2_\Hmud + c \int_0^t 
\left\| \nh\PLF (\tau)\right\|_\Hmud^2  \d \tau
\leqslant C\eta \left( 1+ \exp \left\lbrace  \int_0^t \left\| \nh\PLF (\tau) \right\|^2_\Hz \left( 1+ \left\| \PLF (\tau) \right\|^2_\Hz\right) \d \tau \right\rbrace\right).
\end{equation}

We want to use now the definition of $\PLF$ given in \eqref{definition psi schochet method},  in particular this implies that $ \left\| \PLF \right\|= \left\| W^\varepsilon \right\| + \mathcal{O}_N(\varepsilon)$ for $N$ fixed. This means that $\PLF$ and $W^\varepsilon$ have the same norm up to an error which is comparable to $\varepsilon$ which is, anyway, considered to be small. Whence \eqref{boh2} gives us that
\begin{equation}\label{boh3}
\left\|W^\varepsilon (t) \right\|^2_\Hmud + c \int_0^t 
\left\| \nh W^\varepsilon (\tau)\right\|_\Hmud^2  \d \tau
\leqslant C\eta \left( 1+ \exp \left\lbrace  \int_0^t \left\| \nh W^\varepsilon (\tau) \right\|^2_\Hz \left( 1+ \left\| W^\varepsilon (\tau) \right\|^2_\Hz\right) \d \tau \right\rbrace\right).
\end{equation}

\noindent
For the real numbers $s'\in \left[ - 1/2, s\right]$ we introduce the following continuous function
\begin{equation*}
\label{eq:definition_f}
f_{\varepsilon, s'} (t) = \left\| W^\varepsilon (t) \right\|^2_{H^{0,s'}} +\int_0^t \left( 1+ \left\| W^\varepsilon (\tau) \right\|^2_{H^{0,s'}} \right) \left\| \nh  W^\varepsilon (\tau) \right\|^2_{H^{0,s'}} \d \tau.
\end{equation*}
The function $\left\| W^\varepsilon (t) \right\|^2_{H^{0,s'}} $ is defined on the interval $\left[0, T^\star_\varepsilon \right)$,  by use of \eqref{boh3} we get
\begin{equation}\label{boh4}
f_{\varepsilon, -1/2}(t) \leqslant C \eta ,
\end{equation}
for each $ t \in \left[0, T^\star_\varepsilon \right) $.\\
We consider now an $ s_0 >1 $ and the maximal time
\begin{equation*}
T^{s_0}_\varepsilon = \sup \set{ 0<t<T^\star_\varepsilon \ \Big|  \ 
f_{\varepsilon, s_0} \left( t \right)\leqslant 1 , \text{ for each } 0\leqslant t \leqslant T^{s_0}_\varepsilon
 }.
\end{equation*}
 Interpolating between $\Hmud$ and $\Hz$ we get
\begin{equation}\label{eq:est_f}
f_{\varepsilon, \sigma} \left( t \right)= \mathcal{O} \left( \eta^{\vartheta \left( s_0, \sigma \right)} \right) \leqslant 1, \hspace{1cm} t\in \left[ 0, T^{s_0}_\varepsilon\right),
\end{equation}
where $ 0<\vartheta \left( s_0, \sigma \right)\xrightarrow{\sigma\to s_0}0 $ and $ 0<\sigma \in \left[ -1/2, s_0 \right) $.\\
We consider at this point $U^\varepsilon = W^\varepsilon + U$, since $U$ has zero horizontal mean we can easily point out that
\begin{equation*}
 \underline{U^\varepsilon}(t)= \underline{W^\varepsilon}(t).
\end{equation*}
Whence using Lemma \ref{lem:comparison_Sob}, the definition of the function $ f_{\varepsilon, \sigma} $ given in \eqref{eq:definition_f}, and the smallness property on $ f_{\varepsilon, \sigma} $ given in \eqref{eq:est_f} we deduce:
\begin{align*}
\left\| \underline{W^\varepsilon}(t) \right\|_{H^\sigma_v} \leqslant & \ \left\| {W^\varepsilon}(t) \right\|_{H^{0,\sigma}}, \\
\leqslant & \  C \sqrt{f_{\varepsilon, \sigma} \left( t \right)} , \\
\leqslant & \ C \eta^{\vartheta /2} \ll 1.
\end{align*}

\noindent Since the horizontal average of $ \underline{U^\varepsilon} $ is small we can infer via Proposition \ref{propagation norms NS horizontal} obtaining, for $ \sigma \in \left( 1, s_0 \right) $;
\begin{equation}\label{boffo2}
\left\| U^\varepsilon (t) \right\|_{H^{0,s}}^2 +c \int_0^t \left\| \nh U^\varepsilon (\tau) \right\|_{H^{0,s}}^2 \d \tau 
 \leqslant C \left\| V_0 \right\|_{\cPHs} \exp \left(
\int_0^t \left( 1+ \left\| U^\varepsilon(\tau)\right\|^2_{H^{0,\sigma}}\right) \left\| \nh U^\varepsilon(\tau)\right\|^2_{H^{0,\sigma}} \d \tau \right),
\end{equation}
on the other hand  $0\leqslant t < T^{s_0}_\varepsilon $, and since $ U^\varepsilon= W^\varepsilon + U $ we get
\begin{equation}\label{boffo}
\int_0^t \left( 1+ \left\| U^\varepsilon(\tau)\right\|^2_{H^{0,\sigma}}\right) \left\| \nh U^\varepsilon(\tau)\right\|^2_{H^{0,\sigma}} \d s
 \leqslant f_{\varepsilon, \sigma}(t) +\int_0^t \left( 1+ \left\| U(\tau)\right\|^2_{H^{0,\sigma}}\right) \left\| \nh U(\tau)\right\|^2_{H^{0,\sigma}} \d \tau + F_\sigma(t).
\end{equation}
$ F_\sigma\left( t \right) $ in particular is defined as 
\begin{align*}
F_\sigma\left( t \right)= & \int_0^t \left( 1+ \left\| W^\varepsilon \left( \tau \right) \right\|_{H^{0,\sigma}}^2 \right) \left\| \nh U \left( \tau \right) \right\|_{H^{0,\sigma}}^2 \d\tau +\int_0^t \left( 1+ \left\| U \left( \tau \right) \right\|_{H^{0,\sigma}}^2 \right) \left\| \nh W^\varepsilon \left( \tau \right) \right\|_{H^{0,\sigma}}^2 \d\tau\\
\lesssim & \int_0^t \left( 1+ \left\| W^\varepsilon \left( \tau \right) \right\|_{H^{0,\sigma}}^2 \right) \left\| \nh U \left( \tau \right) \right\|_{H^{0,\sigma}}^2 \d\tau + \left( 1+ \left\| U \right\|^2_{L^\infty \left( \R_+; H^{0,\sigma} \right)} \right)f_{\varepsilon, \sigma}\left( t \right)\\
\lesssim & \ \left( \sup_{[0,t]} f_{\varepsilon, \sigma}  \right) \left\| \nh U \right\|_{L^2\left( \R_+; H^{0,\sigma} \right)}+ \left( 1+ \left\| U \right\|^2_{L^\infty \left( \R_+; H^{0,\sigma} \right)} \right)f_{\varepsilon, \sigma}\left( t \right),
\end{align*}
which in turn implies that, considering the above estimate in \eqref{boffo},
\begin{equation*}
\begin{aligned}
\int_0^t \left( 1+ \left\| U^\varepsilon(\tau)\right\|^2_{H^{0,\sigma}}\right) \left\| \nh U^\varepsilon(\tau)\right\|^2_{H^{0,\sigma}} \d s 
 \leqslant&\ f_{\varepsilon, \sigma}(t)
 +\int_0^t \left( 1+ \left\| U(\tau)\right\|^2_{H^{0,\sigma}}\right) \left\| \nh U(\tau)\right\|^2_{H^{0,\sigma}} \d \tau \\
& +\left( \sup_{[0,t]} f_{\varepsilon, \sigma}  \right) \left\| \nh U \right\|_{L^2\left( \R_+; H^{0,\sigma} \right)}+ \left( 1+ \left\| U \right\|^2_{L^\infty \left( \R_+; H^{0,\sigma} \right)} \right)f_{\varepsilon, \sigma}\left( t \right).
\end{aligned}
\end{equation*}
We have seen though that in $ \left[0, T^{s_0}_\varepsilon\right) $ that $ f_{\varepsilon, \sigma}\left( t \right)\leqslant 1  $ for $ \sigma\in \left( -1/2, s_0 \right) $, and since $ U\in L^\infty \left( \R_+, H^{0,\sigma} \right) $ and $ \nh U\in L^2 \left( \R_+, H^{0,\sigma} \right) $ for $ \sigma\in \left(1, s_0\right] $ (this is simply Proposition \ref{propagation H0s norms for Omega} combined with Lemma \ref{higher regularity VQG}),  we obtained that
\begin{align*}
\int_0^t \left( 1+ \left\| U^\varepsilon(\tau)\right\|^2_{H^{0,\sigma}}\right) \left\| \nh U^\varepsilon(\tau)\right\|^2_{H^{0,\sigma}} \d s \leqslant C.
\end{align*}
If we consider the above bound in \eqref{boffo2} we have hence obtained that 
$$
\left\| U^\varepsilon (t) \right\|_{H^{0,s}}^2 +c \int_0^t \left\| \nh U^\varepsilon (s) \right\|_{H^{0,s}}^2 \d s \leqslant C,
$$
for all times $t\in \left[ 0, T^{s_0}_\varepsilon\right)$ and $ s>1 $. We deduce that $T_\varepsilon^{s_0}=T^\star_\varepsilon$ and since the constant $ C $ is independent of the time $ t $,  this implies that $U^\varepsilon(t)$ can be extended in $\cPHs$ beyond $T^\star_\varepsilon$ and hence we obtain that $T^\star_\varepsilon =\infty$ as long as $\varepsilon $ is sufficiently small. Recalling that $\left\| W^\varepsilon \right\|=o(1)$ in $\left[0, T^\star_\varepsilon \right)$ we deduce that $U^\varepsilon \to U$ globally in time  in $H^{0,\sigma}$ for $-1/2\leqslant \sigma <s$.\\

\textit{Proof of Lemma \ref{convergence zero hi-freq}.} In the following the index $ s $ addressing to the anisotropic Sobolev space  $ \cPHs $ is always considered to be $ s>1 $. An interesting feature is that if $ s>1/2 $ then $ H^s_v $ is a Banach algebra. We shall use this property all along the proof. We perform at first the estimates for the term $R^{\varepsilon, N}_{{\osc}, \HF}$.  Since $U(t)$ is of zero horizontal average for all $t>0$ and $\nh U \in \andue$ we obtain that $U\in \andue$. Consequently $U\in \anuno \cap \andue$, and, interpolating $U\in L^{p'}\left( \mathbb{R}_+, \cPHs\right)$ for each $p'\in \left[2,\infty\right]$.\\
Let us observe that the term $R^{\varepsilon,N}_{{\osc}, \HF}$ can be decomposed as
$$
R^{\varepsilon,N}_{{\osc}, \HF}= R^{\varepsilon,N}_{{\osc},1}+ R^{\varepsilon,N}_{{\osc},2},
$$
where we denoted 
$$
R^{\varepsilon,N}_{{\osc},1}= \mathcal{F}^{-1}\left( 1_{\left\lbrace |n|\geqslant N \right\rbrace} R^{\varepsilon,N}_{{\osc}, \HF}\right),
$$
and
$$
R^{\varepsilon,N}_{{\osc},2} (U) =  \mathcal{F}^{-1}\left(1_{\left\lbrace \left| n\right| \leqslant N\right\rbrace }
 \sum_{\substack{\omega^{a,b,c}_{k,n-k,n}\neq 0\\
1\leqslant j\leqslant 3}} e^{i\frac{t}{\varepsilon}\omega^{a,b,c}_{k,n-k,n}} 1_{\left\lbrace \left| k\right| \geqslant N\right\rbrace } \left( \left.  U^{a,j}(k)\left(n_j-k_j\right) U^b\left( n-k \right)\right| e^c(n)\right) e^c(n).
\right)
$$

For the first term we use the fact that we are on the high frequencies and of an element in $L^p \left( \mathbb{R}_+ , H^{-1,-1/2}\right)$ which tends uniformly at zero as long as $\varepsilon \to 0$ thanks to Lebesgue theorem and Sobolev embeddings. In fact
\begin{align*}
\left\| R^\varepsilon_{\osc} \left( U \right) \right\|_{H^{-1, -1/2}} \leqslant & 
\left\| \mathcal{F}^{-1} \left( \sum_{k+m=n} \left(\left. U^a\left( k \right)\otimes U^b\left( m \right) \right| e^c\left( n \right)  \right)_{\mathbb{Z}^4} e^c \left( n \right) \right) \right\|_{H^{0,1/2}}\\
= & \left\| U \otimes U \right\|_{H^{0,1/2}}\\
\lesssim & \left\| U \right\|_{H^{1/2, s}}^2.
\end{align*}
Now, since $ U $ has null horizontal average we can apply Lemma \ref{lemma:sobolev_embedding_zero_average} to obtain finally that
$$
\left\| R^\varepsilon_{\osc} \left( U \right) \right\|_{H^{-1, -1/2}} \lesssim \left\| U \right\|_\cPHs \left\| \nh U \right\|_\cPHs.
$$
Since $L^2 \left( \left[0,T\right] \right) \subset L^{p'}\left( \left[0,T\right] \right)   $ for $ p'\in\left[ 1, 2 \right) $ if we prove that $ \left\| R^\varepsilon_{\osc} \left( U \right) \right\|_{ L^2 \left( \left[0,T\right]; H^{-1, -1/2}\right)}<\infty $ we can apply Lebesgue theorem and conclude that $ \left\| R^{\varepsilon, N}_{{\osc}, 1} \right\|_{ L^{p'} \left( \left[0,T\right]; H^{-1, -1/2}\right)}\to 0 $ as $ N \to\infty $. But this is in fact true since
\begin{equation}
\left\| \left\| U \right\|_\cPHs \left\| \nh U \right\|_\cPHs \right\|_{L^2_t}^2
= \int_0^t\left\| U \left( \tau \right) \right\|_\cPHs^2 \left\| \nh U \left( \tau \right) \right\|_\cPHs^2 \d \tau\leqslant \left\| U \right\|_{L^\infty \left( \R_+; \cPHs \right)}^2 \left\| \nh U \right\|_{L^2 \left( \R_+; \cPHs \right)}^2.
\end{equation}

For the second term we argue as follows
\begin{align*}
\left\| R^{\varepsilon,N}_{{\osc},2}\right\|_{H^{-1,-1/2}} \leqslant & \left\| \mathcal{F}^{-1} \left( \sum_{k+m=n} 1_{\{|k|\geqslant N \}} \left( \left. \left( U^a(k)\otimes U^b(m)\right)\right| e^c(n)\right) e^c(n) \right)\right\|_{H^{0,1/2}}\\
\leqslant &
\left\| \mathcal{F}^{-1} \left( \sum_{k+m=n} 1_{\{|k|\geqslant N \}} \left( \left. \hat{U}(k)\right| e^a(k)\right) e^a(k) \right) \right\|_{H^{1/2, s}} \left\|u \right\|_{H^{1/2, s}},
\end{align*}
and, using  \eqref{GN type ineq2} we obtain the following bound
\begin{equation*}
\left\| R^{\varepsilon,N}_{{\osc},2}\right\|_{H^{-1,-1/2}}
 \leqslant \left\| \mathcal{F}^{-1} \left( 1_{\{|k|\geqslant N \}} U^a(k)\right) \right\|_\cPHs^{1/2} 
\left\| \mathcal{F}^{-1} \left( 1_{\{|k|\geqslant N \}} \left(\nh U\right)^a(k)\right) \right\|_\cPHs^{1/2} \left\| U\right\|_\cPHs^{1/2} \left\| \nh U\right\|_\cPHs^{1/2},
\end{equation*}
which evidently tends to zero thanks to Lebesgue theorem.\\
For the term $S^{\varepsilon, N}_{{\osc},\HF}$ it comes straightforward since
\begin{equation}
\left\| S^{\varepsilon,N}_{{\osc},\HF} \right\|_{H^{-1,s}} = \left\| 
\mathcal{F}^{-1} \left(
1_{\left\lbrace \left| n\right| \geqslant N\right\rbrace }\sum_{\omega^{a,b}_n \neq 0}
e^{i\frac{t}{\varepsilon}\omega^{a,b}_n}
 \left( \left. \mathbf{D} (n) U^b (n) \right| e^a(n)\right) e^a(n)\right) \right\|_{H^{-1,s}}
  \leqslant C \left\| \nh U\right\|_{\cPHs}.
\end{equation}
\hfill$\Box$

\section{The energy estimates} \label{Section energy estimates}
In this Section we refer to $V_{\QG}$ and $U_{\osc}$ respectively as the solution of equation \eqref{quasi geostrophic equation} and \eqref{osc limit}. Moreover $v_{\QG},u_{\osc}$ represents the projection of the first three components of $V_{\QG}$ and $U_{\osc}$.\\
The aim of this section is essentially to give an energy bound for the bilinear term appearing in equation \eqref{osc limit}. \\
Given a generic vector field $u$ we refer to $\underline{u}$ as the horizontal average of $u$. This gives the natural decomposition $u=\underline{u}+\tilde{u}$. Since $\tilde{u}$ has zero horizontal average the results given in the Subsection \ref{basic estimates} can be applied.

\subsection{Estimates for the global well-posedness of the limit system.}

\begin{prop}
Let $\definizioneVQG$ where $\Omega$ is the potential vorticity defined in \eqref{potential vorticity}, then
\begin{multline}\label{stima forma bilineare omega}
\left( \left.{\cPtv} \left( v^h_{\QG} \cdot \nh \Omega\right) \right| {\cPtv} \Omega\right)\leqslant 
C  2^{-2qs} b_q( t)\\
\times \left[
 \left\| \Omega \right\|_{\cPLtwo}^{1/2} \left\| \nh\Omega \right\|_{\cPLtwo}^{1/2} \left\| \Omega \right\|_{\cPHs}^{1/2} \left\| \nh \Omega \right\|_{\cPHs}^{3/2}
  +  \left\| \nh \Omega \right\|_\cPLtwo  \left\|\Omega \right\|_{\cPHs} \left\|\nh \Omega \right\|_{\cPHs}\right],
\end{multline} 
where $ \left( b_q \right)_q $ is a $ \ell^1 \left( \mathbb{Z} \right) $ positive sequence which depends on $ \Omega $ and such that $ \sum_q b_q \left( t \right)\leqslant 1 $. 
\end{prop}
\begin{proof}
Thanks to Bony decomposition \eqref{bony decomposition asymmetric} we can write
\begin{multline} \label{asymmetric Bony decomp Omega}
{\cPtv} \left( v^h_{\QG} \cdot \nh \Omega\right) =S^v_{q-1}v^h_{\QG} {\cPtv} \nh \Omega +
\\ \sum_{\left|q-q'\right|\leqslant 4}  \left(
\left[ {\cPtv};\cPSvq v^h_{\QG}\right] \cPTv \nh \Omega + \left( \cPSvq v^h_{\QG} -S^v_{q-1} v^h_{\QG}\right){\cPtv}\cPTv \nh \Omega
\right)
+ \sum_{q'>q-4}{\cPtv} \left( S^v_{q'+2} \nh \Omega \cPTv v^h_{\QG}   \right),
\end{multline}
and hence we can decompose  
$
\left( \left. {\cPtv} \left( v^h_{\QG} \cdot \nh \Omega\right) \right| {\cPtv} \Omega \right) = \sum_{k=1}^4 I^k_h \left( q \right).
$

First of all, since $\dive_h v^h_{\QG}=0$ we have 
$
I^1_h=0.
$ 
We remark that we proved in Lemma \ref{propagation horizontal mean VQG} that  $\underline{v^h_{{\QG}}}=0$.  Whence $v_{\QG}^h=\tilde{v}_{\QG}^h$. Moreover $\nh \Omega =\nh \tilde{\Omega}$, hence
\begin{multline*}
I^2_h\left( q \right)=\sum_{\left|q-q'\right|\leqslant 4}  \left( \left.
\left[ {\cPtv};\cPSvq v^h_{\QG}\right] \cPTv \nh \Omega \right| {\cPtv} \Omega\right)\\
= 
\sum_{\left|q-q'\right|\leqslant 4}  \left( \left.
\left[ {\cPtv};\cPSvq \tilde{v}^h_{\QG}\right] \cPTv \nh \tilde{\Omega} \right| {\cPtv} \tilde{\Omega}\right)+\left( \left.
\left[ {\cPtv};\cPSvq \tilde{v}^h_{\QG}\right] \cPTv \nh \tilde{\Omega} \right| {\cPtv} \underline{\Omega}\right)
= I^{2,1}_h\left( q \right)+I^{2,2}_h\left( q \right).
\end{multline*}
We consider first the term $I^{2,1}_h$. By H\"older inequality and Lemma \ref{estimates commutator}  we can deduce
\begin{align*}
I^{2,1}_h\left( q \right) 
\lesssim & \sum_{\left|q-q'\right|\leqslant 4} 2^{-q}\left\| \cPSvq \partial_3 \tilde{v}^h_{\QG} \right\|_{ L^\infty_v L^4_h}\left\| \cPTv \nh \tilde{\Omega}\right\|_\cPLtwo \left\| {\cPtv} \tilde{\Omega}\right\|_{L^2_vL^4_h}
\end{align*}
we can hence apply \eqref{eq:LinfL4embedding} to the term $ \left\| \cPSvq \partial_3 \tilde{v}^h_{\QG} \right\|_{ L^\infty_v L^4_h} $ and \eqref{GN type ineq} to $ \left\| {\cPtv} \tilde{\Omega}\right\|_{L^2_vL^4_h} $, and then \eqref{eq:reg_dyadic_blocks} and Lemma \ref{higher regularity VQG} in order to deduce
\begin{align}
I^{2,1}_h \left( q \right) \lesssim & \sum_{\left|q-q'\right|\leqslant 4} 2^{-q+q'/2}\left\| \partial_3 \tilde{v}^h_{\QG} \right\|_\cPLtwo^{1/2} \left\| \partial_3 \nh \tilde{v}^h \right\|_\cPLtwo^{1/2}\left\| \cPTv \nh \tilde{\Omega}\right\|_\cPLtwo \nonumber \\
& \times \left\| {\cPtv} \tilde{\Omega}\right\|_\cPLtwo^{1/2} \left\| {\cPtv} \nh \tilde{\Omega}\right\|_\cPLtwo^{1/2}\nonumber
 \\
\lesssim & b_q\left( t \right) 2^{-q/2-2qs} \left\| \Omega \right\|_{\cPLtwo}^{1/2} \left\| \nh \Omega \right\|_{\cPLtwo}^{1/2} \left\| \Omega \right\|_{\cPHs}^{1/2} \left\| \nh \Omega \right\|_{\cPHs}^{3/2}.\label{est:I21h}
\end{align}

For the following terms the tools used are the same as for the term $I^{2,1}_h\left( q \right)$, hence,  we shall not explain the procedure in  details.
For the term $I^{2,2}_h\left( q \right)$ 
\begin{align}
I^{2,2}_h \left( q \right) \lesssim & 2^{-q} \left\| \cPSvq \partial_3 \tilde{v}_{\QG}^h \right\|_{ L^\infty_v L^2_h} \left\| \cPTv \nh \tilde{\Omega}\right\|_\cPLtwo \left\| {\cPtv} \underline{\Omega}\right\|_{L^2_v}
\nonumber\\
\lesssim & b_q\left(  t \right) 2^{-2qs-q'/2} \left\| \nh \partial_3 V_{\QG} \right\|_{\cPLtwo} \left\|\Omega \right\|_{\cPHs} \left\|\nh \Omega \right\|_{\cPHs}
\nonumber\\
\lesssim & b_q\left(  t \right) 2^{-2qs-q'/2} \left\| \nh\Omega \right\|_{\cPLtwo} \left\|\Omega \right\|_{\cPHs} \left\|\nh \Omega \right\|_{\cPHs},\label{est:I22h}
\end{align}
where in the first inequality we have used \eqref{eq:LinfL2embedding} and by Poincar\'e inequality in the horizontal variable to obtain
$$
\left\| \cPSvq \partial_3 \tilde{v}_{\QG}^h \right\|_{L^\infty_v L^2_h} \lesssim 2^{q'/2} \left\| \cPSvq \nh \partial_3\tilde{v}_{\QG}^h \right\|_\cPLtwo 
$$

Next, we consider the term 
\begin{multline*}
I^3_h\left( q \right)
=\sum_{|q-q'|\leqslant 4}\left(\left.\left( \cPSvq v^h_{\QG} -S^v_{q-1} v^h_{\QG}\right){\cPtv}\cPTv \nh \Omega \right| {\cPtv} \Omega \right)\\
=\sum_{|q-q'|\leqslant 4}\left(\left.\left( \cPSvq \tilde{v}^h_{\QG} -S^v_{q-1} \tilde{v}^h_{\QG}\right){\cPtv}\cPTv \nh \tilde{\Omega}\right| {\cPtv} \tilde{\Omega}\right)
\\
+\left(\left.\left( \cPSvq \tilde{v}^h_{\QG} -S^v_{q-1} \tilde{v}^h_{\QG}\right){\cPtv}\cPTv \nh \tilde{\Omega}\right| {\cPtv} \underline{\Omega}\right)
=I^{3,1}_h\left( q \right)+I^{3,2}_h\left( q \right).
\end{multline*}

With calculations similar and since $\text{Supp}\; \mathcal{F}\left( \cPSvq v^h_{\QG} -S^v_{q-1} v^h_{\QG}\right)\subset \displaystyle\bigcup_{|q-q'|\leqslant 4} 2^q \mathcal{C}$, and hence localized from above and below in the frequency space, using respectively in the first inequality \eqref{eq:LinfL4embedding}, Bernstein inequality,  \eqref{GN type ineq}, \eqref{eq:reg_dyadic_blocks} and Lemma \ref{higher regularity VQG}
\begin{align}
I^{3,1}_h \left( q \right) \leqslant & \sum_{|q-q'|\leqslant 4} \left\| \cPSvq \tilde{v}^h_{\QG} -S^v_{q-1} \tilde{v}^h_{\QG} \right\|_{ L^\infty _vL^4_h } \left\| {\cPtv}\cPTv \nh \tilde{\Omega} \right\|_\cPLtwo \left\| {\cPtv} \tilde{\Omega} \right\|_{ L^2_v L^4_h}
\nonumber\\
\lesssim &  b_q\left(  t \right)2^{-q/2 -2qs}\left\| \partial_3 V_{\QG} \right\|_\cPLtwo^{1/2}\left\| \nh\partial_3 V_{\QG} \right\|_\cPLtwo^{1/2} \left\|\Omega\right\|_\cPHs^{1/2}\left\|\nh\Omega\right\|_\cPHs^{3/2}
\nonumber\\
\lesssim &  b_q\left(  t \right)2^{-q/2 -2qs}\left\| \Omega \right\|_\cPLtwo^{1/2}\left\| \nh\Omega \right\|_\cPLtwo^{1/2} \left\|\Omega\right\|_\cPHs^{1/2}\left\|\nh\Omega\right\|_\cPHs^{3/2}
\end{align}
The procedure for the term $I^{3,2}_h\left( q \right)$ is almost the same as the one for the term $I^{3,1}_h\left( q \right)$, except that we do not use  \eqref{GN type ineq} and we use Poincar\'e inequality in the horizontal variables
\begin{align}
 I^{3,2}_h\left( q \right) \leqslant b_q\left( t \right) 2^{-q/2-2qs} \left\|  \nh \Omega \right\|_{\cPLtwo}\left\| \Omega \right\|_{\cPHs} \left\| \nh \Omega \right\|_{\cPHs}.
\end{align}
The last term
\begin{multline*}
I^4_h\left( q \right)= \sum_{q'>q-1} \left(\left.{\cPtv} \left( S^v_{q'+2} \nh \Omega \cPTv v^h_{\QG} \right) \right| {\cPtv} \Omega\right)
\\
= \sum_{q'>q-1} \left(\left.{\cPtv} \left( S^v_{q'+2} \nh \Omega \cPTv v^h_{\QG} \right) \right| {\cPtv} \tilde{\Omega}\right) +
\sum_{q'>q-1} \left(\left.{\cPtv} \left( S^v_{q'+2} \nh \Omega \cPTv v^h_{\QG} \right) \right| {\cPtv} \underline{\Omega}\right)
 =I^{4,1}_h\left( q \right)+I^{4,2}_h\left( q \right).
\end{multline*}
Let us deal with the term $ I^{4,1}_h\left( q \right) $. Applying H\"older inequality we deduce 
\begin{align*}
I^{4,1}_h\left( q \right) \leqslant & \sum_{q'>q-1} \left\| \cPTv v^h _{{\QG}} \right\|_{L^\infty_vL^4_h} \left\| S^v_{q'+2}\nh \Omega \right\|_{\cPLtwo} \left\| {\cPtv} \tilde{\Omega} \right\|_{L^2_v L^4_h}.
\end{align*}
Using Bernstein inequality twice,\eqref{GN type ineq}, Lemma \ref{higher regularity VQG}   and lastly \eqref{eq:reg_dyadic_blocks} we deduce
\begin{align*}
\left\| \cPTv v^h _{{\QG}} \right\|_{L^\infty_vL^4_h} \lesssim & \ 2^{q'/2} \left\| \cPTv v^h _{{\QG}} \right\|_{L^2_vL^4_h}\\
\lesssim & \ 2^{-q'/2} \left\| \partial_3 \cPTv v^h _{{\QG}} \right\|_{L^2_vL^4_h}\\
\lesssim &  2^{-q'/2} \left\| \partial_3 \cPTv v^h _{{\QG}} \right\|_{\cPLtwo}^{1/2}
\left\| \partial_3 \nh \cPTv v^h _{{\QG}} \right\|_{\cPLtwo}^{1/2}\\
\lesssim &  2^{-q'/2} \left\|  \cPTv \Omega \right\|_{\cPLtwo}^{1/2}
\left\|  \nh \cPTv \Omega \right\|_{\cPLtwo}^{1/2}\\
\lesssim &  c_{q'} \left( \Omega, t \right) 2^{-q'/2 -q's} \left\|   \Omega \right\|_{\cPHs}^{1/2}
\left\|  \nh  \Omega \right\|_{\cPHs}^{1/2}.
\end{align*}
An application of \eqref{GN type ineq} and \eqref{eq:reg_dyadic_blocks} gives instead
\begin{equation}\label{est:I41h}
\left\| {\cPtv} \tilde{\Omega} \right\|_{L^2_v L^4_h} \lesssim c_{q} \left( \Omega, t \right)2^{-qs}\left\|   \Omega \right\|_{\cPHs}^{1/2}
\left\|  \nh  \Omega \right\|_{\cPHs}^{1/2},
\end{equation}
whence we deduce the bound
\begin{equation*}
I^{4,1}_h\left( q \right) \leqslant C 2^{-2qs-q/2} b_q \left( \Omega, t \right)
\left\|  \nh  \Omega \right\|_{\cPLtwo}
\left\|   \Omega \right\|_{\cPHs}
\left\|  \nh  \Omega \right\|_{\cPHs}.
\end{equation*}
To bound the term $ I^{4,2}_h\left( q \right) $ is a similar procedure and hence is omitted. Whence collecting estimates \eqref{est:I21h}--\eqref{est:I41h} we deduce the bound \eqref{stima forma bilineare omega}.
\end{proof}

\begin{prop}\label{prop:bilinear_estimates_osc_part}
Let  $V_{\QG}$ and $U_{\osc}$ respectively be the solution of equation \eqref{quasi geostrophic equation} and \eqref{osc limit}, then if the horizontal mean of $V_{\QG}$ and $U_{\osc}$ is zero (see Lemmas \ref{propagation horizontal mean VQG} and \ref{propagation horizontal mean osc part}) the following estimates hold
\begin{multline}
\left( \left. {\cPtv}\mathcal{Q}\left( V_{\QG} ,  U_{\osc}\right)\right|{\cPtv} U_{\osc}\right)_{\cPLtwo} + \left( \left. {\cPtv}\mathcal{Q}\left( U_{\osc} , V_{\QG}\right)\right|{\cPtv} U_{\osc}\right)_{\cPLtwo}\\
\begin{aligned}
\leqslant & \
C  2^{-2qs}b_q\left(   t \right)\left\|\nh\Omega\right\|_\cPHs \left\| \nh U_{\osc} \right\|_{\cPHs} \left\|  U_{\osc} \right\|_{\cPHs}\\
&+C  2^{-2qs}b_q\left( t \right)\left\| \Omega \right\|^{1/2}_{\cPHs} \left\| \nh\Omega\right\|^{1/2}_{\cPHs}\left\|U_{\osc}\right\|_{\cPHs}^{1/2} \left\|\nh U_{\osc}\right\|_{\cPHs}^{3/2}
\end{aligned}
\label{prima stima termine bilineare parte oscillante}
\end{multline}
\begin{equation}
\begin{aligned}
\left( \left. {\cPtv}\mathcal{Q}\left( U_{\osc},  U_{\osc}\right)\right|{\cPtv} U_{\osc}\right)_{\cPLtwo}
\leqslant &
 C b_q \left(  t \right) 2^{-2qs}\left\| \nh  U_{\osc} \right\|_{\cPLtwo} \left\|  U_{\osc} \right\|_{\cPHs} \left\| \nh  U_{\osc} \right\|_{\cPHs} \\
&+  C b_q \left(   t \right) 2^{-2qs} \left\|  U_{\osc} \right\|_{\cPLtwo}^{1/2} \left\| \nh  U_{\osc} \right\|_{\cPLtwo}^{1/2} \left\|  U_{\osc} \right\|_{\cPHs}^{1/2} \left\| \nh  U_{\osc} \right\|_{\cPHs}^{3/2}.
\label{terza stima termine bilineare parte oscillante}
\end{aligned}
\end{equation}
  
The sequence $ \left( b_q \right)_q $ is a $ \ell^1 \left( \mathbb{Z} \right) $ positive sequence which depends on $ \Omega, U_{\osc} $ and such that $ \sum_q b_q \left( t \right)\leqslant 1 $. 
\end{prop}

\begin{rem}
From now on $\left( \cdot \left| \cdot\right.\right)=\left( \cdot \left| \cdot\right.\right)_{\cPLtwo}$
\fine
\end{rem}

\begin{proof} We shall divide the proof of the above proposition in two parts, namely one part for each estimate.\\
In the following we always consider $ s>1/2 $, hence in particular the embedding $ H^s_v\hra L^\infty_v $ holds true. Moreover we underline the fact that $ V_{\QG} \left( t \right) $ and $ U_{\osc} \left( t \right) $ have zero horizontal average for each $ t>0 $ is the initial data has zero horizontal average thanks to the results of Lemma \ref{propagation horizontal mean VQG} and Lemma \ref{propagation horizontal mean osc part}, whence the estimates \eqref{GN type ineq} and \eqref{GN type ineq2} can be applied in this context as well as Lemma \ref{lem:prod_Sob_anisotropic}.\\

\textit{Proof of \eqref{prima stima termine bilineare parte oscillante}:}
in order to prove the estimate \eqref{prima stima termine bilineare parte oscillante} we shall substitute the bilinear form $ \mathcal{Q} $  with the transport bilinear form. This choice is done only in order to simplify the notation.\\

Indeed we have
\begin{align*}
\left| \left( \left. {\cPtv} \left( v_{\QG} \cdot \nabla U_{{\osc}} \right) \right| {\cPtv} U_{\osc} \right) \right| = & \
 \left| \left( \left. {\cPtv} \left( v^h_{\QG} \cdot \nh U_{{\osc}} \right) \right| {\cPtv} U_{\osc} \right) \right|,\\
= & \ 
\left| \left( \left. \diveh {\cPtv} \left( v^h_{\QG} \otimes U_{{\osc}} \right) \right| {\cPtv} U_{\osc} \right) \right|,\\
\left| \left( \left. {\cPtv} \left( u_{\osc} \cdot \nabla V_{\QG} \right) \right| {\cPtv} U_{\osc} \right) \right| \leqslant & \
\left| \left( \left. \diveh {\cPtv} \left( u^h_{\osc} \otimes V_{\QG} \right) \right| {\cPtv} U_{\osc} \right) \right|
\\
& \ +
\left| \left( \left. \partial_3 {\cPtv} \left( u^3_{\osc} \  V_{\QG} \right) \right| {\cPtv} U_{\osc} \right) \right|,
\end{align*}
and indeed 
\begin{equation*}
\left| \left( \left. \diveh {\cPtv} \left( v^h_{\QG} \otimes U_{{\osc}} \right) \right| {\cPtv} U_{\osc} \right) \right| +
\left| \left( \left. \diveh {\cPtv} \left( u^h_{\osc} \otimes V_{\QG} \right) \right| {\cPtv} U_{\osc} \right) \right| 
\leqslant
2 \left| \left( \left.  {\cPtv} \left( U_{\osc} \otimes V_{\QG} \right) \right| {\cPtv} \nh U_{\osc} \right) \right|,
\end{equation*}
whence
\begin{multline*}
\left| \left( \left. {\cPtv} \left( v_{\QG} \cdot \nabla U_{{\osc}} \right) \right| {\cPtv} U_{\osc} \right) \right|
+
\left| \left( \left. {\cPtv} \left( u_{\osc} \cdot \nabla V_{\QG} \right) \right| {\cPtv} U_{\osc} \right) \right|
\\
\leqslant
2 \left| \left( \left.  {\cPtv} \left( U_{\osc} \otimes V_{\QG} \right) \right| {\cPtv} \nh U_{\osc} \right) \right|
+
\left| \left( \left. \partial_3 {\cPtv} \left( u^3_{\osc} \  V_{\QG} \right) \right| {\cPtv} U_{\osc} \right) \right|\\
= B_h \left( q \right) + B_v\left( q \right).\end{multline*}
Thanks to \eqref{eq:reg_dyadic_blocks} and Lemma \ref{lem:prod_Sob_anisotropic} we deduce
\begin{align}
B_h\left( q \right) \lesssim & \ 2^{-2qs} b_q \left(  t \right) \left\| U_{\osc} \otimes V_{\QG} \right\|_{\cPHs}
\left\| \nh U_{{\osc}} \right\|_{\cPHs},\nonumber
\\
\lesssim &  \ 2^{-2qs} b_q \left(  t \right) \left\| V_{\QG} \right\|_{H^{1/2, s}} \left\| U_{\osc} \right\|_{H^{1/2, s}} \left\| \nh U_{{\osc}} \right\|_{\cPHs}.\label{123stella}
\end{align}
An application of Poincar\'e inequality and and \eqref{GN type ineq2} allow us to deduce that
\begin{align*}
\left\| V_{\QG} \right\|_{H^{1/2, s}} \lesssim & \left\| \nh V_{\QG} \right\|_{H^{1/2, s}},\\
\lesssim & \
\left\| \nh V_{\QG} \right\|_{H^{0, s}}^{1/2}
\left\| \nh^2 V_{\QG} \right\|_{H^{0, s}}^{1/2}.
\end{align*}
An application of Lemma \ref{higher regularity VQG} leads to
\begin{equation*}
\left\| \nh V_{\QG} \right\|_{H^{0, s}}^{1/2}
\left\| \nh^2 V_{\QG} \right\|_{H^{0, s}}^{1/2} \lesssim
\left\| \Omega \right\|_{H^{0, s}}^{1/2}
\left\| \nh \Omega \right\|_{H^{0, s}}^{1/2},
\end{equation*}
whence with use of \eqref{GN type ineq2} we deduce the bound
\begin{align}\label{eq:bound_Bh}
B_h\left( q \right) \lesssim &
\ 2^{-2qs} b_q\left( t \right) \left\| \Omega \right\|_{H^{0, s}}^{1/2}
\left\| \nh \Omega \right\|_{H^{0, s}}^{1/2} \left\| U_{\osc} \right\|_{H^{0, s}}^{1/2} \left\| \nh U_{{\osc}} \right\|_{\cPHs}^{3/2}.
\end{align}
The term $ B_v $ can instead be written as 
\begin{equation*}
B_v\left( q \right) = \left| \left( \left. {\cPtv} \left( \diveh u^h_{\osc} \  V_{\QG} \right) \right| {\cPtv} U_{\osc} \right) \right| + \left| \left( \left.  {\cPtv} \left( u^3_{\osc} \ \partial_3 V_{\QG} \right) \right| {\cPtv} U_{\osc} \right) \right| = B_v^1\left( q \right) + B_v^2\left( q \right).
\end{equation*}
For the term $ B^1_v\left( q \right) $, applying \eqref{eq:reg_dyadic_blocks} and Lemma \ref{lem:prod_Sob_anisotropic}
\begin{align*}
B^1_v\left( q \right) \lesssim & \ 2^{-2qs} b_q \left( \ t \right) \left( \Omega, U_{\osc} \right) \left\| \diveh u^h_{\osc} \  V_{\QG} \right\|_{H^{-1/2, s}} \left\| U_{\osc} \right\|_{H^{1/2, s}}\\
\lesssim &  \ 2^{-2qs} b_q \left(  t \right) \left\| V_{\QG} \right\|_{H^{1/2, s}} \left\| U_{\osc} \right\|_{H^{1/2, s}} \left\| \nh U_{{\osc}} \right\|_{\cPHs},
\end{align*}
which is the same estimate as \eqref{123stella} 
and whence we can deduce the same bound as for $ B_h\left( q \right) $. i.e. \eqref{eq:bound_Bh}.\\
The term $ B^2_v\left( q \right) $ is indeed less regular due to the presence of the vertical derivative. Similarly as before we can apply \eqref{eq:reg_dyadic_blocks} and Lemma \ref{lem:prod_Sob_anisotropic} to deduce
\begin{align*}
B^2_v\left( q \right) \lesssim 2^{-2qs} b_q \left(  t \right) \left\| \partial_3 V_{\QG} \right\|_\cPHs \left\| U_{\osc} \right\|_{H^{1/2, s}}^2.
\end{align*}
Poincar\'e inequality and Lemma \ref{higher regularity VQG} imply
\begin{equation*}
\left\| \partial_3 V_{\QG} \right\|_\cPHs \lesssim \left\| \partial_3 \nh V_{\QG} \right\|_\cPHs \lesssim \left\| \nh \Omega \right\|_\cPHs,
\end{equation*}
while using \eqref{GN type ineq2} we can  conclude with the following bound
\begin{align}\label{est:B2v}
B^2_v\left( q \right) \lesssim 2^{-2qs} b_q \left(  t \right) \left\| \nh \Omega \right\|_\cPHs \left\| U_{\osc} \right\|_{H^{0, s}}\left\| \nh U_{\osc} \right\|_{H^{0, s}}.
\end{align}
Whence \eqref{eq:bound_Bh} and \eqref{est:B2v} prove \eqref{prima stima termine bilineare parte oscillante}.

\textit{Proof of \eqref{terza stima termine bilineare parte oscillante}}:
Lastly we consider the term
\begin{align*}
\left(\left. {\cPtv} \mathcal{Q}\left( U_{\osc}, U_{\osc} \right) \right| {\cPtv} U_{\osc}  \right)= & \left(\left. {\cPtv} \mathcal{Q}^h\left(  U_{\osc}, U_{\osc} \right) \right| {\cPtv} U_{\osc}  \right)+\left(\left. {\cPtv} \mathcal{Q}^3\left(  U_{\osc} , U_{\osc} \right) \right| {\cPtv} U_{\osc}  \right)\\
= & C^h\left( q \right)+ C^v\left( q \right),
\end{align*}
where $ \mathcal{Q}^h $ and $ \mathcal{Q}^3 $ are respectively defined as
\begin{align}
\label{Qh}
\mathcal{Q}^h \left( U_{\osc}, U_{\osc} \right) = & \lim_{\varepsilon \to 0} \Lminus \left[
\left( \Lplus U_{\osc}^\varepsilon \right)^h \cdot \nh \Lplus U_{\osc}
\right],
\\
\label{Q3}
\mathcal{Q}^3 \left( U_{\osc}, U_{\osc} \right) = & \lim_{\varepsilon \to 0} \Lminus \left[
\left( \Lplus U_{\osc}^\varepsilon \right)^3 \partial_3 \Lplus U_{\osc}
\right]
\end{align}
By aid of Bony decomposition as in \eqref{Paicu Bony deco} we can say that
\begin{multline*}
C^h\left( q \right) = \sumf 
\left(\left. {\cPtv} \mathcal{Q}^h\left( \cPSvq   U_{\osc}, \cPTv U_{\osc} \right)  \right| {\cPtv} U_{\osc} \right)
\\
 +
 \sumi \left(\left. {\cPtv} \mathcal{Q}^h\left( \cPTv   U_{\osc}, S_{q'+2}^v U_{\osc} \right)  \right| {\cPtv} U_{\osc} \right)
=C^h_1\left( q \right)+C^h_2\left( q \right).
\end{multline*}
By use of Lemma \ref{product rule}
\begin{align*}
C^h_1\left( q \right) \lesssim & \sumf \left\| \cPSvq U_{\osc} \right\|_{\Hud} \left\| \cPTv \nh U_{\osc} \right\|_\cPLtwo \left\| {\cPtv} U_{\osc} \right\|_{\Hud},
\end{align*}
moreover since $ U_{\osc} $ is a vector field with zero horizontal average we can apply  \eqref{GN type ineq}
\begin{align*}
\left\| \cPSvq U_{\osc} \right\|_{\Hud} \lesssim & \left\|  U_{\osc} \right\|_{\cPLtwo}^{1/2} \left\| \nh  U_{\osc} \right\|_{\cPLtwo}^{1/2} , 
\\
\left\| {\cPtv} U_{\osc} \right\|_{\Hud} \lesssim & \left\| {\cPtv} U_{\osc} \right\|_{\cPLtwo}^{1/2} \left\| \nh {\cPtv} U_{\osc} \right\|_{\cPLtwo}^{1/2}, 
\end{align*}
whence thanks to \eqref{eq:reg_dyadic_blocks} and the fact that we are summing on a finite set of $ q' $
\begin{equation}\label{est:C1h}
C^h_1 \left( q \right) \lesssim b_q \left( t \right) 2^{-2qs} \left\|  U_{\osc} \right\|_{\cPLtwo}^{1/2} \left\| \nh  U_{\osc} \right\|_{\cPLtwo}^{1/2} \left\|  U_{\osc} \right\|_{\cPHs}^{1/2} \left\| \nh  U_{\osc} \right\|_{\cPHs}^{3/2}.
\end{equation}
Similar computations give us the result for $ C^h_2 $, here we sketch the procedure. Respectively using \eqref{product rule equation}, \eqref{eq:reg_dyadic_blocks} and summing on the summation set
\begin{align}
C^h_2\left( q \right) = & \sumi \left(\left. {\cPtv} \mathcal{Q}^h\left( \cPTv U_{\osc}, S_{q'+2}^v U_{\osc} \right)  \right| {\cPtv} U_{\osc} \right)\nonumber\\
\lesssim & b_q \left(  t \right) 2^{-2qs} \left\|  \nh U_{\osc} \right\|_\cPLtwo \left\| U_{\osc} \right\|_\cPHs \left\| \nh U_{\osc} \right\|_\cPHs.\label{est:C2h}
\end{align}
On the term $ C^v $ we apply instead Bony decomposition as in \eqref{asymmetric Bony decomp Omega} obtaining
\begin{multline*}
C^v\left( q \right) = \left(\left. \mathcal{Q}^3\left( S^v_{q-1} U_{\osc} , {\cPtv} U_{\osc} \right) \right|  {\cPtv} U_{\osc} \right)
\\
+
\sum _{|q-q'|\leqslant 4} \left( \left.   \mathcal{Q}^3 \left( S^v_{q-1} U_{\osc}-S^v_{q'-1} U_{\osc}, {\cPtv} \cPTv U_{\osc} \right) \right| {\cPtv} U_{\osc} \right)\\
+\sum _{|q-q'|\leqslant 4} \lim_{\varepsilon\to 0 } \left( \left. \left[ {\cPtv}, S^v_{q'-1} \left( \Lplus U_{\osc} \right)^3 \right] \cPTv \partial_3 \Lplus U_{\osc} \right| {\cPtv} \Lplus U_{\osc} \right)\\
+ \sum_{q'>q-4} \left( \left. {\cPtv} \mathcal{Q}^3 \left( \cPTv  U_{\osc} , S^v_{q'+2} U_{\osc} \right)\right| {\cPtv} U_{\osc}\right)=\sum_{k=1}^4 C^v_k\left( q \right),
\end{multline*}
where $ \mathcal{Q}^3 $ is defined in \eqref{Q3}. Let us consider the term $ C^v_1\left( q \right) $ first. Integration by parts and the fact that we are considering divergence-free vector fields gives us
\begin{align*}
C^v_1\left( q \right)= & \lim_{\varepsilon\to 0} \int_{\T^3} S^v_{q-1} \left( \Lplus U_{\osc} \right)^3 \partial_3 \Lplus {\cPtv} U_{\osc}\Lplus {\cPtv} U_{\osc} \dx\\
= & -\frac{1}{2} \lim_{\varepsilon\to 0}  \int_{\T^3} S^v_{q-1} \diveh \left( \Lplus U_{\osc} \right)^h \left| \Lplus {\cPtv} U_{\osc} \right|^2\dx.
\end{align*}
Moreover using the fact that $ \Lplus $ is an isometry on Sobolev spaces, and \eqref{product rule equation} we deduce
\begin{align}
C^v_1\left( q \right)= & \ \lim_{\varepsilon\to 0} \int_{\T^3} S^v_{q-1}\diveh \left( \Lplus U_{\osc} \right)^h \left( \Lplus {\cPtv} U_{\osc} \right)^2\dx \nonumber \\
 \lesssim & b_q ( t) 2^{-2qs} \left\| \nh U_{\osc} \right\|_\cPLtwo \left\| U_{\osc} \right\|_\cPHs \left\| \nh U_{\osc} \right\|_\cPHs.\label{est:C1v}
\end{align}
Let us consider the term $ C^v_2 $ which is defined as
\begin{align*}
C_2^v\left( q \right) = &\sum _{|q-q'|\leqslant 4} \left( \left.   \mathcal{Q}^3 \left( S^v_{q-1} U_{\osc}-S^v_{q'-1} U_{\osc}, {\cPtv} \cPTv U_{\osc} \right) \right| {\cPtv} U_{\osc} \right)\\
= & \sumf \sum_{\substack{\left( k,m,n \right)\in \mathcal{K}^\star\\ a,b,c,d=\pm}} \left( \widehat{ S^v_{q-1}U}^{a,3}\left( k \right) -\widehat{ S^v_{q'-1}U}^{a,3}\left( k \right) \right)  m_3 \widehat{{\cPtv} \cPTv U}^{b,c}\left( m, n \right) \widehat{{\cPtv} U}^d \left( n \right),
\end{align*}
where $ \hat{U}^{b,c}\left( m, n \right)= \left(\left. \hat{U}^b \left( m \right) \right| e^c\left( n \right)  \right) e^c\left( n \right) $. Since the eigenvectors $ e^c $ can always be considered normalized to norm one we deduce $ \left| \hat{U}^{b,c}\left( m, n \right) \right|\lesssim \left| U^b \left( m \right) \right| $. At this point we can use 
  Lemma \ref{product rule} to obtain the bound
\begin{align*}
C^v_2 \left( q \right)\lesssim& \sumf \sum_{a=\pm} \left\|  S^v_{q-1}{U}^{a,3}-S^v_{q'-1}{U}^{a,3}   \right\|_{\cPLtwo} \left\| {\cPtv} \cPTv \partial_3 U_{\osc}  \right\|_\Hud \left\| {\cPtv} U_{\osc} \right\|_\Hud.
\end{align*}
We remark that the term $  {U}^{a}  $ is in fact divergence-free.\\
Thanks to Lemma \ref{bernstein inequality}
\begin{align*}
\left\|  S^v_{q-1}{U}^{a,3}-S^v_{q'-1}{U}^{a,3}   \right\|_{\cPLtwo} \lesssim & 
2^{-q} \left\| \left( S^v_{q-1}-S^v_{q'-1} \right)\partial_3{U}^{a,3}  \right\|_{\cPLtwo}\\
\lesssim & 
2^{-q} \left\| \left( S^v_{q-1}-S^v_{q'-1} \right) \nh  U_{\osc} \right\|_{\cPLtwo}\\
\left\| {\cPtv} \cPTv \partial_3 U_{\osc}  \right\|_\Hud \lesssim & 2^q \left\| {\cPtv} \cPTv  U_{\osc}  \right\|_\Hud
\end{align*}
Hence using first \eqref{eq:reg_dyadic_blocks} and then \eqref{GN type ineq2}
\begin{align}
C^v_2\left( q \right) \lesssim b_q \left(  t \right) 2^{-2qs} \left\| \nh U_{\osc} \right\|_\cPLtwo \left\| U_{\osc} \right\|_{\cPHs} \left\| \nh U_{\osc} \right\|_{\cPHs}.
\end{align}
The term $ C^v_3\left( q \right)$ will be handled in a different way. First of all, writing $f_\varepsilon = \Lplus f$ and considering that commutators can be expressed as convolutions (as it has been expressed in detail in the Section \ref{sec:dy_convolution}, see equation \eqref{eq:dyadic_as_convolution}) we can write $C^v_3\left( q \right)$ as
\begin{multline*}
C^v_3\left( q \right)= \lim_{\varepsilon \to 0} \sum _{|q-q'|\leqslant 4} \intT \int_{\mathbb{T}^1_v \times \left[0,1\right]} \tilde{h}\left( 2^q y_3 \right) \left( \cPSvq \partial_3 U^3_{{\osc},\varepsilon}\right) \left( x_h, x_3 +\tau \left(x_3-y_3\right)\right)\\
\times \partial_3 \cPTv U_{{\osc}, \varepsilon}\left( x_h, x_3-y_3\right)  {\cPtv} U_{{\osc}, \varepsilon}(x) \d y_3 \d \tau \d x_h \d x_3,
\end{multline*}
with $\tilde{h}(z)= zh(z)$ and $h= \mathcal{F}^{-1}\varphi$. Taking the limit as $\varepsilon\to 0$, using the divergence free-property  we obtain the following bound
\begin{multline*}
\left| C^v_3\left( q \right) \right|\leqslant \sum _{|q-q'|\leqslant 4}  \int_{\mathbb{T}^1_v \times \left[0,1\right]}
\sum_{(k,n)\in\mathcal{K}^\star}
 \tilde{h}\left( 2^q y_3 \right) 
 \left| \mathcal{F}\left(\left( \cPSvq \nh U_{{\osc}}\right) \left( x_h, x_3 +\tau \left(x_3-y_3\right)\right)\right)(k) \right|\\
\times \left| \mathcal{F}\left(\partial_3 \cPTv U_{{\osc}}\left( x_h, x_3-y_3\right)\right)(n-k)  \widehat{{\cPtv} U_{{\osc}}}(n) \right| \d y_3 \d \tau ,
\end{multline*}
applying Lemma \ref{product rule} 
\begin{multline*}
C^v_3\left( q \right) \lesssim  \sum _{|q-q'|\leqslant 4}  \int_{\mathbb{T}^1_v \times \left[0,1\right]}
 \tilde{h}\left( 2^q y_3 \right) 
\left\| \cPSvq \nh U_{{\osc}} \left( x_h, x_3 +\tau \left(x_3-y_3\right)\right)\right\|_\cPLtwo\\
\times \left\|\partial_3 \cPTv U_{{\osc}}\left( x_h, x_3-y_3\right)\right\|_\Hud  \left\|{\cPtv} U_{{\osc}}\right\|_\Hud \d y_3 \d \tau ,
\end{multline*}
by standard calculations, localization of the term $ \partial_3 \cPTv U_{{\osc}} $ and  \eqref{GN type ineq} we obtain
\begin{align}
C^v_3\left( q \right) \lesssim  b_q( t) 2^{-2qs}  \left\|\nh U_{\osc} \right\|_\cPLtwo \left\| U_{\osc}\right\|_\cPHs \left\| \nh U_{\osc}\right\|_\cPHs.
\end{align} 
 Lastly, for the reminder term  $ C^v_4\left( q \right) $, if we apply Lemma \ref{product rule} and Lemma \ref{bernstein inequality} as for the term $C^v_2\left( q \right)$ we get
 \begin{align*}
 C^v_4\left( q \right)
  \lesssim \sumi \left\|  S^v_{q'+2} U_{\osc} \right\|_\Hud \left\| \triangle^v_{q'} \nh U_{\osc} \right\|_\cPLtwo \left\| {\cPtv} U_{\osc} \right\|_\Hud,
 \end{align*}
 hence by localization and the interpolation \eqref{GN type ineq2} we obtain
 \begin{align}\label{est:C4v}
C^v_4\left( q \right) \lesssim  b_q ( t) 2^{-2qs} \left\|  U_{\osc} \right\|_{\cPLtwo}^{1/2} \left\| \nh  U_{\osc} \right\|_{\cPLtwo}^{1/2} \left\|  U_{\osc} \right\|_{\cPHs}^{1/2} \left\| \nh  U_{\osc} \right\|_{\cPHs}^{3/2}.
\end{align}
The estimates \eqref{est:C1h}--\eqref{est:C4v} prove hence \eqref{terza stima termine bilineare parte oscillante}.
\end{proof}
\subsection{The bilinear form $\mathcal{Q}$.}\label{bil form limit Section}
In this section we state some particular property of the quadratic limit form defined in \eqref{limit quadratic}. 
In particular we state a product rule which can be applied thanks to the particular structure of the resonance set $ \mathcal{K}^\star = \displaystyle\bigcup_{n\in\mathbb{Z}^3} \mathcal{K}^\star_n $, 
which is  a crucial feature in the energy estimates for the limit system.\\
The following property has been remarked at first by A. Babin et al. in \cite{BMNresonantdomains}, but was first explicitly proved by M. Paicu in \cite{paicu_rotating_fluids}. The proof is based on the fact that, fixed $ \left( k_h, n \right) $, the fiber $ \mathcal{J}\left(  k_h, n \right)= \left\{ k_3 : \left( k,n \right)\in \mathcal{K}^\star \right\} $ is of finite cardinality.
\begin{lemma}\label{product rule}
Let $a,b \in H^{1/2,0}\left(\mathbb{T}^3\right), c\in \cPLtwo$ vector fields of zero horizontal average on $\mathbb{T}^2_h$.  Then there exists a constant $C$ which depends only of $a_1/a_2$ such that
\begin{equation}\label{product rule equation}
\left| \sum_{(k,n)\in \mathcal{K}^\star } \hat{a}\left( k \right) \hat{b}\left( {n-k} \right) \hat{c}\left( n \right) \right| \leqslant \frac{C}{a_3} \left\| a\right\|_{H^{1/2,0}\left(\mathbb{T}^3\right)} \left\|b \right\|_{H^{1/2,0}\left(\mathbb{T}^3\right)} \left\|c\right\|_\cPLtwo
\end{equation}
\end{lemma}
The following proof can be found  \cite[Lemma 6.6, p. 150]{monographrotating} or \cite[Lemma 6.4, p. 222]{paicu_rotating_fluids}.

\begin{proof}
We shall give the proof on the torus $\left[0,2\pi\right)^3$, whence
\begin{align}
I_{\mathcal{K}^\star}= \left| \sum_{(k,n)\in \mathcal{K}^\star } \hat{a}_k \hat{b}_{n-k} \hat{c}_n \right|\leqslant & \sum_{\left( k_h,n\right)\in \mathbb{Z}^2\times \mathbb{Z}^3} \sum_{\left\lbrace k_3:(k,n)\in \mathcal{K}^\star\right\rbrace } \left| \hat{a}_k \hat{b}_{n-k} \hat{c}_n \right|\nonumber\\
\leqslant & \sum_{\left( k_h,n\right)\in \mathbb{Z}^2\times \mathbb{Z}^3} \left| \hat{c}_n\right| \sum_{\left\lbrace k_3:(k,n)\in \mathcal{K}^\star\right\rbrace } 
\left|\hat{a}_k\right|\left| \hat{b}_{n-k}\right|,\label{res ineq 1}
\end{align}
by Cauchy-Schwarz inequality
$$
\sum_{\left\lbrace k_3:(k,n)\in \mathcal{K}^\star\right\rbrace } 
\left|\hat{a}_k\right|\left| \hat{b}_{n-k}\right|\leqslant
\left( \sum_{\left\lbrace k_3:(k,n)\in \mathcal{K}^\star\right\rbrace } 
\left|\hat{a}_k\right|^2\left| \hat{b}_{n-k}\right|^2\right)^{1/2} 
\left( \sum_{\left\lbrace k_3:(k,n)\in \mathcal{K}^\star\right\rbrace }  1\right)^{1/2},
$$
now, fixing $\left( k_h,n\right)\in \mathbb{Z}^2\times \mathbb{Z}^3$ there exists only a finite number (8) or resonant modes $k_3$, i.e. $\# \left(\left\lbrace k_3:(k,n)\in \mathcal{K}^\star\right\rbrace\right) \leqslant 8$. Let us briefly explain why this is true. We write explicitly the resonant condition $ \omega^{+,+,+}_{k, n-k,n}=0 $ (the same procedure holds for the generic case $ \omega^{a,b,c}_{k, n-k,n}=0, a,b,c\neq 0 $), this reads as 
$$
\left(\frac{\left|F k_3\right|^2+\left|k_h\right|^2}{\left|k_3\right|^2+\left|k_h\right|^2}\right)^{1/2}
+\left(\frac{\left( F\left|n_3-k_3\right| \right)^2+\left| n_h-k_h\right|^2}{\left| n_3-k_3\right|^2+\left|n_h-k_h\right|^2}\right)^{1/2}
= \left(\frac{\left|F n_3\right|^2+\left(n_h\right|^2}{\left|n_3\right|^2+\left| n_h\right|^2}\right)^{1/2}.
$$
Taking squares several times on both sides of the above equation give us an expression which is free of square roots. Moreover putting everything to common factor and recalling that $ n,k_h $ are fixed we transformed the above equation in the form $ R\left( k_3 \right)=0 $, $ R\in \R \left[x\right] $, hence thanks to fundamental theorem of algebra it has a finite number of roots.

 From this we deduce
$$
\sum_{\left\lbrace k_3:(k,n)\in \mathcal{K}^\star\right\rbrace } 
\left|\hat{a}_k\right|\left| \hat{b}_{n-k}\right|\leqslant 
\sqrt{8}\left( \sum_{\left\lbrace k_3:(k,n)\in \mathcal{K}^\star\right\rbrace } 
\left|\hat{a}_k\right|^2\left| \hat{b}_{n-k}\right|^2\right)^{1/2},
$$
which considered into inequality \eqref{res ineq 1} gives
$$
I_{\mathcal{K}^\star} \leqslant \sqrt{8}  \sum_{k_h,n_h} \sum_{n_3 } \left| \hat{c}_n\right| \left( \sum_{k_3 } 
\left|\hat{a}_k\right|^2\left| \hat{b}_{n-k}\right|^2\right)^{1/2}.
$$
Moreover
$$
\sum_{n_3 } \left| \hat{c}_n\right| \left( \sum_{k_3 } 
\left|\hat{a}_k\right|^2\left| \hat{b}_{n-k}\right|^2\right)^{1/2}
\leqslant
\left( \sum_{n_3 } \left| \hat{c}_n\right|^2 \right)^{1/2} \left(\sum_{n_3, k_3 } 
\left|\hat{a}_k\right|^2\left| \hat{b}_{n-k}\right|^2\right)^{1/2},
$$
and hence
\begin{equation}\label{res ineq 2}
I_{\mathcal{K}^\star} \leqslant \sqrt{8} \sum_{\left( k_h,n\right)\in \mathbb{Z}^2\times \mathbb{Z}^3} 
\left( \sum_{n_3 } \left| \hat{c}_n\right|^2 \right)^{1/2} \left( \sum_{p_3}\left| \hat{b}_{n_h-k_h,p_3}\right|^2\right)^{1/2} \left( \sum_{k_3} \left| \hat{a}_k\right|^2\right)^{1/2}.
\end{equation}
Let us denote at this point
\begin{align*}
\tilde{a}_{n_h}= &  \left( \sum_{n_3} \left| \hat{a}_n\right|^2\right)^{1/2}, & \tilde{b}_{n_h}= &  \left( \sum_{n_3} \left| \hat{b}_n\right|^2\right)^{1/2}, & \tilde{c}_{n_h}= &  \left( \sum_{n_3} \left| \hat{c}_n\right|^2\right)^{1/2},
\end{align*}
and the following distributions
\begin{align*}
\tilde{a}\left( x_h\right)= & \mathcal{F}_h^{-1}\left(\tilde{a}_{n_h} \right) &
\tilde{b}\left( x_h\right)= & \mathcal{F}_h^{-1}\left(\tilde{b}_{n_h} \right) &
\tilde{c}\left( x_h\right)= & \mathcal{F}_h^{-1}\left(\tilde{c}_{n_h} \right). &
\end{align*}
Whence the inequality \eqref{res ineq 2} can be read, applying Plancherel theorem and the product rules for Sobolev spaces, as
\begin{align*}
I_{\mathcal{K}^\star}\leqslant & \left( \left. \tilde{a}\tilde{b}\right| \tilde{c}\right)_{L^2\left(\mathbb{T}^2_h\right)}\\
\leqslant & \left\| \tilde{a}\tilde{b} \right\|_{L^2\left(\mathbb{T}^2_h\right)}\left\|\tilde{c}\right\|_{L^2\left(\mathbb{T}^2_h\right)}\\
\leqslant & \left\| \tilde{a}\right\|_{H^{1/2}\left(\mathbb{T}^2_h\right)}\left\|\tilde{b} \right\|_{H^{1/2}\left(\mathbb{T}^2_h\right)}\left\|\tilde{c}\right\|_{L^2\left(\mathbb{T}^2_h\right)}\\
=& \left\| a\right\|_{H^{1/2,0}\left(\mathbb{T}^3\right)} \left\|b \right\|_{H^{1/2,0}\left(\mathbb{T}^3\right)} \left\|c\right\|_\cPLtwo.
\end{align*}

To lift this argument to a generic torus $\prod_{i=1}^3 \left[ 0, 2\pi a_i\right)$ it suffice to use the transform 
$$
\tilde{v}\left( x_1, x_2, x_3\right) = v\left( a_1x_1, a_2 x_2, a_3 x_3\right),
$$
and the identity
$$
\left\| \tilde{v}\right\|_{L^2\left( \left[ 0,2\pi\right)^3\right)}= \left( a_1a_2a_3\right)^{-1/2} \left\| v\right\|_{L^2\left( \prod_{i=1}^3\left[ 0, 2\pi a_i\right)\right)}.
$$
\end{proof}

\begin{footnotesize}
\providecommand{\bysame}{\leavevmode\hbox to3em{\hrulefill}\thinspace}
\providecommand{\MR}{\relax\ifhmode\unskip\space\fi MR }
\providecommand{\MRhref}[2]{%
  \href{http://www.ams.org/mathscinet-getitem?mr=#1}{#2}
}
\providecommand{\href}[2]{#2}

\end{footnotesize}

\end{document}